\documentclass[a4paper,11 pt]{amsart}

 \usepackage[sc]{mathpazo}
\usepackage{eulervm}
\usepackage{hyperref}

\usepackage[utf8]{inputenc}
\usepackage[T1]{fontenc}

\usepackage[english]{}
\usepackage{mathtools}
\usepackage{braket}
\usepackage{amsmath, amssymb, stmaryrd, mathabx}
\usepackage{enumerate}
\usepackage{tikz-cd}
\usepackage{tikz,float, hyperref}
\usepackage{titlesec}

\titleformat{\section}
{\normalfont\normalsize\bfseries}{\thesection.}{1em}{}
\titleformat{\subsection}
{\normalfont\normalsize\itshape}{\thesubsection.}{1em}{}
\titleformat{\subsubsection}
{\normalfont\normalsize\itshape}{\thesubsubsection.}{1em}{}

\newcommand{\scp}[2]{\langle #1, #2\rangle}

\newcommand{\bta}[1]{\beta^{[#1]}}

\newcommand{\Z}{\mathbb{Z}}
\newcommand{\R}{\mathbb{R}}
\newcommand{\C}{\mathbb{C}}

\newcommand{\res}{\operatornamewithlimits{Res}}

\newcommand{\OO}{\mathcal{O}}

\newcommand{\Lvec}{{\vec L}}
\newcommand{\coneg}{\mathrm{Cone}_G}
\renewcommand{\L}{\mathcal{L}}

\newcommand{\kk}{\widehat{k}}
\newcommand{\lala}{\widehat{\lambda}}
\newcommand{\fw}{\mathcal{F}_\Phi}
\newcommand{\p}{\prime}
\newcommand{\pp}{\prime\prime}
\newcommand{\rr}{{\Phi}}
\newcommand{\HH}{\mathcal{H}}

\newcommand{\op}[1]{\mathop{\mathchoice{\mbox{\rm #1}}{\mbox{\rm #1}}
{\mbox{\rm \scriptsize #1}}{\mbox{\rm \tiny #1}}}\nolimits}

\newcommand{\B}{{\op{B}}}
\newcommand{\Bases}{\mathcal{B}}
\newcommand{\bb}{\mathbf{B}}

\newcommand{\flag}{\mathrm{Fl}}
\newcommand{\DD}{\mathcal{D}}

\newcommand{\iber}{\operatornamewithlimits{iBer}}
\newcommand{\ver}{\mathrm{Ver}}
\newcommand{\tree}{\mathrm{Tree}}

\newcommand{\affweyl}{\widetilde{\Sigma}[{k}]}
\newcommand{\Ver}{\mathrm{Ver}}
\newcommand{\lfquot}{\mathrm{LFQuot}}

\newcommand{\alphalink}{\beta_{\mathrm{link}}}
\newcommand{\Univ}{UQ}

\newcommand{\bes}{\begin{eqnarray*}}
	\newcommand{\ees}{\end{eqnarray*}}
\newcommand{\beq}{\begin{eqnarray}}
	\newcommand{\eeq}{\end{eqnarray}}

 \theoremstyle{plain}
 \newtheorem{theorem}{Theorem}[section]
\newtheorem{proposition}[theorem]{Proposition}
 \newtheorem{lemma}[theorem]{Lemma}
 
\newtheorem{corollary}[theorem]{Corollary}
\theoremstyle{definition}
 \newtheorem{definition}[theorem]{Definition}
\newtheorem{example}{Example}
\newtheorem{assumptions}[theorem]{Assumptions}
 \newtheorem{remark}[theorem]{Remark}

\title[The parabolic Verlinde formula]{The parabolic Verlinde formula: iterated 
residues and wall-crossings}
\author{Andras Szenes}
\address{Section de mathématiques, Université de Genève}
\email{Andras.Szenes@unige.ch}
\author{Olga Trapeznikova}
\address{Section de mathématiques, Université de Genève}
\email{Olga.Trapeznikova@unige.ch}
\begin{document}	
\begin{abstract} 
We give a new proof for the parabolic Verlinde formula in all ranks based on a comparison of 
wall-crossings in Geometric Invariant Theory and certain iterated residue 
functionals. On the way, we develop a tautological variant of Hecke 
correspondences, calculate the Hilbert polynomials of the moduli spaces, and 
present a new, transparent, local approach to the $\rho$-shift problem of the 
theory.
\end{abstract}

\maketitle

\begin{section}{Introduction}

	\subsection{The Verlinde formula}
	\label{sec:Verlinde}
	The Verlinde formula  is a strikingly beautiful 
	statement 
	in Enumerative Geometry motivated by quantum physics 
	\cite{Ver}. Our focus in this paper will be the more difficult, parabolic 
	variant, which we briefly describe 
	below.
	
	Let $ C $ be a smooth, complex projective curve of genus $ 
	g\ge1 $, and fix an auxiliary point $ p\in C $. We will call a 
	vector $ c=(c_1>c_2>\dots>c_r)\in\R^r $ satisfying $ 
	\sum 	c_i=0 $ and $c_1-c_r<1 $ \textit{regular} if no nontrivial 
	subset of its coordinates sums to an integer. For such a $ c\in\R^r $,
	there exists a smooth projective moduli space $P_0(c)$ 
	(\cite{Seshadri, MehtaSeshadri, Bhosle}), whose points are in one-to-one correspondence 
	with the equivalence classes of pairs $ 
	(W,F_*) $, where $ W \to C$ is a vector bundle  of rank $ r $ 
	on $ C $ with trivial determinant, $ F_* $ is a full flag 
	of the fiber $ W_p $, and the pair satisfies a certain  
	parabolic stability condition depending on $ c $ (cf. \S\ref{S2.1}).  
	This condition roughly states that 
	for a proper subbundle $ W'\subset W $, the degree $ 
	\deg(W') $ is strictly smaller than the sum of a subset of 
	the coordinates of $ c $ depending on the position of $ 
	W^{\p}_p $ with respect to $ F_* $.
	
	There is a natural way to associate to a positive integer $ k $ and an 
	integer vector $ \lambda \in \Z^r$ satisfying 
	$ \lambda_1+\dots+\lambda_r=0$ 
	a line bundle $ \L({k;\lambda}) $ on $ P_0(c)	$, in
	such a way that if $ c =\lambda/k $, then $ \L({k;\lambda}) $ is 
	ample. The \textit{parabolic Verlinde formula} is the 
	following 	expression for the Euler characteristic of the ample line 
	bundle 	$\L({k;\lambda}) $: assume $ c =\lambda/k $ is regular; then
	\begin{equation}\label{eq:ver0}
		\chi(P_0(c),\L({k;\lambda}))  = N_{r,k}\cdot\sum
		\frac{(-i)^{r \choose 2}\exp(2\pi i\, \lala\cdot x)}{\prod_{i<j} 
			\left( 
			2\sin\pi(x_i-x_j)\right) ^{2g-1}}
	\end{equation}
	where $ N_{r,k}=r(r(k+r)^{r-1})^{g-1} $, 	
	$ \lala=\lambda+ 
	\frac12 (r-1,r-3,	\dots,1-r)$,
	and the sum is taken over the finite set of those points  
	in the interior of the parallelopiped
	$$ \{x=(x_1,x_2,\dots,x_r=0)|\; 
	0<x_i-x_{i+1}<1\text{ for }i=1,\dots,r-1\}
	$$
	which satisfy the conditions
	\begin{itemize}
		\item $ (k+r)x\in\Z^r $
		\item $ x_i-x_j\notin\Z $ for $ 1\le i<j< r $.
	\end{itemize} 
	
	We note that this finite set is a set of lattice points in the interior of  
	$(r-1)! $ identical simplices (cf. the rhombus on Figure 
	\ref{fig:lattice}).	
	
	\begin{figure}[H]
		\centering
		\begin{tikzpicture}[scale=0.5]
			\draw (0,0) -- ({sqrt(12)},6);
			\draw (0,0) -- ({-sqrt(12)},6);
			\draw ({-sqrt(12)},6) -- ({sqrt(12)},6);
			\draw [red] (0,0) -- (0,6);
			\draw [teal,fill] (0,0) circle [radius=0.08];
			\draw [teal,fill] ({sqrt(3)},1) circle [radius=0.08];
			\node [below] at ({sqrt(3)},1) {\tiny (0,1,-1)};
			\draw [teal,fill] ({-sqrt(3)},1) circle [radius=0.08];
			\node [below] at ({-sqrt(3)},1) {\tiny (1,-1,0)};
			\draw [teal,fill] ({sqrt(12)},2) circle [radius=0.08];
			\draw [teal,fill] ({-sqrt(12)},2) circle [radius=0.08];
			\draw [teal,fill] (0,2) circle [radius=0.08];
			\draw [teal,fill] ({sqrt(3)},3) circle [radius=0.08];
			\draw [teal,fill] ({-sqrt(3)},3) circle [radius=0.08];
			\draw [teal,fill] (0,4) circle [radius=0.08];
			\draw [teal,fill] ({sqrt(3)},5) circle [radius=0.08];
			\draw [teal] ({sqrt(3)},5) circle [radius=0.3];
			\draw [teal,fill] ({-sqrt(3)},5) circle [radius=0.08];
			\draw [teal] ({-sqrt(3)},5) circle [radius=0.3];
			\draw [teal,fill] ({sqrt(12)},4) circle [radius=0.08];
			\draw [teal,fill] ({-sqrt(12)},4) circle [radius=0.08];
			\draw [teal,fill] ({sqrt(12)},6) circle [radius=0.08];
			\draw [teal,fill] ({-sqrt(12)},6) circle [radius=0.08];
			\draw [teal,fill] (0,6) circle [radius=0.08];
			\node [below] at (0,0) {\tiny (0,0,0)};
			\node [right] at ({sqrt(12)},6) {\tiny (2,2,-4)};
			\node [left] at ({-sqrt(12)},6) {\tiny (4,-2,-2)};
		\end{tikzpicture}\hspace{1cm}
		\begin{tikzpicture}[scale=0.2]
			\draw (0,0) -- ({sqrt(27)},9);
			\draw (0,0) -- ({-sqrt(27)},9);
			\draw ({-sqrt(27)},9) -- ({sqrt(27)},9);
			\draw ({sqrt(27)},9) -- (0,18);
			\draw ({-sqrt(27)},9) -- (0,18);
			\draw [red] (0,0) -- (0,18);
			\draw [fill] (0,0) circle [radius=0.06];
			\draw [orange,fill] (0,2) circle [radius=0.06];
			\draw [orange,fill] (0,4) circle [radius=0.06];
			\draw [orange,fill] (0,6) circle [radius=0.06];
			\draw [fill] ({sqrt(3)},3) circle [radius=0.06];
			\draw [fill] ({-sqrt(3)},3) circle [radius=0.06];
			\draw [orange,fill] (0,4) circle [radius=0.06];
			\draw [orange,fill] (0,8) circle [radius=0.06];
			\draw [orange,fill] (0,10) circle [radius=0.06];
			\draw [orange,fill] (0,12) circle [radius=0.06];
			\draw [orange,fill] (0,14) circle [radius=0.06];
			\draw [orange,fill] (0,16) circle [radius=0.06];
			\draw [fill] (0,18) circle [radius=0.06];
			\draw [fill] ({sqrt(12)},6) circle [radius=0.06];
			\draw [fill] ({-sqrt(12)},6) circle [radius=0.06];
			\draw [orange,fill] ({sqrt(12)},8) circle [radius=0.06];
			\draw [orange,fill] ({-sqrt(12)},8) circle [radius=0.06];
			\draw [orange,fill] ({sqrt(3)},5) circle [radius=0.06];
			\draw [orange,fill] ({-sqrt(3)},5) circle [radius=0.06];
			\draw [orange,fill] ({sqrt(3)},7) circle [radius=0.06];
			\draw [orange,fill] ({-sqrt(3)},7) circle [radius=0.06];
			\draw [fill] ({sqrt(3)},9) circle [radius=0.06];
			\draw [fill] ({-sqrt(3)},9) circle [radius=0.06];
			\draw [fill] ({sqrt(1/3)},1) circle [radius=0.06];
			\draw [orange,fill] ({sqrt(1/3)},3) circle [radius=0.06];
			\draw [orange,fill] ({sqrt(1/3)},5) circle [radius=0.06];
			\draw [orange,fill] ({sqrt(1/3)},7) circle [radius=0.06];
			\draw [fill] ({sqrt(1/3)},9) circle [radius=0.06];
			\draw [orange,fill] ({sqrt(1/3)},11) circle [radius=0.06];
			\draw [orange,fill] ({sqrt(1/3)},13) circle [radius=0.06];
			\draw [orange,fill] ({sqrt(1/3)},15) circle [radius=0.06];
			\draw [fill] ({sqrt(1/3)},17) circle [radius=0.06];
			\draw [orange,fill] ({-sqrt(1/3)},11) circle [radius=0.06];
			\draw [orange,fill] ({-sqrt(1/3)},13) circle [radius=0.06];
			\draw [orange,fill] ({-sqrt(1/3)},15) circle [radius=0.06];
			\draw [fill] ({-sqrt(1/3)},17) circle [radius=0.06];
			\draw [fill] ({sqrt(4/3)},2) circle [radius=0.06];
			\draw [orange,fill] ({sqrt(4/3)},4) circle [radius=0.06];
			\draw [orange,fill] ({sqrt(4/3)},6) circle [radius=0.06];
			\draw [orange,fill] ({sqrt(4/3)},8) circle [radius=0.06];
			\draw [orange,fill] ({sqrt(4/3)},10) circle [radius=0.06];
			\draw [orange,fill] ({sqrt(4/3)},12) circle [radius=0.06];
			\draw [orange,fill] ({sqrt(4/3)},14) circle [radius=0.06];
			\draw [fill] ({sqrt(4/3)},16) circle [radius=0.06];
			\draw [orange,fill] ({-sqrt(4/3)},10) circle [radius=0.06];
			\draw [orange,fill] ({-sqrt(4/3)},12) circle [radius=0.06];
			\draw [orange,fill] ({-sqrt(4/3)},14) circle [radius=0.06];
			\draw [fill] ({-sqrt(4/3)},16) circle [radius=0.06];
			\draw [fill] ({sqrt(3)},15) circle [radius=0.06];
			\draw [fill] ({-sqrt(3)},15) circle [radius=0.06];
			\draw [orange,fill] ({sqrt(3)},13) circle [radius=0.06];
			\draw [orange,fill] ({-sqrt(3)},13) circle [radius=0.06];
			\draw [orange,fill] ({sqrt(3)},11) circle [radius=0.06];
			\draw [orange,fill] ({-sqrt(3)},11) circle [radius=0.06];
			\draw [orange,fill] ({sqrt(12)},10) circle [radius=0.06];
			\draw [orange,fill] ({-sqrt(12)},10) circle [radius=0.06];
			\draw [fill] ({sqrt(12)},12) circle [radius=0.06];
			\draw [fill] ({-sqrt(12)},12) circle [radius=0.06];
			\draw [orange,fill] ({-sqrt(4/3)},8) circle [radius=0.06];
			\draw [fill] ({-sqrt(1/3)},1) circle [radius=0.06];
			\draw [fill] ({-sqrt(4/3)},2) circle [radius=0.06];
			\draw [orange,fill] ({-sqrt(1/3)},3) circle [radius=0.06];
			\draw [orange,fill] ({-sqrt(1/3)},7) circle [radius=0.06];
			\draw [fill] ({-sqrt(1/3)},9) circle [radius=0.06];
			\draw [orange,fill] ({-sqrt(4/3)},4) circle [radius=0.06];
			\draw [orange,fill] ({-sqrt(1/3)},5) circle [radius=0.06];
			\draw [orange,fill] ({-sqrt(4/3)},6) circle [radius=0.06];
			\draw [fill] ({sqrt(16/3)},4) circle [radius=0.06];
			\draw [orange,fill] ({sqrt(16/3)},6) circle [radius=0.06];
			\draw [orange,fill] (-{sqrt(16/3)},8) circle [radius=0.06];
			\draw [fill] ({sqrt(16/3)},14) circle [radius=0.06];
			\draw [orange,fill] ({sqrt(16/3)},8) circle [radius=0.06];
			\draw [fill] (-{sqrt(16/3)},14) circle [radius=0.06];
			\draw [orange,fill] ({sqrt(16/3)},10) circle [radius=0.06];
			\draw [orange,fill] ({sqrt(16/3)},12) circle [radius=0.06];
			\draw [orange,fill] (-{sqrt(16/3)},10) circle [radius=0.06];
			\draw [orange,fill] (-{sqrt(16/3)},12) circle [radius=0.06];
			\draw [fill] ({-sqrt(16/3)},4) circle [radius=0.06];
			\draw [orange,fill] ({-sqrt(16/3)},6) circle [radius=0.06];
			\draw [fill] ({sqrt(25/3)},5) circle [radius=0.06];
			\draw [fill] ({-sqrt(25/3)},5) circle [radius=0.06];
			\draw [orange,fill] ({sqrt(25/3)},7) circle [radius=0.06];
			\draw [orange,fill] ({-sqrt(25/3)},7) circle [radius=0.06];
			\draw [fill] ({sqrt(25/3)},9) circle [radius=0.06];
			\draw [fill] ({-sqrt(25/3)},9) circle [radius=0.06];
			\draw [orange,fill] ({sqrt(25/3)},11) circle [radius=0.06];
			\draw [orange,fill] ({-sqrt(25/3)},11) circle [radius=0.06];
			\draw [fill] ({sqrt(25/3)},13) circle [radius=0.06];
			\draw [fill] ({-sqrt(25/3)},13) circle [radius=0.06];
			\draw [fill] ({-sqrt(27)},9) circle [radius=0.06];
			\draw [fill] ({sqrt(27)},9) circle [radius=0.06];
			\draw [fill] ({sqrt(49/3)},7) circle [radius=0.06];
			\draw [fill] ({sqrt(49/3)},9) circle [radius=0.06];
			\draw [fill] ({-sqrt(49/3)},7) circle [radius=0.06];
			\draw [fill] ({-sqrt(49/3)},9) circle [radius=0.06];
			\draw [fill] ({sqrt(49/3)},11) circle [radius=0.06];
			\draw [fill] ({-sqrt(49/3)},11) circle [radius=0.06];
			\draw [fill] ({sqrt(64/3)},10) circle [radius=0.06];
			\draw [fill] ({-sqrt(64/3)},10) circle [radius=0.06];
			\draw [fill] ({sqrt(64/3)},8) circle [radius=0.06];
			\draw [fill] ({-sqrt(64/3)},8) circle [radius=0.06];
			\node [below] at (0,0) {\tiny (0,0,0)};
			\node [left] at ({-sqrt(27)},9) {\tiny (0,1,1)};
			\node [right] at ({sqrt(27)},9) {\tiny (1,1,0)};
			\node [above] at (0,18) {\tiny (1,2,1)};
		\end{tikzpicture}
		\setlength{\belowcaptionskip}{-8pt}\caption{The set of $\lambda$s (left), and the finite set from 
		\eqref{eq:ver0} (right) for $ k=6 $, $ r=3 $.}\label{fig:lattice}
	\end{figure}
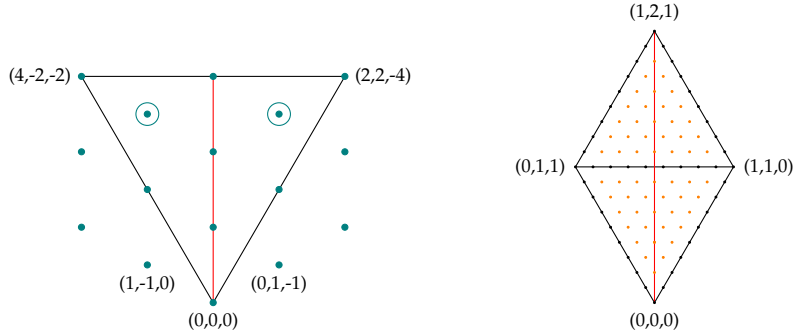

\begin{remark}
Equality \eqref{eq:ver0} remains valid in greater generality, for certain cases 
when $\lambda/k$ is non-regular. 
This slightly more technical 
statement will be given in Theorems \ref{residuethm} and \ref{main}.
\end{remark}
\noindent\textbf{Notation:} We denote the discrete sum in \eqref{eq:ver0} 
depending on 
$ k,\lambda,r$ and $g $ by $ \ver(k,\lambda) $. In what follows, the shift 
$\frac12 (r-1,r-3,	\dots,1-r)$ will be denoted by $ \rho $, and thus we have $ 
\lala =\lambda+\rho $.
\vspace{1mm}

Equality \eqref{eq:ver0}, the \textit{parabolic Verlinde formula} has attracted 
a lot of attention over 	the 
	years, and there is a number of different proofs. There is a 
	generalization of this formula  associated to a simply connected compact 
	Lie group, and the form presented here corresponds to the case of the
	group $ SU(r) $.
	
	In this article, we give a novel proof of this result, 
	which stands out with 	its technical simplicity. We believe the methods 
	and ideas described in the paper will have other  applications in Geometric 
	Invariant Theory and the study of moduli spaces.
	
	Below, we give a quick sketch of the 
	strategy of the proof, treating the example of 
	the case of rank 3 in \S1.2-4. Next, we give a short guide to the contents 
	of the paper in \S1.5.
	
	Our work has 
	close relationship with several earlier approaches, and we 
	describe these links in \S\ref{sub:history}

	\textbf{Acknowledgments}. The authors gratefully acknowledge the help and 
	insights of 
	Michael Thaddeus at several stages of this work, as well as the advice and 
	encouragement of Tam\'as Hausel and Michèle Vergne. We had useful 
	discussions with Eckhard Meinrenken, G\'abor Tardos and Chris Woodward.	
	This research was supported by SNF grant 175799, and the NCCR SwissMAP.

	\subsection{The residue formula} \label{sec:theres}
		The proof is based on 3 ideas. We will follow the 
	arguments below for the case $ r=3 $. We fix thus an integer 
	$ k>1 $ and an integer vector 
$\lambda=(\lambda_1> \lambda_2> \lambda_3)$, 
	such that $ \lambda_1+\lambda_2+\lambda_3=0$ and $\lambda_1-\lambda_r<k$.
		
	We start with the study of the right hand side of 
	\eqref{eq:ver0}, which, for $ r=3 $, may be written in 
	the following somewhat simplified form 
	\[ \mathrm{Ver}(k,\lambda)=N_{3,k}
	\sum_{0<n_2<n_1<k+3}
	\frac{2i\sin2\pi\frac{(\lambda_1+1)n_1+\lambda_2n_2}{k+3}
	}
	{\left(8\sin\pi\frac{n_1-n_2}{k+3}
		\sin\pi\frac{n_2}{k+3}\sin\pi\frac{n_1}{k+3}\right)^{2g-1}},
	\]
	where $ n_1,n_2 $ are integers.
	Using Theorem \ref{residuethm} and Remark \ref{hamform} one can show  that
	\[ \mathrm{Ver}(k,\lambda)=\begin{cases}
		p_+(k;\lambda),\; \lambda_2>0, \\
		p_-(k;\lambda),\; \lambda_2<0,
	\end{cases} \]
	where $ p_+ $ and $ p_- $ are two polynomials, given by  
	the right hand sides of the expressions of Example 
	\ref{ex2poly} on page \pageref{ex2poly}. We note two properties of $ 
	p_\pm(k,\lambda) $:
	\begin{enumerate}
		\item[A.] The wall-crossing difference $ p_--p_+ $ has a 
		relatively 
		simple form (cf. Example \ref{exdiff} with $\lambda_1+\lambda_3$ replaced by $-\lambda_2$):
		\begin{equation*}
			p_-(k;\lambda)-p_+(k;\lambda)	= 
			\res_{y=0}\res_{x=0}  
			\frac{(-3(k+3)^2)^g\cdot e^{(\lambda_1+1)x-\lambda_2y}}
			{(1-e^{x(k+3)})w_\Phi(x,y)^{2g-1}}dx dy,
		\end{equation*}
		where  
		$w_\Phi(x,y)=2\mathrm{sinh}(\frac{x}{2})\cdot 2\mathrm{sinh}
		(\frac{y}{2})\cdot2\mathrm{sinh}(\frac{x+y}{2})$.
		
		\item[B.] An easy calculation via substitution shows that
		for any permutation on 3 elements $ \sigma\in\Sigma_3 $,
		our polynomials have the following symmetries:
		\begin{equation}\label{eq:symplus}
			p_+(k;\sigma\cdot\lambda+\theta_1[k])=
			(-1)^\sigma p_+(k;\lambda+\theta_1[k])
		\end{equation}
		and
		\begin{equation}\label{eq:symminus}
			p_-(k;\sigma\cdot\lambda+\theta_{-1}[k])=
			(-1)^\sigma p_-(k;\lambda+\theta_{-1}[k]),
		\end{equation}
		where
		$$ \theta_1[k]=\frac{k}{3}(1,1,-2)+(0,1,-1) 
		\quad \theta_{-1}[k]=
		\frac{k}{3}(2,-1,-1)+(1,-1,0). 
		$$
	\end{enumerate}
	
	\subsection{Wall-crossings in moduli spaces}
	Now consider the left hand side of \eqref{eq:ver0}. It is 
	easy to check that the set of isomorphism classes of 
	parabolic bundles in $ P_0(c) $ remains unchanged as long as 
	$ c_2 $ does not change sign. Hence, effectively, we have 
	two moduli spaces $ P_0(>) $ and $ P_0(<) $, corresponding to 
	the two chambers separated by the red ($ c_2=0 $) line in 
	Figure \ref{fig:lattice}. Introduce the notation
	\[ q_{+}(k;\lambda)=\chi(P_0(>),\L({k;\lambda})), \quad 
	q_{-}(k;\lambda)=\chi(P_0(<),\L({k;\lambda}))  \]
	for the generalized Hilbert polynomials of these two spaces.
	
	In \S\ref{sec:wcmaster}, we derive a simple formula 
	\eqref{wallcrosseq} for the wall-crossing difference in 
	Geometric Invariant Theory.  The 
	formula has the form of a residue of an equivariant 
	integral, taken with respect to the equivariant parameter. 
	In our case, the space on which we integrate is the space 
	of rank-3 parabolic bundles which split into a direct sum 
	of a rank-2 and a rank-1 bundle. This equivariant integral 
	may be evaluated using induction on the rank (cf. the 
	detailed calculation in Example \ref{exint} on page \pageref{exint}), and 
	the 
	result is 
	\begin{equation}\label{qr3eq}
		q_{-}(k;\lambda)-q_{+}(k;\lambda)= 
		\res_{u=0}\res_{z=0}
		\frac{(-3(k+3)^2)^g\cdot e^{\lambda_1z+\lambda_2u+z}}
		{-w_\Phi(z,-u)^{2g-1}(1-e^{(k+3)z})}dzdu,
	\end{equation}
	where $ u $ plays the role of the equivariant parameter, 
	the generator of $ H^*_{\C^*}(\mathrm{pt}) $.
	This iterated residue coincides with the expression above after changing 
	$(z,u)$ to $(x,-y)$, 
	and thus we have
	\begin{equation}\label{eqwcr}
		p_+-p_- = q_+-q_-.
	\end{equation}

	\subsection
	{Hecke correspondences, Serre duality and the 
		symmetry argument}
	Hecke correspondences between moduli spaces of bundles of 
	different degrees were introduced by Narasimhan and Ramanan in \cite{NRHecke}.
	 In \S\ref{sec:hecke} of our paper, we describe a 
	"tautological" variant of this construction, which 
	identifies the same space with several moduli spaces of parabolic 
	bundles with different degrees and weights. Using this 
	construction we can fiber our two moduli spaces, $ P_0(>) $ 
	and $ P_0(<) $ over the moduli spaces of stable bundles (without parabolic 
	structure) of 
	degrees 1 and $ -1 $:
	\[ \mathrm{Flag}_3\to P_0(<) \to N_{-1},\quad
	N_{1}\leftarrow P_0(>) \leftarrow \mathrm{Flag}_3,\]  
	where the fibers are full flags of  3-dimensional vector 
	spaces.  Serre duality applied to a $\mathrm{Flag}_3  $-bundle implies a 
	$ \Sigma_3 $-antisymmetry of the Euler characteristics of line bundles on 
	this space, and after careful identification of these bundles, we derive 
	the same symmetry properties for the functions $ 
	q_\pm $ as we did for the polynomials $ p_\pm $: $ 
	q_+(k;\lambda) $ satisfies \eqref{eq:symplus}, while $ 
	q_-(k;\lambda) $ satisfies \eqref{eq:symminus}.

	The final argument is elegant: we can rearrange equation \eqref{eqwcr} 
	describing the equality of wall-crossings as 
	\[ p_+-q_+=p_--q_-, \]
	and we introduce the notation $ \Theta(k;\lambda) $ for this polynomial.
	Then $ \Theta $ satisfies both 
	\eqref{eq:symplus} and \eqref{eq:symminus}, and thus it is 
	anti-invariant with respect to an affine Weyl group action 
	in the plane for each fixed $ k $. This implies that 
	$ \Theta(k;\lambda) $ vanishes and this completes the proof.
	
	\subsection{Contents of the paper}
	There are a number of complications which arise when $ r>3 $. We will 
	highlight 
	these in this section, and also give a brief guide to the contents of the 
	paper.
	
	We start with a quick introduction into the theory of parabolic bundles in 
	\S\ref{sec:parabolic}. Here we describe the line bundles we are considering, 
	as 	well as the chamber structure of the space of parabolic weights induced by 
	the stability 	condition. The combinatorics of the iterated residue formulas mentioned 
	in \S\ref{sec:theres} above is considerably more complicated in the higher 
	rank 	case, and is best treated using the notion of \textit{diagonal bases} of hyperplane arrangements
	introduced in \cite{Szimrn}; we review this construction in the special case of the $ A_r $ root arrangement in \S\ref{sec:wcverlinde}. 
	
	Using this notion, in \S\ref{sec:residuemain}, we present a residue 
	formula for the Verlinde sums on the right hand side of \eqref{eq:ver0} 
	obtained in \cite{Szduke} (Theorems \ref{diaginv} and \ref{residuethm}). It 
	turns out that because of a standard $ \rho $-shift type effect in the 
	theory, this residue formula does not have a manifestly polynomial form on our chambers, 
	and thus, we formulate our main result, Theorem \ref{main} in two parts: in 
	part I. we state the equality of the Euler characteristics of line bundles 
	with a modified residue formula, which is manifestly polynomial on our 
	chambers, and 	
	in 	part II. we state the equality of the modified formula with the 
	original residue 
	formula from \cite{Szduke}.  Part II. is proved in \S\ref{S10}, while the 
	the proof of part I. takes up the rest of the paper.
	
	At the end of \S\ref{sec:residuemain}, we present our wallcrossing formula 
	for 
	Verlinde sums in Proposition \ref{wcrnice}, which uses in an essential 
	manner 
	the yoga of diagonal bases (cf. property A. above for the case of $ r=3 $).
	
	The geometric part of our work starts in \S\ref{sec:wcmaster}, where we 
	derive 
	a simple general result, formula \eqref{wallcrosseq}, for wallcrossings in 
	GIT.
	We apply this result to parabolic moduli spaces in \S\ref{sec:wallcrpara}, 
	and, 
	using induction on the rank, obtain Theorem \ref{wallcrossint}, the higher 
	rank 
	version of formula \eqref{qr3eq} above.
	
	It is downhill from here: in \S\ref{sec:hecke} we describe the tautological 
	Hecke correspondences we need in several places in the paper, and 
	in 
	\S\ref{sec:weylsymm} we derive the Weyl-symmetries of the polynomials $ 
	q_{\pm}$, and finish the proof along the lines sketched above.
	
	We are essentially done, but we hit a snag when checking the beginning of 
	our 
	induction on the rank: our argument does not work for $ r=2 $. Roughly, the 
	reason for this is that we need our simplex of parabolic weights to have at least 2 
	regular 
	vertices, and for $ r=2 $, we have only 1. The way out is to consider the 
	moduli space with two punctures and then all the pieces fall in place. This 
	argument is carried out in \S\ref{sec:2points}.
	
	\subsection{Historical remarks}\label{sub:history}
	There is a long list of proofs of the Verlinde 
	formula, and we cannot do justice to all the approaches in this short 
	introduction. We will thus focus on the historical lineage of our paper, 
	and 
	the works that are closest in spirit to what we do (cf. \cite{Sorger} for a 
	more comprehensive overview). 
	
	The proofs of the Verlinde formula fall in two categories: proofs of the 
	fusion 
	rules and proofs that find some interpretation of the "Fourier transformed" 
	discrete sum on the right hand side of \eqref{eq:ver0}; our work belongs to 
	this second group. Another line of division concerns the model which one 
	uses 
	for the moduli spaces: via the Narasimhan-Seshadri correspondence, the 
	moduli 
	spaces of vector bundles may equally be presented as symplectic manifolds 
	of 
	certain types of flat connections on punctured Riemann surfaces, and this 
	opens 
	the way of using the methods of symplectic geometry. While these symplectic 
	approaches lead to results equivalent to the ones coming out of the  
	algebro-geometric setup, the fields of 
	applications of the two approaches seem to be very different. 
	
	The idea of proving the Verlinde formula via wall crossings appeared in the 
	seminal paper of Michael Thaddeus \cite{ThaddeusVer}. He used a geometric 
	approach and managed to prove the Verlinde formula in rank 2 by crossing 
	walls 
	in the moduli of stable pairs.  The \textit{master space 
		construction}, which plays a central role in our paper, first appeared 
		in his 
	works as well \cite{ThaddeusFlip}. In a sense, our paper 
	may be thought of as the completion of his program.
	
	A paper closely related to our work is that of Jeffrey and Kirwan 
	\cite{JeffreyK}, who also use the residue calculus introduced in 
	\cite{Szimrn,Szduke}. This paper approaches the problem from a 
	symplectic/cohomological point of view, and has a somewhat different angle 
	form 
	ours. The use of iterated residues is not quite as consistent in 
	\cite{JeffreyK} as in our work, and the parabolic case was not resolved 
	from 
	this point of view \cite{Jeffrey}. The geometric model used to represent 
	the 
	moduli spaces as quotients is rather complicated.
	
	In a comprehensive paper \cite{BismutLabourie} covering the case of all 
	compact 
	groups, Bismut, Laborie used a differential–geometric approach to find the 
	generating function for the parabolic Verlinde formula. This work was the 
	motivation for the residue formula in \cite{Szduke}, which is also used in 
	the 
	present paper.
	
	In a remarkable series of papers of Alekseev, Meinrenken and Woodward 
	\cite{AlexeevMW}, again, 
	approaching the subject from the symplectic point of view, gave a direct 
	proof of \eqref{eq:ver0}, using reduction in infinite dimensions. A general approach related to twisted K-theory was introduced by Meinrenken in \cite{Meinrenken}.
	We should also mention recent work by Loizides and Meinrenken in
	\cite{Loizides}, which employs the residue techniques of  \cite{Szduke}.

	Finally, we drew motivation from the paper of Teleman and Woodward 
	\cite{TelemanW}, where the 
	Verlinde formula is put in the framework 
	of localization in K-theory of stacks. This very impressive work is 
	probably accessible to a small number of experts only. In the present 
	article, we demonstrate, in particular, that the sophisticated tools 
	employed in \cite{TelemanW}, 
	at least in this instance, may be replaced by a simple combinatorial device.
	
	In summary, the virtues of this article are: 
	\begin{itemize}
		\item A proof of the parabolic Verlinde formula which needs as 
		background only 
		the basics of GIT.
		\item The discrete sum, and the generating function giving the 
		coefficients of 
		the Hilbert polynomial are treated at the same time, and the $ \rho 
		$-shift is 
		dealt with explicitly.
		\item A few technical innovations such that an efficient wall-crossing 
		formula 
		in GIT (Theorem \ref{ThaddTh}) and the tautological Hecke 
		correspondences keep 
		the 
		arguments simple, and the technical difficulties related to infinite 
		dimensional quotients or singularities, in our approach, are absorbed 
		by a combinatorial device: the theory of iterated residues.
	\end{itemize}

\end{section}

\begin{section}{Parabolic bundles}
	\label{sec:parabolic}
	
\subsection{Definitions}\label{S2.1}
Let $C$ be a smooth complex projective curve of genus $g\geq2$, and fix a point $p\in{C}$. 
\begin{itemize}
\item  A \textit{parabolic bundle} on $C$ is a vector bundle $W$ of rank $r$ 
with a full flag $ F_* $ in the fiber over $p$:
$$W_p=F_r\supsetneq...\supsetneq F_1\supsetneq F_0=0$$ and \textit{parabolic 
weights} $c=(c_1,...,c_r)$ assigned to $F_r, F_{r-1},..., F_1$,
 satisfying the conditions
$$c_1 > c_{2}> . . . > c_r \text{ and } c_1-c_r<1.$$

\item The \textit{parabolic degree}\footnote{For technical reasons, we have chosen a sign convention opposite to that in the majority of treatments in the literature.} and the \textit{parabolic slope} of $W$ are 
defined as
$$\textit{pardeg}(W)=\mathrm{deg}(W)-\sum_{i=1}^{r}c_i ; \, \, \, \, \, \textit{parslope}(W)=\frac{\textit{pardeg}(W)}{\mathrm{rank}(W)}.$$ 
\item A \textit{morphism} $f: W\to W'$ of parabolic bundles is a morphism of 
vector bundles satisfying $f_p(F_i)\subset{F'_{j-1}}$ if $c_{r-i+1}< 
c'_{r-j+1}$. In particular, an \textit{endomorphism} of a parabolic bundle $W$ 
is a vector bundle endomorphism  preserving the flag $F_*$.
\item Denote by $\mathrm{ParHom}(W, W')$ the sheaf of parabolic morphisms from $W$ to $W'$. Then there is a short exact sequence of sheaves
\begin{equation}\label{SESparhom}
0\to\mathrm{ParHom}(W,W')\to\mathrm{Hom}(W,W')\to T_p\to{0},
\end{equation}
where $T_p$ is a torsion sheaf supported at $p$. The rank of $T_p$ is the 
number of pairs $(i,j)$, s.t. $c_i<c^\p_j$  (cf. \cite{BodenH}).
\end{itemize}
If $W'\subset{W}$ is a subbundle of $W$, then both $W'$ and the quotient $W/W^\p$ inherit a parabolic structure from $W$ in a natural way (cf. \cite{MehtaSeshadri}, definition 1.7).
\begin{itemize}
\item A parabolic bundle $W$ is \textit{stable of weight $ c $}, if any proper 
subbundle $W'\subset{W}$ satisfies 
$\textit{parslope}(W')<\textit{parslope}(W)$; and $W$ is \textit{semistable of weight $ c $}, if the inequality is not strict.
\end{itemize}
\begin{remark}\label{rem:projweights}
Note that the parabolic stability condition depends on the parabolic weights only up to adding the same constant to all weights $c_i$. 
\end{remark}

\subsection{Construction of the moduli spaces}\label{S2.2}
We start with a quick review of the construction of Mehta and Seshadri 
\cite{MehtaSeshadri} 
of the moduli space of 
stable parabolic bundles. 
It follows from Remark \ref{rem:projweights} that, without loss of generality, we can assume that the parabolic weights of a rank-$ r $  degree-$ d $ bundle belong to the simplex 
$$\Delta_{d}= \left\{(c_1,c_{2}, ..., c_r) \, | \, c_1>{c_{2}}> ... >{c_r}, \, c_1-c_r<1, \, \sum_ic_i=d\right\}.$$
\begin{definition} \label{reg}
We will call a vector $ c=(c_1,\dots,c_r)\in\R^r $ such that $ \sum_ic_i\in\Z $ 
\textit{regular} if for any nontrivial subset $ \Psi\subset\{1,2,\dots,r\} $, 
we have $ \sum_{i\in\Psi}c_i\notin\Z $.
\end{definition}
Now choose an integer $d\gg0$ such that $H^1(W)=0$ and $W$ is generated by global sections for any rank-$r$ degree-$d$ 
semistable parabolic bundle $W$ of parabolic degree $0$. Put $\chi=r(1-g)+d$ and 
consider the
\begin{itemize}
\item 	Groethendieck quot scheme $Quot(\chi,r)$ (\cite{Grothendieck}) parametrizing quotients 
	 $\mathcal{O}^\chi \twoheadrightarrow W$, where $W$ is a coherent sheaf of 
	degree 	$d$ and rank $r$.
	\item This space is endowed with a universal bundle $ \Univ $, and a 
	genericly free 	action of the group $G=PSL(\chi)$, which does not, 
	however, lift to $ \Univ$.
	\item Let $\lfquot\subset Quot(\chi,r)$ be the open subscheme 
	consisting of  locally free quotients $W$, such that the induced map 
	$H^0(\mathcal{O}^\chi)\to H^0(W)$ is an isomorphism.
	\item Denote by $XQ$  the total space of the flag bundle  
	$\mathrm{Flag}({\Univ}_p)$ on $\lfquot\times{p}$. This space is endowed with the 
	flag of vector bundles $ Fl_1\subset\dots\subset Fl_{r-1}\subset 
	Fl_r=\Univ_p $.
	\item Let $k\in\mathbb{Z}$ and $(\lambda_1,...,\lambda_r)\in\mathbb{Z}^r$, 
	such that 
	$\sum_{i=1}^r\lambda_i=kd$, and consider the line bundle 
	$$L(k;\lambda)=\mathrm{det}({\Univ}_p)^{k(1-g)}\otimes 
	det(\pi_*\Univ)^{-k}\otimes
	({Fl}_r/{Fl}_{r-1})^{\lambda_1}\otimes...\otimes({Fl}_1)^{\lambda_r}$$
	on $XQ$, which does carry a $G$-linearization (lift of the $ G $-action 
	from $ XQ $).
	\item Finally, assume  $ c\in\Delta_d $ is regular (cf. Definition 
	\ref{reg} 
	above) and define $\widetilde{P}_d(c)$, 
	the moduli space  
	of stable parabolic weight-$ c $ vector bundles on $ C $ as the GIT 
	quotient ${XQ\sslash^c G}$ of $XQ$ with 
	respect to any linearization $L(k;\lambda)$, such that  $\lambda/k=c$. 
\end{itemize}

\begin{theorem}[\cite{Seshadri}]  Assume that $c\in \Delta_d$ is a regular 
weight vector. Then the moduli space 
$\widetilde{P}_d(c)$ is a smooth projective variety of dimension $ 
r^2(g-1)+\binom r2+1 $, whose points are in one-to-one correspondence with the 
set of isomorphism classes of stable parabolic bundles of weight $ c $ (cf. 
\S\ref{S2.1}). 
\end{theorem}
\begin{remark}
Via the determinant map, the moduli space $\widetilde{P}_d(c)$ fibers over the 
Jacobian of degree-$ d $ 
line bundles with isomorphic fibers, and in this paper, we will focus on the 
moduli space
\[ P_d(c) =  \{W\in\widetilde{P}_d(c)|\; \det W \simeq\mathcal{O}(d{p}) \},\]
which is smooth, projective and has dimension $ 
(r^2-1)(g-1)+\binom r2 $.
\end{remark}
\begin{remark}\label{tensoriso}
	Note that tensoring with  the line bundle $\mathcal{O}(mp)$ induces an 
	isomorphism: $\otimes\mathcal{O}(mp): P_{d}(c)\to P_{d+rm}(c)$, so the 
	moduli spaces $P_d(c)$, essentially, depend only on $d$ modulo $r$. 
\end{remark}

\subsection{The Picard group of  $P_d(c)$}\label{S2.3}

For a regular  $c\in \Delta_d$,
there exist universal bundles $U$ over $P_d(c)\times{C}$ endowed with a flag  
$ \mathcal{F}_{1}\subset\dots\subset\mathcal{F}_{r-1}\subset\mathcal{F}_{r}=U_p 
$, and satisfying the obvious tautological properties.
In general, such universal bundles $U$, and hence the flag line bundles 
$\mathcal{F}_{i+1}/\mathcal{F}_i$ are unique only up to tensoring by the 
pull-back of a line bundle from $P_d(c)$. Nevertheless, we have the following 
statement, which is easy to verify.
\begin{lemma}\label{parpic}
For $k\in\mathbb{Z}$ and $\lambda=(\lambda_1,...,\lambda_r)\in\mathbb{Z}^r$, such that $\sum_{i=1}^r\lambda_i=kd$, the line bundle 
\begin{multline*}
\mathcal{L}_d(k; {\lambda})=\mathrm{det}(U_p)^{k(1-g)}\otimes\mathrm{det}(\pi_*U)^{-k}  \otimes(\mathcal{F}_r/\mathcal{F}_{r-1})^{\lambda_1}\otimes...\otimes(\mathcal{F}_1)^{\lambda_r}
\end{multline*} 
on  $P_d(c)$ is independent of the choice of the 
universal bundle $ U $. 
\end{lemma}
\begin{remark}
	The line bundle $L(k;\lambda)$ defined in \S\ref{S2.2} descends to the line 
	bundle $\L_d(k;\lambda)$ on the GIT quotient $P_d(c)$.
\end{remark}
\noindent \textbf{Notation:} We will say that  $U$ is \textit{normalized} if 
the line subbundle 
$\mathcal{F}_1\subset U_p$ is trivial. The parameter $k$ is often called the 
\textit{level}.

Let $\omega\in H^2(C)$ be the fundamental class of our curve $C$,  and $e_1, 
..., e_{2g} $ a basis of $H^1(C)$, such that $e_ie_{i+g}=\omega$ for $1\leq 
i\leq g$, and all other intersection numbers $e_ie_j$ equal 0. For a class 
$\delta\in H^*(P\times C)$ of a product, we introduce the following notation 
for its K\"unneth components (cf. 
\cite{Wittenrevisited}):
\begin{equation}\label{Kunneth}
\delta=\delta_{(0)}\otimes1+\sum_i\delta_{(e_i)}\otimes 
e_i+\delta_{(2)}\otimes\omega\in \bigoplus_{i=0}^2H^{*-i}(P)\otimes H^i(C).
\end{equation}
Later, we will need the following formula, which can be proved by a 
straightforward calculation.
\begin{lemma}\label{devisiblelevel} 
$2c_1(\mathcal{L}_d(r; d,...,d)) = c_2(\mathrm{End}_0(U_d))_{(2)},$ where $\mathrm{End}_0$ stands for traceless endomorphisms. 
\end{lemma}
  
\begin{subsection}{Walls and chambers}\label{S2.4}
The central question we address in this paper is how the moduli space of 
stable parabolic bundles depends on the choice of parabolic weights. Let $W$ be 
a vector bundle of degree $ d $ with a fixed full flag $ \mathrm{F}_* $ 
of the 
fiber $W_p$, and 
let us try to determine the structure of the set of parabolic weights 
$c\in\Delta_d$ for which $ W $ is stable. Clearly, for this we need to study 
the set of parabolic weights $ c=(c_1,c_2,\dots c_r) $ for which one can find a 
proper subbundle $ W'\subset W $ such that
\begin{equation}\label{parzero}
\textit{parslope}(W')=\textit{parslope}(W)=0.
\end{equation}
A subbundle  $ W^\p\subset W $ determines a  short exact sequence of parabolic 
bundles $$0\to W^\p\to W\to W^{\pp}\to 0$$  and the position of $ W^\p_p $ with 
respect to $ \mathrm{F}_*$ gives rise to a 
nontrivial partition of the 
set $ \{1,2,\dots,r\}$ into two sets, $\Pi^{\p}$ and  $\Pi^{\pp}$ (cf. \cite{MehtaSeshadri}, definition 1.7); the 
parabolic weights of $W^\p$ and $W^{\pp}$ are then $c^\p=(c_i)_{i\in\Pi^{\p}}$ 
and  $c^{\prime\prime}=(c_i)_{i\in\Pi^{\pp}}$, correspondingly. 
The slope condition \eqref{parzero} translates into a pair of equivalent 
equalities:
\begin{equation}\label{wall}
	d^{\p}=\sum_{i\in\Pi^{\pp}}{c_i}, \quad d^{\pp}=\sum_{i\in\Pi^{\pp}}{c_i},
\end{equation}
where $d^\p$, $d^{\pp}=d-d^\p$  are the degrees  of $W^\p$ and 
$W^{\pp}$, respectively. This means that the critical values of 
$ c\in\Delta_d $ for which \eqref{parzero} is possible 
lie on the union of affine hyperplanes (or \textit{walls}) defined by the 
equations
\[    \sum_{i\in\Pi^{\p}}{c_i}=l,\;\text{where }l\in\Z\text{, and 
}\Pi'\subset\{1,2,\dots,r\}\text{ 
nontrivial.}
 \]

As only finitely many of these walls intersect the simplex $ \Delta_d $, their 
complement is a finite union of open polyhedral \textit{chambers}. It is easy 
to 
verify that as we vary $ c $ inside one of these chambers, the stability 
condition, and thus the moduli space $ 
P_d(c) $ does not change. 

\begin{example}\label{exwall}
Consider the case of rank-3 degree-0 stable parabolic bundles with parabolic 
weights $c=(c_1,c_2,c_3)\in\Delta_0$. The set $\Delta_0$ is an open triangle 
with vertices $(0,0,0), (\frac{2}{3},-\frac{1}{3},-\frac{1}{3})$ and 
$(\frac{1}{3},\frac{1}{3},-\frac{2}{3})$ (cf. Figure \ref{fig:3small}), and 
there exist only two essentially different stability conditions. The wall 
separating the two regimes is given by the condition $c_2=0$. We write $ P_0(>) 
$ for the moduli space
$P_0(c_1,c_2,c_3)$  with $c_2>0$, and $ P_0(<) $ for
$P_0(c_1,c_2,c_3)$ with  $c_2<0$.
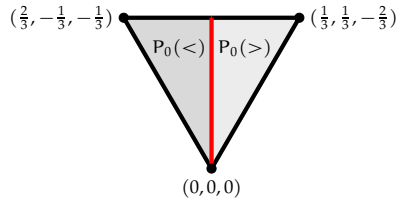
\begin{figure}[H]
\centering
\begin{tikzpicture}[scale=2]
\draw [ultra thick , fill=lightgray!30] (0,0) -- ({1/sqrt(3)},1) -- (0,1);
\draw [ultra thick , fill=gray!30] (0,0) -- ({-1/sqrt(3)},1) -- (0,1);
\draw  [red, ultra thick] (0,0) -- (0,1);
\node [below] at (0,0) {\tiny $(0,0,0)$};
\node [left] at ({-1/sqrt(3)},1) {\tiny $(\frac{2}{3},-\frac{1}{3},-\frac{1}{3})$};
\node [right] at ({1/sqrt(3)},1) {\tiny $(\frac{1}{3},\frac{1}{3},-\frac{2}{3})$};
\draw [fill] (0,0) circle [radius=0.03];
\draw [fill] ({-1/sqrt(3)},1) circle [radius=0.03];
\draw [fill] ({1/sqrt(3)},1) circle [radius=0.03];
\node [above] at ({sqrt(3)/8},{2/3}) {\tiny $P_0(>)$};
\node [above] at ({-sqrt(3)/8},{2/3}) {\tiny $P_0(<)$};
\end{tikzpicture}
\setlength{\belowcaptionskip}{-8pt}\caption{The space of admissible parabolic weights for rank $r=3$.}\label{fig:3small}
\end{figure}
\end{example} 
\end{subsection}
\end{section}

\begin{section}{Wall-crossing in the Verlinde formula}
	\label{sec:wcverlinde}

A key component of our approach is the notion of
{\em diagonal basis}  and the associated
generalized Bernoulli polynomials introduced for general 
hyperplane 
arrangements in \cite{Szimrn}. Using this formalism, we will 
be able to formulate our main result, Theorem 
\ref{main}.

\subsection{Notation}\label{sec:notation}
We begin by setting up some extra notation for the space of
parabolic weights introduced in \S\ref{S2.1}. 

\begin{itemize}
\item Let $ V=\R^r/\R(1,1,\dots,1) $ be the $ r-1 
$-dimensional 
vector space, obtained as the quotient of $ \R^r  $. 
The dual space $ V^* $ is then 
naturally represented as
\[ V^*=\{a=(a_1,\dots,a_r)\in\R^r|\; a_1+\dots+a_r=0\}. \]
Let $ x_1,x_2,\dots,x_r $ be the coordinates on $ \R^r $; 
given $ a\in V^* $, we will write $\scp ax  $ for the 
linear function $ \sum_ia_ix_i $ on $ V $. We will 
sometimes identify this linear function with the vector $ a 
$ itself.
\item The vector space $ V^*$ is endowed with a lattice $ 
\Lambda $ of full rank:
\[ \Lambda=\{\lambda=(\lambda_1,\dots,\lambda_r)\in\Z^r|\; 
\lambda_1+\dots+\lambda_r=0\}. \]
In particular, for  $ 1\le i\neq j\le r $, we can define 
the 
element $ \alpha^{ij} =x_i-x_j $ in $ \Lambda $.
\item 
Our arrangement is the set of  hyperplanes  
$\{x_i=x_j\}\subset V $, $ 1\le i< j\le r $.
It will be convenient for us to think about this set as the set of roots
of the $A_{r-1} $ root system with the opposite roots identified:
\[ \rr = \{\ \pm\alpha^{ij}  |\,1\le i< j\le r\}.  \]
Note that $ V^*$ carries a natural action of the permutation group $ \Sigma_r 
$, permuting the coordinates $ x_j,\,j=1,\dots,r $, \label{gact} and this action restricts 
to an action on $ \Phi $ as well.
\item The basic object of the theory is an \textit{ordered} 
linear basis $ \mathbf{B} $  of $ V^* $ consisting of the 
elements of $ \Phi$. Let us denote the set of these objects 
by $ 
\Bases $:
\[ \Bases= 
\left\{\mathbf{B}=\left(\bta1,\dots,\bta{r-1}\right)
\in\Phi^{r-1}|\;
 \bb 
\text{ -- basis of }  V^*
\right\} \]
\item For  $ \bb\in\Bases $, we will write $ 
\flag(\mathbf{B}) $ for the full flag $$ \left[ V^*=\langle 
\bta1,\bta2,\dots, \bta{r-1}\rangle_{\mathrm{lin}}, \dots,\langle 
\bta{r-1},\bta{r-2} 
	\rangle_{\mathrm{lin}},\langle 
	\bta{r-1}
	\rangle_{\mathrm{lin}} 
\right] ,$$
where $ \langle\cdot\rangle_{\mathrm{lin}} $ stands for 
linear span.
\end{itemize}

\subsection{Diagonal bases}
\begin{definition}
	\begin{itemize}
		\item For $ \tau\in\Sigma_{r-1} $ and  
		$\bb\in\Bases 
		$, we will write $ \bb\circlearrowleft \tau$ for the permuted 
		sequence $ 
		(\bta{\tau(1)},\bta{\tau(2)},\dots,\bta{\tau(r-1)}) 
		$.
		\item For two elements $\bb,\mathbf{C}\in\Bases $ 
		we will write $ \bb\dashv\mathbf{C} $ if for any $ 
		\tau\in\Sigma_{r-1} $, we have $ 
		\flag(\bb\circlearrowleft \tau)\ne\flag(\mathbf{C}) $.
		\item A subset $ \DD\subset\Bases $ of $ (r-1)! $ 
		elements is called a \textit{diagonal basis} if for 
		any two different elements $ \bb,\mathbf{C}\in\DD 
		$, 
		we have  $ \bb\dashv\mathbf{C} $.
	\end{itemize} 
\end{definition}

\begin{remark}
	This definition is motivated by a construction \cite{Szimrn}, which 
	associates to each diagonal basis $ \DD $  a pair of 
	dual 
	bases of the middle homology and the cohomology of the 
	complexified hyperplane arrangement on $ 
	V\otimes_{\R}\C $ defined by $ \Phi $. The 
	dimension of these (co)homology spaces is $ (r-1)! $.
\end{remark}

\subsection{Combinatorial interpretation}
\label{sec:combinat}
	This notion has the following purely combinatorial 
	form.\begin{itemize}
		\item  We can think of $ \Phi $ as the edges of 
		the 
		complete graph on $ r $ vertices. 
		\item Then the set $ \Bases $ may be thought of as 
		the set 
	of spanning trees of this graph with edges 
	enumerated from $ 1 $ to $ r-1 $. We will introduce the 
	notation \label{tree}
\[ \bb \mapsto \tree(\bb)\]
for this ordered tree.
\item In this language, the flag $\flag(\mathbf{B}) $ 
corresponds to a sequence of $ r $ nested 
	partitions of the vertices (starting with the total 
	partition into 
	1-element sets and ending with the trivial 
	partition) associated to $ \tree(\bb) $, the $ j $th 
	partition being the one 
	induced by the first $ j-1 $ edges. For 
	example, the 
	ordered tree $ [(2,4)(1,3),(1,2)] $ induces the same 
	sequence of partitions as $ [(1,4),(2,3),(1,2)] $ (see Figure \ref{fig:hpa})
	\item A diagonal 
	basis $ \DD $ is then a set of $ (r-1)! $ ordered 
	trees such that the $ (r-1)! $ partition sequences 
	obtained by 
	reordering the edges of any one of the ordered trees 
	are different from $ (r-1)!-1 $ sequences of partitions 
	obtained from the remaining elements of $ \DD $. 
	\end{itemize}
	
\begin{figure}[H]\label{fig:hpa}
\centering
\begin{tikzpicture}[scale=1.2]
\draw (0,0)-- (1,1) node [midway,  above, sloped, red] {\tiny 1};
\draw (0,0) -- (1,0) node [midway, below, sloped, red] {\tiny 2};
\draw (0,1) -- (1,1) node [midway, above, sloped, red] {\tiny 3};
\node [below] at (0,0) {\tiny 1};
\draw [fill] (0,0) circle [radius=0.06];
\node [below] at (1,0) {\tiny 2};
\draw [fill] (1,0) circle [radius=0.06];
\node [above] at (0,1) {\tiny 4};
\draw [fill] (0,1) circle [radius=0.06];
\node [above] at (1,1) {\tiny 3};
\draw [fill] (1,1) circle [radius=0.06];
\end{tikzpicture}\hspace{1.5cm}
\begin{tikzpicture}[scale=0.8]
\node [below] at (0,2) {\tiny 1};
\draw [fill] (0,2) circle [radius=0.04];
\draw (0,1.9) circle [radius=0.35];
\node [below] at (1,2) {\tiny 2};
\draw [fill] (1,2) circle [radius=0.04];
\draw (1,1.9) circle [radius=0.35];
\node [below] at (2,2) {\tiny 3};
\draw [fill] (2,2) circle [radius=0.04];
\draw (2,1.9) circle [radius=0.35];
\node [below] at (3,2) {\tiny 4};
\draw [fill] (3,2) circle [radius=0.04];
\draw (3,1.9) circle [radius=0.35];
\node [below] at (0,1) {\tiny 1};
\draw [fill] (0,1) circle [radius=0.04];
\draw (0,0.9) circle [radius=0.35];
\node [below] at (1,1) {\tiny 2};
\draw [fill] (1,1) circle [radius=0.04];
\draw (1,0.9) circle [radius=0.35];
\node [below] at (2,1) {\tiny 3};
\draw [fill] (2,1) circle [radius=0.04];
\node [below] at (3,1) {\tiny 4};
\draw [fill] (3,1) circle [radius=0.04];
\draw  (2.5,0.9) ellipse (0.85 and 0.4);
\node [below] at (0,0) {\tiny 1};
\draw [fill] (0,0) circle [radius=0.04];
\node [below] at (1,0) {\tiny 2};
\draw [fill] (1,0) circle [radius=0.04];
\node [below] at (2,0) {\tiny 3};
\draw [fill] (2,0) circle [radius=0.04];
\node [below] at (3,0) {\tiny 4};
\draw [fill] (3,0) circle [radius=0.04];
\draw  (2.5,-0.1) ellipse (0.85 and 0.4);
\draw  (0.5,-0.1) ellipse (0.85 and 0.4);
\node [below] at (0,-1) {\tiny 1};
\draw [fill] (0,-1) circle [radius=0.04];
\node [below] at (1,-1) {\tiny 2};
\draw [fill] (1,-1) circle [radius=0.04];
\node [below] at (2,-1) {\tiny 3};
\draw [fill] (2,-1) circle [radius=0.04];
\node [below] at (3,-1) {\tiny 4};
\draw [fill] (3,-1) circle [radius=0.04];
\draw  (1.5,-1.15) ellipse (1.8 and 0.55);
\end{tikzpicture}
\vskip 0.2cm
\setlength{\belowcaptionskip}{-8pt}\caption{$\bb=(\alpha^{1,3},\alpha^{1,2},\alpha^{3,4})$}
\end{figure}
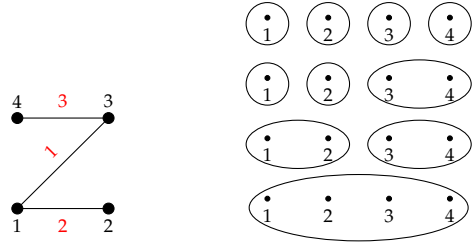

\subsection{Examples}
There are essentially 2 known constructions of diagonal 
bases 
\cite{Szimrn}.

 \textbf{ The Hamiltonian basis.} For each permutation $ 
\sigma\in \Sigma_r $, we can define 
\begin{equation}\label{eq:defsigmaB}
\sigma(\bb) = 
(\alpha^{\sigma(r-1),\sigma(r)},\alpha^{\sigma(r-2),\sigma(r-1)},\dots,
\alpha^{\sigma(1),\sigma(2)})
\in\Bases. 
\end{equation}
The set $ \HH_m
=\{\sigma(\bb)|\;\sigma\in\Sigma_r,\,\sigma(1)=m\} $
is then a diagonal basis. In the combinatorial description, 
this diagonal basis corresponds to the set of Hamiltonian 
paths starting at vertex $ m $, and endowed with the reversed natural 
ordering of edges.
\begin{example}\label{exsymmbasis}\mbox{}
Here are some examples of Hamiltonian bases:  \begin{itemize}
\item for $r=3$:    $\HH_1 = \{(\alpha^{2,3},\alpha^{1,2}), 
(\alpha^{3,2},\alpha^{1,3})\}$,
\item  and for $r=4$: \begin{multline*}
\HH_1 = \{(\alpha^{3,4},\alpha^{2,3},\alpha^{1,2}), (\alpha^{2,4},\alpha^{3,2},\alpha^{1,3}), (\alpha^{4,3},\alpha^{2,4},\alpha^{1,2}), \\ (\alpha^{3,2},\alpha^{4,3},\alpha^{1,4}), (\alpha^{4,2},\alpha^{3,4},\alpha^{1,3}), (\alpha^{2,3},\alpha^{4,2},\alpha^{1,4})
\}
\end{multline*}
\end{itemize}
\end{example}
II. \textbf{ The no-broken-circuit bases.} Let $ 
\upsilon:\{1,\dots,r(r-1)/2\}\to\Phi$ be a total 
ordering, which we will represent as 
an order relation $ \overset{\upsilon}{<} $ on $ \Phi
$. To 
this ordering, one can associate the following, so called 
\textit{noncommutative no-broken-circuit diagonal basis} 
\cite{Szimrn}:
\begin{multline*}
\DD[\upsilon] = 
\left\{\left(\bta1,\dots,\bta{r-1}\right)\in\Bases\left| 
\quad
\bta1\overset{\upsilon}{<}\dots\overset{\upsilon}{<}\bta{r-1},\right.
 \,\text{and}
\right. \\ 
\left. \alpha^{ij}\overset{\upsilon}{<}\bta m\Rightarrow 
(\alpha^{ij},\bta m,\dots,\bta{r-1})\text{ linearly 
independent} \right\}.
\end{multline*}
\begin{example}\label{exnbcbasis}  Let $\alpha^{1,3}\overset{\upsilon}{<}\alpha^{1,4}\overset{\upsilon}{<}\alpha^{2,3}\overset{\upsilon}{<}\alpha^{2,4}\overset{\upsilon}{<}\alpha^{1,2}\overset{\upsilon}{<}\alpha^{3,4}$ be the ordering of the positive roots for rank $r=4$. Then 
\begin{multline*}
\DD[\upsilon] = \{(\alpha^{1,3},\alpha^{1,2},\alpha^{3,4}), (\alpha^{1,3},\alpha^{1,4},\alpha^{2,3}), (\alpha^{1,3},\alpha^{1,4},\alpha^{2,4}), \\ (\alpha^{1,3},\alpha^{1,4},\alpha^{1,2}), (\alpha^{1,3},\alpha^{2,3},\alpha^{3,4}), (\alpha^{1,3},\alpha^{2,3},\alpha^{2,4})
\}
\end{multline*}
is the corresponding no-broken-circuit diagonal basis.
\end{example}
\begin{remark}\label{rem:symmdiag} The hyperplane 
arrangement induced by $ \Phi $ is invariant under the 
natural action of $\Sigma_r $ on the vector space $V$.		
It follows easily from the definition that if $ \DD $ is a 
diagonal basis and $ \sigma\in\Sigma_r $ is a permutation, 
then $ \sigma(\DD) $ is also a diagonal basis.
\end{remark}
\end{section}

\begin{section}{The residue formula and the main result}
	\label{sec:residuemain}

In this section, we recall the residue formula from \cite{Szimrn} for $ 
Ver(k,\lambda) $, the 
discrete Verlinde sum on the right hand side of \eqref{eq:ver0}. The key 
feature of this formula is that it exposes the 
piecewise polynomial nature of $ Ver(k,\lambda) $, which is 
key for our wall-crossing analysis. While the objects are relatively simple, the formalism is heavy with notation, so we begin by describing the 
1-dimensional case. 

\subsection{The residue formula in dimension 1} 
The story begins with the Fourier series
\begin{equation}\label{eq:Bernoulli}
 \frac{1}{(2\pi 
i)^m}\sum_{n\in\Z\setminus{0}}\frac{\exp(2\pi i a n)}{n^m}
\end{equation}
for $m\geq 2$, which is a periodic, piecewise polynomial function given by 
the 
formula
\[ 
\res_{x=0} 
\frac{\exp(\{a\}x)}{1-\exp(x)}\frac{dx}{x^m}, \]
where $ \{a\} $ is the fractional part of the real number $ 
a $. The polynomial functions thus obtained on the 
interval 
$ [0,1] $ are called \textit{Bernoulli polynomials}. The 
polynomial on the interval containing the real number
 $ c\in\R\setminus\Z $  is given by
\[  
\res_{x=0} 
\frac{\exp((a-[c])x)}{1-\exp(x)}\frac{dx}{x^m},  \]
where $ [c] $ is the integer part of $ c $.

Now we pass to a trigonometric version of this formula,
calculating finite  sums of values of rational 
trigonometric functions 
over rational points with denominators equal to an integer 
$ k $. 

We replace thus the rational function $ x^{-m} $ by the 
(hyperbolic) trigonometric function $ f(x)=(2\sinh(x/2))^{-2m} 
$, and introduce an integer parameter $ \lambda $ related 
to $ a $ via $ ka=\lambda $.
We consider the sum of values of the function $ f $ over a finite set of rational points in analogy with \eqref{eq:Bernoulli}: 
\[ \sum_{n=1}^{k-1} 
\frac{exp(2\pi i
\lambda n/k)}{(2\sin(\pi n/k))^{2m}},
 \]
 where $ \lambda,k\in\Z $. 
This sum is again periodic in $ \lambda \mod k$, and for $m\geq 2$ we 
can evaluate 
it via the residue theorem as
\begin{multline*}
(-1)^m\res_{z=1} 
\frac{z^{k\{\lambda/k\}}}{(z^{1/2}-z^{-1/2})^{2m}}
\cdot\frac{k\,dz}{z(1-z^k)}
\overset{z=\exp(x/k)}=\\
(-1)^m\res_{x=0}\frac{\exp(\{\lambda/k\}\cdot 
x)}{1-\exp(x)}\cdot
f(x/k) \,dx.
\end{multline*}
Again, this is a piecewise polynomial function in the pair 
$ (k,\lambda) $, which is polynomial in the cones bounded 
by the lines $ \lambda=qk $, $ q\in\Z $.

Note that in these calculations, a key role is played by 
the Bernoulli operator:
\begin{equation}\label{berop}
 f \mapsto \mathrm{Ber}[f](a) = \frac{f(x)\exp(ax)\, 
 dx}{1-\exp(x)},
\end{equation} 
which transforms meromorphic functions in the variable $ x 
$ into polynomials in $ a $, and  plays the role of a 
generalized Fourier operator.

\begin{subsection}{The multidimensional case}
Now we return to the setup of \S\ref{sec:wcverlinde} with the vector space $ V$ 
endowed with the hyperplane arrangement $ \Phi	 $.   We introduce 
the notation $ \fw $ \textit{for 
the space of meromorphic functions defined in a neighborhood of $ 0 
$ in $ V\otimes_\R \C $ with poles on the union of hyperplanes} 
\[  \bigcup_{1\le i<j\le r}\{x|\;\scp{\alpha^{ij}}x=0\}.  \]
In particular, the function 
\[ w_\Phi=\prod_{i<j} \left( 2\mathrm{sinh}(\pi(x_i-x_j))\right)\]
is an element of $ \fw $.

To write down our residue formula, we need a 
multidimensional generalization of the notions of integer 
and fractional parts. 
Given a basis $ \bb=(\bta1,\dots,\bta{r-1})\in\Bases$ of $ 
V^*$, and an element $ a\in V^* $, we define  $[a]_{\bb} $ 
and  $\{a\}_{\bb} 
$ to be the unique elements of $ V^*$  satisfying 
\begin{itemize}
	\item 	$ [a]_{\bb}=a-\{a\}_{\bb}\in 
	\Lambda$, 
	and
	\item $ \{a\}_{\bb}\in\sum_{j=1}^{r-1}[0,1)\bta j. 
	$
\end{itemize}
This notion naturally induces a chamber structure on $ 
V^* $: we will call $ a\in V^* $ 
\textit{regular}  if $ a $ is a point of continuity for the 
functions $ a\mapsto [a]_{\bb}, \{a\}_{\bb}$ for all $ 
\bb\in\Bases $, i.e. when
$ \{a\}_{\bb}\in\sum_{j=1}^{r-1}(0,1)\bta j$. Now, for 
regular $ 
a $ and $ b $ we  define the equivalence relation 
 \begin{equation}\label{equivalence}
 	 a\sim b 
 \text{ when }[a]_{\bb} = [b]_{\bb}\quad \forall 
 \bb\in\Bases.
 \end{equation} 
The equivalence classes for this relation form a $ 
\Lambda$-periodic system of \textit{chambers} in $ V^* $.

\noindent \textbf{Convention}: We will think of  a  partition $ \Pi $ of $ 
\{1,2,\dots, r\} $ into two nonempty sets as an ordered partition $ 
\Pi=(\Pi^{\p},\Pi^{\pp})$  such that $ r\in\Pi^{\pp} $, and we will call these 
objects \textit{nontrivial partitions} for short.

\begin{lemma}\label{lemmaBwall}
	The equivalence classes of the relation $ \sim $ are precisely the chambers  
	in $ V^* $ created by the walls parameterized by a nontrivial
	 partition $ \Pi=(\Pi^{\p},\Pi^{\pp})$ of the 
	first $ r $ positive integers, and an integer $l$:
\begin{equation}\label{Swall}
		 S_{\Pi,l} =\{c\in V^*|\;\sum_{j\in\Pi^{\p}}c_j=l\}
\end{equation}
\end{lemma}

\begin{remark}\label{rem:2walls}
Note that the walls given in \eqref{Swall} are precisely 
the same as the ones given in \eqref{wall} for the case 
$ d=0 $, where 
they play the role of walls separating the chambers of 
parabolic weights $ c $ in which the parabolic moduli 
spaces $ P_0(c) $ are naturally the same. This "coincidence" is precisely 
what we need for our comparative wall-crossing strategy.
There is a small terminological issue here: the "chambers" 
in \S\ref{S2.4} are the intersections of the 
equivalence classes of $ \sim $ defined above with the open
simplex $ 
\Delta_0\overset{\mathrm{def}}=\Delta$, where the parabolic weights live (cf. Figures \ref{fig:3small} and \ref{fig:3big}). We will use 
the term 
"chamber" in both cases if this causes no confusion.
\end{remark}
\begin{figure}[H]
\centering
\begin{tikzpicture}[scale=2]
\draw  (0,1) -- (0,2) node [midway, above, sloped] {\tiny $S_{(\{2\}, \{1,3\}),0}$};
\draw  [red] (0,0) -- (0,2);
\draw [ultra thick , fill=lightgray!30] (0,0) -- ({1/sqrt(3)},1) -- (0,1);
\draw [ultra thick , fill=gray!30] (0,0) -- ({-1/sqrt(3)},1) -- (0,1);
\draw  [red, ultra thick] (0,0) -- (0,1);
\draw  (0,2) -- ({sqrt(3)},1) node [midway, above, sloped] {\tiny $S_{(\{1,2\}, \{3\}),1}$};
\draw (0,2) -- ({-sqrt(3)},1) node [midway, above, sloped] {\tiny $S_{(\{1\}, \{2,3\}),1}$};
\draw  [red] (0,2) -- ({sqrt(3)},1);
\draw [red] (0,2) -- ({-sqrt(3)},1);
\node [below] at (0,0) {\tiny (0,0,0)};
\node [right] at ({sqrt(3)},1) {\tiny (0,1,-1)};
\node [left] at ({-sqrt(3)},1) {\tiny (1,-1,0)};
\node [above] at (0,2) {\tiny (1,0,-1)};
\node [above] at ({sqrt(3)/8},{2/3}) {\tiny $P_0(>)$};
\node [above] at ({-sqrt(3)/8},{2/3}) {\tiny $P_0(<)$};
\draw (0,0) -- ({sqrt(3)},1) node [midway, below, sloped] {\tiny $S_{(\{1\}, \{2,3\}),0}$};
\draw (0,0) -- ({-sqrt(3)},1) node [midway, below, sloped] {\tiny $S_{(\{1,2\}, \{3\}),0}$};
\draw [red] (0,0) -- ({sqrt(3)},1);
\draw [red] (0,0) -- ({-sqrt(3)},1);
\end{tikzpicture}
\setlength{\belowcaptionskip}{-8pt}\caption{Chambers for rank $r=3$.}\label{fig:3big}
\end{figure}
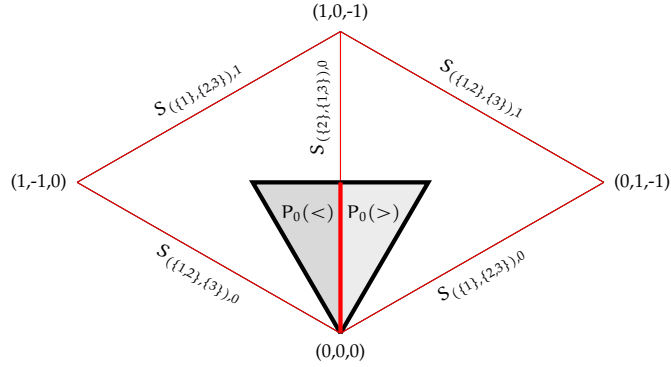

Each element $ 
\mathbf{B}=(\beta^{[1]},\dots,\beta^{[r-1]})\in\Bases $ 
defines an \textit{iterated} version of the Bernoulli 
operator
\eqref{berop} on the space of functions $\fw$:
 interpreting the elements $ a  , \bta j \in V^*$ as linear 
 functions on $ V $, we define

  \begin{equation}\label{defiber}
	\iber_{\mathbf{B}}   \left[ f(x) \right](a)=
\frac1{(2\pi i)^{r-1}} \int\displaylimits_{Z(\bb)} \frac{ 
	f(x)\exp \scp {a}x\;d\scp{\beta^{[1]}}x
	\wedge 
	\dots\wedge d\scp{\beta^{[r-1]}}x\;}{(1-exp(\scp{\beta^{[1]}}x))\;
	\dots(1-exp(\scp{\beta^{[r-1]}}x))\;},
\end{equation}
where the naturally oriented cycle $ Z_{\bb} $
is defined by
\[ Z_{\bb} = \{v\in 
V\otimes_{\R}\C:\,|\scp{\beta^{[j]}}x|\;=\varepsilon_j,\, 
j=\,\dots,r-1 \}\subset 
V\otimes_{\R}\C\setminus\{w_\Phi(x)=0\},\]
with real constants $ \varepsilon_j $ satisfying
$ 0\le\varepsilon_{r-1}\ll\dots\ll \varepsilon_{1} $. Thus 
again,  $ \iber_\bb $ 
is a linear operator associating to a function in $ \fw $ 
a polynomial on $ V^* $.

\begin{remark}\label{rem:iterated}
Let us make a small remark about the computational aspects 
of the operator $ \iber_\textbf{B} $. Denoting the coordinate $ 
\scp{\beta^{[j]}}x $ by $ y_j $, $ j=1,\dots,r-1 $, and 
writing $ f $ and $ a $ in these coordinates: $ f(x)=\hat 
f(y) $, $ \scp ax=\scp {\hat a}y $, we can rewrite 
\eqref{defiber} 
as 
\[  \iber_{\mathbf{B}}   \left[ f(x) \right] (a) =
\res_{y_1=0}\dots\res_{y_{r-1}=0}\frac{ 
	\hat f(y)\exp \scp {\hat a}y\;dy_1
	\wedge\dots\wedge dy_{r-1}}{(1-exp(y_1))\;
	\dots(1-exp(y_{r-1}))\;},
\]
where \textit{iterating} the residues here means that we keep the 
variables with lower indices as unknown constants, and then use 
geometric series expansions of the type
\[ \frac{1}{1-exp(y_1-y_2)} = \frac{y_1-y_2}{1-exp(y_1-y_2)}\cdot\frac1{y_1-y_2}=\frac{y_1-y_2}{1-exp(y_1-y_2)}\cdot \sum_{n=0}^\infty 
\frac{y_2^n}{y_1^{n+1}}. \]
\end{remark}
\end{subsection}
\subsection{Invariance of diagonal bases and the main 
results}\label{S4.3}

Diagonal bases have the following key invariance property.
\begin{theorem}[\cite{Szimrn}]
\label{diaginv}
Let $ f\in\fw $, and $ c\in V^* $ be regular; let $ \DD $ 
be a diagonal basis  of $ \Phi $. Then the functional (cf. 
\eqref{defiber} above)
\[ f\mapsto \sum_{\bb\in\DD}
\iber_{\bb}[f(x)](a-[c]_{\bb})  \]
transforming a meromorphic function $ f\in\fw  $ into a 
polynomial in the variable $ a\in V^*$ is independent of 
the choice of the 
diagonal basis $ \DD$. In particular, for regular $a\in V^*$, the functional
\begin{equation}\label{eq:defbertrafo}
	f\mapsto 
\sum_{\bb\in\DD}
\iber_{\bb}[f(x)](\{a\}_{\bb}) 
\end{equation}
transforms $ f $ into a well-defined piecewise polynomial 
function on $ V^* $, which is polynomial in each chamber.
\end{theorem}

As this functional is invariantly defined, it is not 
surprising that it is equivariant with respect to the 
symmetries of our hyperplane arrangement. For $ \sigma\in\Sigma_r $, we define, as usual
\begin{equation}\label{eq:sigmaaction}
\sigma\cdot f(x)= f(\sigma^{-1}x).
\end{equation}
This convention is consistent with \eqref{eq:defsigmaB}.
\begin{lemma}\label{lem:symres}
Let $ f\in\fw $, and $ \sigma\in\Sigma_r $, and pick any diagonal basis $ \DD $.
Then
\[ \sum_{\bb\in\DD}
\iber_{\bb}[f(x)](\sigma\cdot a-[\sigma\cdot c]_{\bb}) =\sum_{\bb\in\DD}
\iber_{\bb}[\sigma^{-1}\cdot f(x)](a-[c]_{\bb}) \]
\end{lemma}
\begin{proof}
Indeed, it is sufficient to note that  
\[ \scp{\sigma\cdot a-[\sigma\cdot c]_{\bb}}x=  \scp{\sigma\cdot (a-[c]_{\bb})}x
=  \scp{a-[c]_{\bb}}{\sigma^{-1}(x)}, \]
perform the linear substitution $ x=\sigma(y) $, and 
conclude that 
\[ \sum_{\bb\in\DD}
\iber_{\bb}[f(x)](\sigma\cdot a-[\sigma\cdot c]_{\bb})
=\sum_{\bb\in\sigma^{-1}(\DD)}
\iber_{\bb}[\sigma^{-1}\cdot f(x)](a-[c]_{\bb}).  \] 
Now the statement follows from the fact that $ 
\sigma\in\Sigma $ takes a diagonal basis to  another 
diagonal basis (cf. Remark \ref{rem:symmdiag}). 
\end{proof}

\begin{remark}\label{hamform}
By picking the Hamiltonian diagonal basis $ \HH_1=\{\sigma\cdot\bb_0|\;\sigma\in\mathrm{Stab}(1,\Sigma_r)\} $,  we can turn the argument in the proof above around, and obtain the following formula:
\begin{multline*}\label{hamformula}
\sum_{\bb\in\HH_1}
\iber_{\bb}[f(x)](a-[c]_{\bb})=
\sum_{\sigma\in\mathrm{Stab}(1,\Sigma_r)}
\iber_{\bb_0}[\sigma\cdot f(x)](\sigma\cdot a-[\sigma\cdot 
c]_{\bb})=
\\
\res_{y_1=0}\dots\res_{y_{r-1}=0}
\sum_{\sigma\in\mathrm{Stab}(1,\Sigma_r)}
\frac{ 
	\sigma\cdot f(y)\exp \scp {\sigma\cdot a-[\sigma\cdot 
	c]_{\bb} }y\;dy_1
	\wedge\dots\wedge dy_{r-1}}{(1-exp(y_1))\;
	\dots(1-exp(y_{r-1}))\;},
\end{multline*}
where
\[ {\bf{B_0}} = 
(y_{1}=x_{r-1}-x_{r},\dots,y_{r-2}=x_2-x_3,y_{r-1}=x_1-x_2)\in\Bases.
 \]

\end{remark}
Now we are ready to write down the residue formula for the 
Verlinde sums proved in \cite[Theorem 4.2]{Szduke}. Recall that we denoted by 
$ \Ver(k,\lambda) $ the finite sum on the right hand side of \eqref{eq:ver0}.

\begin{theorem}\label{residuethm}
Let $g\geq 1$, $ k\in\Z^{>0} $, $ 
\lambda
\in \Lambda$, and let $ \DD $
 be any diagonal basis of $ \Phi$. Introducing the 
notation $ \kk=k+r $, and $ \lala=\lambda+\rho $, we have  
\begin{equation}\label{eq:residueformula}	 
\Ver(k,\lambda)
= \tilde{N}_{r,k} \sum_{\bb\in\DD}
\iber_{\bb}  
\left[w_\Phi^{1-2g}(x/\kk)\right] \left( 
\lala/\kk-[\, \widehat{\k c} \,]_{\bb}\right),
\end{equation}
where  $\tilde{N}_{r,k}=(-1)^{{r \choose 2}(g-1)} N_{r,k}$ (cf. \eqref{eq:ver0}) and $\widehat{\k c}\in V^*$ is a regular point in a chamber that contains $\lala/\kk$ in its closure.
\end{theorem}

Now, if we look at our main goal \eqref{eq:ver0}: 
proving the equality
\begin{equation}\label{mainequality}
	 \Ver(k,\lambda)  = 
\chi(P_0(\lambda/k),\L_0({k;\lambda})),
\end{equation} 
then we discover a rather embarrassing mismatch. Both sides 
are piecewise polynomial functions, however,
\begin{itemize}
\item according to the HRR theorem, $ 
\chi(P(\lambda/k),\L_0({k;\lambda})) $ is  polynomial on the 
cones over the the equivalence classes (cf. \eqref{equivalence}) of $ 
\lambda/k $, while 
\item  according to \eqref{eq:residueformula}, $ 
\Ver(k,\lambda) $ is  polynomial on the cones over 
the equivalence classes  of $ \lala/\kk $,
\end{itemize}
and these conic partitions of $ \{(k,\lambda)|\;\lambda/k\in\Delta\}$ could clearly be different (cf. Figure \ref{orangegreen} for a sketch of this problem).

\begin{figure}[H]\label{fig4}
\begin{tikzpicture}[scale=0.6]
\path[fill=orange!20] (4.5,6) -- (6,6) -- (3,3) ;
\draw  [lightgray, ->] (0,0) -- (9,0);
\node [right] at (9,0) {\tiny $\lambda$};
\draw  [red, ->] (0,-3) -- (0,6);
\node [above] at (0,6) {\tiny $k$};
\draw [red] (0,0) -- (3,6);
\draw [red] (0,0) -- (6,6);
\draw  [gray] (0,1) -- (9,1);
\node [left] at (0,1) {\tiny $1$};
\draw [ultra thick, teal] (0,1) -- (0.5,1);
\draw [ultra thick, orange] (0.5,1) -- (1,1);
\draw  [lightgray] (0,-3) -- (9,-3);
\node [left] at (0,-3) {\tiny $-r$};
\node [left] at (0,-2) {\tiny $-r+1$};
\node [left] at (0,0) {\tiny $(0,0)$};
\draw [purple] (0,-3) -- (4.5,6);
\draw [purple] (0,-3) -- (9,6);
\draw  [dashed,purple] (0,-3) -- (0,6);
\draw  [gray] (0,-2) -- (9,-2);
\draw [ultra thick, teal] (0,-2) -- (0.5,-2);
\draw [ultra thick, orange] (0.5,-2) -- (1,-2);
\draw  [dashed, lightgray] (0,5) -- (9,5);
\node [above] at (3.05,5) {\tiny $\lambda$};
\node [above] at (3.6,5) {\tiny $\widehat{\lambda}$};
\node [above] at (0.6,1) {\tiny $\frac{\lambda}{k}$};
\draw [fill] (0.425,-2) circle [radius=0.035];
\node [above] at (0.42,-2) {\tiny $\frac{\widehat{\lambda}}{\widehat{k}}$};
\draw [ultra thick, teal] (0,5) -- (2.5,5);
\draw [ultra thick, orange] (2.5,5) -- (5,5);
\draw [ultra thick, teal] (0,4.85) -- (3.93,4.85);
\draw [ultra thick, orange] (3.93,4.85) -- (7.85,4.85);
\draw [dashed, lightgray] (3.1,5) -- (0,0);
\draw [dashed, lightgray] (3.4,5) -- (0,-3);
\draw [fill] (3.1,5) circle [radius=0.035];
\draw [fill] (3.4,5) circle [radius=0.035];
\draw [fill] (3.1,5) circle [radius=0.035];
\draw [fill] (3.4,5) circle [radius=0.035];
\draw [fill] (0.62,1) circle [radius=0.035];
\end{tikzpicture}
\setlength{\belowcaptionskip}{-8pt}\caption{$\lambda/k$ is in the orange chamber, while ${\widehat{\lambda}}/{\widehat{k}}$ is in the green chamber.}\label{orangegreen}
\end{figure}
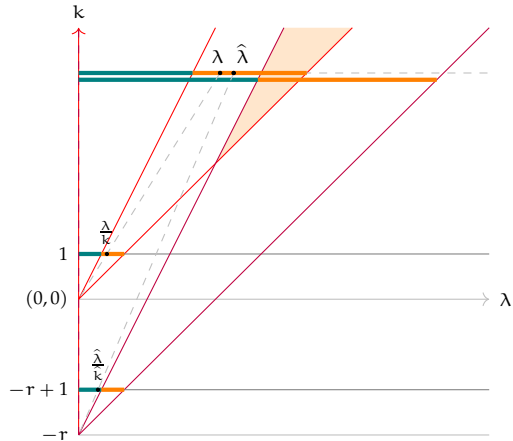

Thus for  \eqref{eq:ver0} to be true, some miracle needs to occur, and these 
miracles are well-known in the area of "quantization commutes with reduction" 
\cite{MSj, Vergne, SzV}. We will return to this problem in \S\ref{S10}, but for 
now, we will be satisfied to use \eqref{eq:residueformula} to write down a 
(conjectural for the moment) formula for $ 
\chi(P_0(\lambda/k),\L_0({k;\lambda})) $, which is manifestly polynomial on the 
cones where $ \lambda/k $ is in a fixed equivalence class.

Let us fix a regular $ c\in\Delta $ marking a particular chamber in $ \Delta $.
The two cones $ \{(k;\lambda)|\,\lambda/k\sim c\} $ and  $ 
\{(k;\lambda)|\,\lala/\kk\sim c\} $ intersect along an open cone (this cone is 
shaded in orange on Figure \ref{orangegreen}),
 and on this intersection, the expression
\begin{equation}\label{goodexpr}
	 \sum_{\bb\in\DD}
\iber_{\bb}  
\left[w_\Phi^{1-2g}(x/\kk)\right] \left( 
\lala/\kk-[\lambda/k]_{\bb}\right) 
\end{equation} 
coincides with the right hand side of \eqref{eq:residueformula}. As  
\eqref{goodexpr} is manifestly polynomial on each cone where $ \lambda/k $ is 
in a particular chamber in $ \Delta $, this expression will be then our main 
candidate for $ \chi(P_0(\lambda/k),\L_0({k;\lambda})) $.

Our plan is thus to split the proof of \eqref{mainequality} into three parts: 
the first is equality \eqref{eq:residueformula}, and the other two are given in 
our main theorem below. We formulated all our statements in a manner that 
allows us to treat the cases when $ \lambda/k $ or $ \lala/\kk $ are on a 
boundary separating two of our chambers in $ \Delta $.

\begin{theorem}\label{main}
	Let $ \lambda \in\Lambda$ and $ k\in\mathbb{Z}^{>0} $ be such that $ 
	\lambda/k\in\Delta$. Let ${\k c}$ and $ \widehat{\k c} \in\Delta$ be 
	regular elements, specifying two chambers in $\Delta$, which contain 
	$\lambda/k$ 
	and $\lala/\kk$ in their closures, correspondingly.
Then for any diagonal basis $ 
	\DD  $, the following two equalities hold:
\begin{equation} \label{thmI}
	\chi(P_0({\k c}),\L({k;\lambda}))
	=  \tilde{N}_{r,k}\sum_{\bb\in\DD}
	\iber_{\bb}  \tag{I.}
	\left[w_\Phi^{1-2g}(x/\kk)\right] \left( 
	\lala/\kk-[{\k c}]_{\bb}\right),   
\end{equation}
{ and }
 \begin{equation}\label{thmII}
 	\sum_{\bb\in\DD}
	\iber_{\bb}  
	\left[w_\Phi^{1-2g}(x/\kk)\right] \left( 
	\lala/\kk-[{\k c}]_{\bb}\right) =   \sum_{\bb\in\DD} \iber_{\bb}  \tag{II.}\left[w_\Phi^{1-2g}(x/\kk)\right] \left(\lala/\kk-[\, \widehat{\k c} \,]_{\bb}\right). 
 \end{equation}
\end{theorem}
\begin{remark}\label{well-def-non-reg}
Part \ref{thmI} of the theorem implies that if $\lambda/k\in\Delta$ is not 
regular, then  
$$\chi(P_0({c^+}),\L({k;\lambda}))=\chi(P_0({c^-}),\L({k;\lambda})),$$ for 
regular $c^\pm\in\Delta$ in two neighboring chambers that contain $\lambda/k$ 
in their closure (cf. Proposition \ref{wcrgen} and Remark 
\ref{well-def:lambda/k}).
\end{remark}

Before we proceed, we formulate a mild generalization of part \ref{thmI} 
of our theorem. 
As observed above, if we fix a generic $ c\in\Delta $, and vary $ (\lambda,k )$ 
in such a way that $ \lambda/k\sim c $, then both sides of the 
equality \eqref{thmI}  are manifestly polynomial, and thus we can extend the 
validity of this equality as follows.
\begin{corollary}\label{cor:thmwithk}
Let $ c\in\Delta $ be a regular element, which thus 
specifies a chamber in $ \Delta $ and a parabolic moduli 
space $ P_0(c) $ as well. Then for a diagonal basis $ \DD $, 
an arbitrary weight $\lambda\in\Lambda$, and a positive integer $ k $, 
we have 
\begin{equation}\label{eqnpoly}
\chi(P_0(c),\L(k;\lambda))
	=  \tilde{N}_{r,k}\sum_{\bb\in\DD}
	\iber_{\bb} [ w_\Phi^{1-2g}(x/\kk)] (\lala/\kk- 
	[c]_\bb).
\end{equation}
\end{corollary}

\begin{example}\label{ex2poly} Let us write down these 
formulas in case of $ r=3 $ explicitly. Let $\DD$ be the 
diagonal basis from Example \ref{exsymmbasis}; then using Remark \ref{hamform}, we obtain
\begin{multline*}
 \chi(P_0(<),\mathcal{L}(k;\lambda))= (-1)^{g-1}(3(k+3)^2)^{g} \\  \res_{y=0}\res_{x=0}  \frac{e^{\lambda_1x+(\lambda_1+\lambda_2)y+x+y}-e^{\lambda_1x+(\lambda_1+\lambda_3)y+x}}{(1-e^{x(k+3)})(1-e^{y(k+3)})w_\Phi(x,y)^{2g-1}}dx dy
\end{multline*}
and 
\begin{multline*}
 \chi(P_0(>),\mathcal{L}(k;\lambda))= (-1)^{g-1}(3(k+3)^2)^{g} \\ \res_{y=0}\res_{x=0} 
 \frac{e^{\lambda_1x+(\lambda_1+\lambda_2)y+x+y}-
 	e^{\lambda_1x+(\lambda_1+\lambda_3)y+x+y(k+3)}}
 {(1-e^{x(k+3)})(1-e^{y(k+3)})w_\Phi(x,y)^{2g-1}}\;dx dy,
\end{multline*}
where 
$w_\Phi(x,y)=2\mathrm{sinh}(\frac{x}{2})2\mathrm{sinh}(\frac{y}{2})2\mathrm{sinh}(\frac{x+y}{2})$.
\end{example}

\subsection{The walls}
Our first step is to identify the wall-crossing terms of the residue 
formula \eqref{eqnpoly}, which originate in the 
discontinuities of the function $c\mapsto \{c\}_{\bb} $. 
These discontinuities occur on "walls": the affine 
hyperplanes  \eqref{Swall}. The following is straightforward:
\begin{lemma}\label{lemmaBBwall}
	Let $S_{\Pi,l} $ be the wall defined by \eqref{Swall}, and
	 $ \bb=(\bta1,\dots,\bta{r-1})\in\Bases $ an ordered basis of $ V^* $.
	Then, as a function of $c$, the fractional part function
	$ \{c\}_{\bb} $ has a discontinuity on the wall $ 
	S_{\Pi,l}  $ exactly when $ \tree(\bb) $ (cf. page \pageref{tree}) is a union 
	of a tree on $ \Pi^{\p}$, a tree on $ \Pi^{\pp} $ (the 
	enumeration of the edges is irrelevant here) and a 
	single 
	edge  (which we will call the \textbf{link}) connecting $ \Pi^{\p}$ 
	and $ \Pi^{\pp} $. 
\end{lemma}
\noindent \textbf{Notation:} We will denote the element of $\bb$  corresponding 
to this edge by $\beta_{\mathrm{link}}$; this vector thus depends on $\bb$ and 
the partition $\Pi$.

Now choose two regular elements $ c^+,c^-\in V^*$ in two neighboring chambers separated by the wall 
$ S_{\Pi,l}$, in such a way that
\begin{equation}\label{cpm}
 [c^+_{\Pi^\p}]=l\text{ \, and \, }
[c^-_{\Pi^\p}]=l-1,
\end{equation}
where $$c_{\Pi^\p} \overset{\mathrm{def}}= \sum_{i\in\Pi^{\p}}c_i,$$
and, as usual, $ [q] $ stands for the integer part of the real number $ q $.
 Now  introduce the notation 
\[ p_\pm(k;\lambda) 
= \tilde{N}_{r,k}\sum_{\bb\in\DD}
	\iber_{\bb} [ w_\Phi^{1-2g}(x/\kk)] (\lala/\kk- 
	[c^\pm]_\bb)
 \]
for the two polynomial functions in $ (k,\lambda) $ corresponding to 
$ c^+ $ and $ c^- $, respectively.
We define the \textit{ wall-crossing term} in our residue formula 
\eqref{eqnpoly} as the difference between these two
polynomials:
$$ p_+(k;\lambda) - p_-(k;\lambda). $$
Using Lemma \ref{lemmaBwall} and \eqref{cpm}, we 
obtain the following simple residue formula for this difference.
\begin{lemma}\label{wallcrossprop}
Let $ (\Pi,l) $, $ c^+ $ and $ c^- $ be as above, and let 
us fix a diagonal basis $ \DD\subset\Bases $.  Denote by $ 
\DD|\Pi $	the subset of those elements of $ \DD $, which 
satisfy the condition described in Lemma \ref{lemmaBBwall}.
Then 
	\begin{multline}\label{wcresidue}
		p_+(k,\lambda) - p_-(k,\lambda)  = \\ \tilde{N}_{r,k}
		\sum_{\bb\in\DD|\Pi} 
		\iber_{B}\left[(1-\exp(\beta_{\mathrm{link}}(x)))
		w_\Phi^{1-2g}(x/\kk)\right]
		\left(\lambda/\widehat{k}-[c^+]_\bb\right),
	\end{multline}
where  $ \beta_{\mathrm{link}} $ is the "link" element of $ 
\bb $ (depending on $ \Pi $ and $ \bb $) defined
after Lemma \ref{lemmaBBwall}.
\end{lemma}
\begin{remark}
	Note that the multiplication by $ 
1-\exp(\beta_{\mathrm{link}}(x)) $ in \eqref{wcresidue} has the effect of 
canceling one of the factors in the denominator in the 
definition \eqref{defiber} of the operation $ \iber $.

\end{remark}\begin{example}\label{exdiff} Calculating the difference of two 
polynomials from Example \ref{ex2poly}, we get the wall-crossing term for rank 
3 case:
\begin{equation*}
\begin{split}
p_-(k;\lambda)-p_+(k;\lambda)= (-3(k+3)^2)^{g}  \res_{y=0}\res_{x=0}  \frac{e^{\lambda_1x+(\lambda_1+\lambda_3)y+x}}{(1-e^{x(k+3)})w_\Phi(x,y)^{2g-1}}dx dy.
\end{split}
\end{equation*}
\end{example}

\subsection{Wall-crossing and diagonal bases}\label{S4.5}

Now we pass to the study of the combinatorial object $ 
\DD|\Pi $ defined in Lemma \ref{wallcrossprop}.
One thing we will discover is that even though each 
diagonal basis consists of $ (r-1)! $ elements and the 
right hand side of \eqref{wcresidue} does not depend on the 
choice of $ \DD $, the number of elements in $ \DD|\Pi $ 
might vary with $ \DD $.

First we look at the case of the Hamiltonian basis $ \HH_1 $.
Form now on, we will use the notation $ |\Pi'|=r' $ and $ 
|\Pi''|=r'' $ for a nontrivial partition $ \Pi=(\Pi',\Pi'') $, (recall the convention $ r\in\Pi^{\pp}$). The following statement is easy to verify.
\begin{lemma}
	Let $ \Pi=(\Pi',\Pi'') $ be a nontrivial partition, such that $1\in\Pi^{\p}$ (the other case 
	is analogous). Then 
	$$ \HH_1|\Pi = \{\sigma(\bb)|\;\sigma(1)=1,\text{ 
	and }\sigma(\Pi')\in\Pi'\}. $$
	In particular, $ |\HH_1|\Pi|=(r'-1)!\cdot r''!. $
\end{lemma}

It turns out that for our geometric applications, instead of $ \HH_1 $, we will 
need to choose a particular nbc-basis, where the ordering is chosen to be 
consistent with $ \Pi $.

To simplify our terminology, we will use the language of 
graphs and edges introduced in \S\ref{sec:combinat}, and we 
will think of $ 
\alpha^{ij}\in\Phi $ as an edge in the complete graph on 
$ r $ vertices.
To define the ordering $ \upsilon $, we need to choose an edge between 
$ \Pi',\Pi'' $; the choice is immaterial, but for simplicity we settle for $ 
m\overset{\mathrm{def}}=\mathrm{max}\{i\in\Pi^{\p}\}$ and $ r\in\Pi^{\pp}$, and set $ \alphalink=\alpha^{m,r}$ 
to be the smallest element according to $ \upsilon $.

The $ \upsilon $-ordered list of edges thus starts with $ 
\alphalink$, and then continues with the remaining $ 
r^\p\cdot r^{\pp}-1 $ edges connecting $ \Pi^{\p}$ and $ \Pi^{\pp} $. 
Next we list the $ r'(r'-1)/2  $ edges connecting vertices 
in $ \Pi^{\p}$ in any order, and finally, we list the 
remaining edges, those connecting vertices in $ \Pi^{\pp}$.

\noindent \textbf{Notation:}
We introduce the natural notation $ \Phi' $ and $ 
\Phi^{\pp} $ for the $ A_{r^{\p}} $ and $ A_{r^{\pp}} $ 
root systems corresponding to $ 
\Pi^{\p} $ and $ \Pi^{\pp} $, and we 
denote by $ \DD[\upsilon] $, $ 
\DD'[\upsilon] $ and $ \DD^{\pp}[\upsilon] $, the diagonal 
nbc-bases 
induced by the ordering $ \upsilon$ on $ \Phi$, $ \Phi' $ and $ 
\Phi^{\pp} $, respectively.

\begin{lemma}\label{diagonalwall}
Given elements $ \bb'\in  \DD'[\upsilon]$ and $ \bb''\in  
\DD''[\upsilon]$, we can define an element of $ \DD[\upsilon] $ as 
follows: we start with $ \alphalink$, then append $ 
\bb' $, and then continue with $ \bb'' $. This construction 
creates a	one-to-one correspondence 
 \begin{equation}\label{dtauisom}
 	\DD'[\upsilon] \times \DD''[\upsilon]  \to  \DD[\upsilon]|\pi;
 \end{equation}
in particular, $ |\DD[\upsilon]|\pi|=(r^\p-1)!\cdot(r^{\pp}-1)! $.
\end{lemma}

Finally, putting Lemmas \ref{wallcrossprop} and 
\ref{diagonalwall} together, we arrive at the following 
elegant statement:
\begin{proposition}\label{wcrnice}
	Let $ (\Pi,l) $, $ c^+ $ and $ c^- $ be as in Lemma  
	\ref{wallcrossprop}, and let $ \DD' $ and $ \DD'' $ be 
	diagonal bases of $ \Phi^\p $ and $ \Phi^{\pp}$ 
	correspondingly.  Then 
	\begin{multline}\label{wcfinal}
		p_+(k;\lambda) - p_-(k;\lambda) = (k+r)\tilde{N}_{r,k}\cdot \\
		\sum_{\bb^\p\in\DD^\p} \sum_{\bb^{\pp}\in\DD^{\pp}}
	\res_{\alphalink=0}	
	\iber_{\bb^\p}\iber_{\bb^{\pp}}\left[w_\Phi^{1-2g}(x/\kk)\right]
		\left(\lala/\widehat{k}-[c^+]_\bb\right)\,d\alphalink,
	\end{multline}
where $ \res_{\alphalink=0}	
\iber_{B^\p}\iber_{B^{\pp}}\,d\alphalink $ is simply $ \iber_{\bb} $ (cf 
\eqref{defiber}) with $ \bb $ obtained by appending $\bb'$, 
and then $ \bb^{\pp} $ to $ \alphalink$, 
and 
with  the factor $ (1-\exp\scp{\alphalink}x) $ 
removed from the denominator.
\end{proposition}
\begin{remark}\label{rem:ibers}
	The expression 
	\[ 
	\res_{\alphalink=0}	
	\iber_{\bb^\p}\iber_{\bb^{\pp}}\left[w_\Phi^{1-2g}(x/\kk)\right]
	\left(\lala/\widehat{k}-[c^+]_\bb\right)\,d\alphalink \]
	 may equally be interpreted as follows. We write $$ 
	 \Lambda\ni
	 \lala/\widehat{k}-[c^+]_\bb =m_{\mathrm{link}} \alphalink + n^\p + 
	 n^{\pp}$$ according to the splitting of $ \bb $, think 
	 of $ w(x/\kk) $ as a function in $ \mathcal{F}_{\Phi^{\pp}} $ with some 
	 fixed values of the parameters from $ \bb^{\p} $ and $ 
	 \alphalink$, and then calculate
	 \[ \iber_{\bb^{\pp}} [w_\Phi^{1-2g}(x/\kk)](n^{\pp}). \]
	 The result will be a rational function $ Q $ in the 
	 variables from $ \bb^{\p} $ and $ 
	 \alphalink $, and we proceed to calculate
	 $ \iber_{\bb^{\p}}[Q](n^{\p}) $ to obtain a function $ F $ 
	 in the variable $ \alphalink$, and finally 
	 the answer is $ \res_{\alphalink=0} 
	 \exp(\alphalink)F(\alphalink) 
	 d\alphalink  $.
	 
	 We observe that since the trees $ \tree(\B^{\p}) $  
	 and $ \tree(\B^{\pp}) $ are disjoint, the order of the 
	 application of the operations $ \iber_{B^\p}$ and $\iber_{B^{\pp}} $ 
	 is immaterial.
\end{remark}

\end{section}

\begin{section}{Wall crossing in master space}
\label{sec:wcmaster}
	Master spaces were introduced by Thaddeus  in 
	\cite{ThaddeusFlip} in order to 
	understand GIT quotients when varying linearizations. 
	Following his footsteps, 
	in this section, we describe a simple but very 
	effective 
	method to control the changes in the Euler 
	characteristics 
	of line bundles when crossing a wall in the space of 
	linearizations.
	
\subsection{Wall-crossing and holomorphic Euler 
characteristics}

We begin by recalling the basic notions of Geometric 
Invariant Theory.

 Let $X$ be a smooth projective variety over $\mathbb{C}$,  
 and $G$ a reductive group acting on $X$. 
A \textit{linearization} of this action is a line 
bundle $L$ on $X$ with a 
lifting of the $G$-action to a linear action on $L$. 
An ample linearization is $G$-effective, if $L^n$ 
has a nonzero $G$-invariant section for some $n>0$; 
the space of such linearizations is called the 
$G$\textit{-effective ample cone}; we denote this cone by $ \coneg(X) $.

For $ L\in\coneg(X) $, we define the invariant-theoretic quotient 
$X\sslash^LG$ as the Proj of the graded ring of invariant 
sections of the powers of $ L $:
\[ 
M_L = \mathrm{Proj}  \bigoplus_nH^0(X, L^n)^G.
 \]
According to Mumford's Geometric Invariant Theory  \cite{MumFog}, 
there is a partition 
of $ X$ 
depending on $ L $:
\begin{equation}\label{Xpartition}
	X  = X^\mathrm{s}[L]\cup X^{\mathrm{sss}}[L]\cup X^{\mathrm{us}}[L] 
\end{equation}
into the set of stable, strictly semistable, and unstable 
points, such that there is a surjective map $ (X^\mathrm{s}[L]\cup 
X^{\mathrm{sss}}[L])/G\to M_L $, which is a bijection if $ 
X^{\mathrm{sss}}[L] $ is empty, and the quotient $ X^\mathrm{s}[L]/G 
$ is a smooth orbifold.

In \cite{DolgachevH}, Dolgachev and Hu studied the 
dependence of the GIT quotient $M_L=X\sslash^LG$ on $L$. 
They showed that
 $\coneg(X)$ is divided by 
hyperplanes, called walls, into finitely many convex 
chambers, such that when $L$ varies within a chamber, the 
partition \eqref{Xpartition} and thus the GIT quotient $ 
M_L $ remains unchanged. Moreover, an ample effective linearization 
lies on a wall precisely when it possesses a strictly 
semistable point. 

Now let us consider two neighboring chambers, with smooth GIT 
quotients $ M_+ $ and $ M_- $. We  pick an arbitrary 
linearization  $ \L $  of the $G$-action on $ X $, which descends to $ M_+ $ 
and $ M_- $. This last condition means that if $ S\subset G $ 
is the stabilizer of a generic point in $ X $, then $ S $ 
acts trivially on the fibers of $ \L $. We will call such 
linearizations
\textit{descending}.

Thus, given such a descending linearization $ \L $  of the $G$-action on $ X $, 
we obtained 
two line bundles: one on $ M_+ $ and one on $ M_- $, 
which,  by abuse of notation, we will denote by the same 
letter $ \L $.
Via taking Chern classes, this construction creates a 
correspondence between classes in $ H^2(M_+,\Z) $ and $ 
H^2(M_-,\Z) $, which we will assume to be an isomorphism of 
free $\Z $-modules. We will thus identify these lattices, and 
introduce the notation $ \Gamma $ for them:
\[ \Gamma = H^2(M_+,\Z) \simeq H^2(M_-,\Z).  \]
The walls mentioned above can be thought of as hyperplanes in
$\Gamma_{\R}=\Gamma\otimes_{\Z}\R $.

Our goal in this section is to compare the
holomorphic Euler characteristics  $ \chi(M_+,\L)$ and 
$\chi(M_-,\L) $, which are given by the Hirzebruch-Riemann-Roch theorem:
\[ \chi(M_\pm,\L) = \int_{M_\pm} \exp(c_1(\L)) 
\mathrm{Todd}(M_\pm). \] 
As this expression is manifestly polynomial in 
 $ c_1(\L) $, we obtain thus two polynomials on $ \Gamma $, and our goal is
to calculate their difference, the \textit{wall-crossing term}
 \begin{equation}\label{chidiff}
	\chi(M_+,\L) - \chi(M_-,\L).
\end{equation}

\subsection{The master space construction}\label{S5.2}

To simplify our setup, we will make some additional 
assumptions. 
\begin{assumptions}\label{assumptions}
\begin{enumerate}
	\item The generic stabilizer of $ X $ is trivial. 
	\item Let $L_+$ and $L_-$ be two ample 
	linearizations  of the $G$-action
	on $X$ from the adjacent chambers corresponding to the 
	quotients $M_+$ and 
	$M_-$. Without loss of 
	generality, we can assume that the linearization 
	$L_0=L_+\otimes{L}_-$ lies on the single wall 
	separating 
	the 
	two chambers, and that the interval connecting $ 
	c_1(L_+) $ and $ c_1(L_-) $ in $\Gamma_{\R}=\Gamma\otimes_{\Z}\R $ does not 
	intersect any other walls. 
	\item 	Let	
	 $ X^0 $ be the set of those semistable points 
	 $ x\in X^\mathrm{ss}[L_0] $ which are not stable for $ L_\pm $: 
	 \[ X^0 :=  X^\mathrm{ss}[L_0] \setminus(X^s[L_+]\cup X^s[L_-]) 
	 \]
	We assume that $ X^0 $ is  smooth, and that for $ x\in X^0 $ the  stabilizer subgroup $ G_x\subset G $ is isomorphic to 
	$\mathbb{C}^*$. 
	\item Assume that	
	there is a 
	linearization 
	$ \Lvec $  of the $G$-action on $ X $ such that $ L_+=L_-\otimes \Lvec^n $
	for some positive 	integer $ n	$, and such 
	that for each $ x\in X^0 $, the 
	stabilizer subgroup $ G_x $ acts freely on $ 
	L_x\setminus{0}$.
\end{enumerate}
\end{assumptions}

Now we introduce the \textit{master space } construction of Thaddeus  
\cite{ThaddeusFlip}. Consider the variety 
$Y=\mathbb{P}(\mathcal{O}\oplus{\Lvec})$, which is a  
$\mathbb{P}^1$-bundle over $X$ endowed with the additional 
$\mathbb{C}^*$-action $(1, t^{-1})$. As $ Y $ is a 
projectivization of a vector bundle 
on $ X $, it comes equipped with $ \OO(1) $, which is  the standard $ 
{G}\times \C^*_{} $-equivariant line bundle. To simplify our notation, we will denote the same 
	way the linearizations  of the $G$-action on $ X $ and their pull-backs (with tautological $G$-action) to $ 
	Y $.

The \textit{master space} $ Z  $ then is the GIT quotient 
of  $ Y $ with respect to the linearization $ L_{-}(n)=L_-\otimes 
\OO(n)$:
\[ Z = Y\sslash^{L_{-}(n)}{{G}}, \]
which inherits a $ \C^* $-action from $ Y $. Some additional notation:
\begin{itemize}
	\item We will denote this copy of $\mathbb{C}^*$ by $ 
	T$,
	\item  the projection $ Y\to X $ by $ \pi $, and the 
	quotient map $ Y^s\to Z  $ by $ \psi$.
	\item Introduce the notation $ Y(0:\cdot) $ and $ Y(\cdot:0) $ for 
	the two copies of $ X $ in $ Y $, corresponding to the 
	two poles of the projective line; then $ Y $ is partitioned into 3 sets:
	\[ Y = Y(0:\cdot) \sqcup Y(\cdot:0) \sqcup \Lvec^\circ, \]
	where $ \Lvec^\circ $ is the line bundle $ \Lvec $ with the zero-section 
	removed. We will write $ \pi_\circ $ for the restriction of $\pi$ to $ 
	\Lvec^\circ $. We can collect our maps on the following diagram.
\begin{center}
	\begin{tikzcd}
		\Lvec^\circ\arrow[rd, "\pi_\circ"] \arrow[r, 
		hook] & Y= \mathbb{P}(\mathcal{O}\oplus\Lvec) \supset Y^{\mathrm{s}} 
		\arrow[d, "\pi"]  \arrow[r, "\psi"]  & Z\\
		& X 
	\end{tikzcd}
\end{center}
\end{itemize}

\begin{proposition}\label{BundleN}
	\begin{enumerate}
		\item 
		There are embeddings 
		\[ \iota_-:M_-\to Z\quad\text{and}\quad 
		\iota_+:M_+\to Z \]
		obtained as the quotients $ Y^{\mathrm{s}}\cap Y(\cdot:0) /G $ 
		and\\
		$ Y^s\cap Y(0:\cdot)/ G $, correspondingly.
		\item The strictly semistable locus of $ Y $ with 
		respect to the 
		linearization $L_-(n) $
		 is empty, and 
		the GIT quotient $ Z=Y^\mathrm{s}/ G $ is smooth. 
		\item There is an embedding  
		$ \iota_0:X^0/G \to Z 
		$, obtained via $ \psi({\pi_\circ}^{-1}(X^0)) $.
		 We denote the image of $\iota_0$ by $Z^0$.
		\item 
		The fixed point locus $ Z^T $ is the disjoint 
		union of $ \iota_+(M_+) $, $ \iota_-(M_-) $, and $ 
		Z^0 $.
	
	\end{enumerate}	
\end{proposition}

\begin{proof}
(1)-(3) follow from \cite[4.2, 4.3]{ThaddeusFlip}. To prove (4), first note that $Y(\cdot:0)$ and $Y(0:\cdot)$ are fixed by $T$, so we immediately obtain that $M_\pm \subset Z$ are fixed components. 
Also the $G$-action on $Y$ commutes with the $T$-action, so a point $\psi(y)\in{\psi(\pi_\circ^{-1}(X))}$ is fixed by $T$ if and only if the $T$-orbit $T\cdot{y}\subset\pi_\circ^{-1}(X)$ is contained in the $G$-orbit $G\cdot{y}\subset\pi_\circ^{-1}(X)$. Since $T\cdot{y}\subset\pi_\circ^{-1}(x)$ for some $x\in{X}$, we need $y\in\pi_\circ^{-1}(X^0)$. Moreover, for any $y\in\pi_\circ^{-1}(x)\subset{\pi_\circ^{-1}(X^0)}$, $T\cdot{y}=\pi_\circ^{-1}(x)=G_x\cdot{y}$, so a point $\psi(y)\in{\psi(\pi_\circ^{-1}(X))}$ is fixed by $T$ if and only if $\psi(y)\in{\psi(\pi_\circ^{-1}(X^0))=Z^0}$. 
\end{proof}
\noindent \textbf{Construction:} \label{LBdescent} Given a $ G $-equivariant vector 
bundle $E$ on $X$, we can construct 
a $ T $-equivariant vector bundle $ \zeta(E)\to Z $ on $ Z $ by first pulling $ 
E $ back from $ X $ to $ 
Y $, and endowing the resulting bundle $ \pi^*E $ with  the trivial action of $ 
T $, and the  action 
of $ G $ pulled back from $ X $. We then obtain $ \zeta(E)\to Z $ by 
descending $ \pi^*E $ to $ Z $.  Then it is easy to verify the following.
\begin{lemma}\label{trivrestr}
The restriction of the line bundle $\zeta(\Lvec)$ to $Z^0$ is trivial with $ T 
$-weight 1.
\end{lemma}

Before we formulate our wall-crossing formula, we need one 
more ingredient: the identification of the normal bundles of the fixed point 
components of $ Z $.

\begin{lemma}\label{normallemma}
	\begin{enumerate}
	\item The normal bundle on the component $ M_+ $ of $ Z^T $ is $ \zeta(\Lvec^{-1})\big{|}_{M_+} $, 
	and the normal bundle of $ M_- $ is $ \zeta(\Lvec)\big{|}_{M_-}  $.
	\item The normal bundle $ N_{Z^0} $ of $ Z^0=X^0/G \subset Z $ may be 
	described as the descent of the normal bundle $N_{X^0}$ of $ X^0\subset X $. 
	The weights of the action may be computed by fixing $ x\in X^0 $, 
	identifying the stabilizer subgroup $ G_x\subset G$ with $ 
	T $ via its action on the fiber $ \Lvec_x$, and then considering the action 
	of $ G_x $ on $ N_{X^0} $.
	\end{enumerate}
\end{lemma}

\begin{definition}\label{def:AB}
Given a $ T $-vector bundle $ V $ on a manifold on which $ T $ acts trivially,
the $T$-equivariant \textit{K-theoretical Euler class} of $V^*$, which we denote by  $E_t(V)$, may be described as follows: let $ x_1,\dots, x_n $ be 
the Chern roots of $ V $, and $ l_1,\dots l_n\in\Z $ be the corresponding  $ T 
$-weights. Then 
\[  E_t(V) = \prod_{j=1}^{n}  \left(1-t^{-l_j}\exp(-x_j)\right).     
\]

\end{definition}

Now we are ready to write down our wall-crossing  formula 
for \eqref{chidiff}. A key role will be played by the 
following notion: given a rational 
differential 1-form on the Riemann sphere, let us denote 
taking the sum of residues at 0 and at 
infinity by $ \mu\mapsto \res_{t=0,\infty} \mu $:
\[ \res_{t=0,\infty} \overset{\mathrm{def}}= \res_{t=0} + \res_{t=\infty}.
\]
\begin{theorem}\label{ThaddTh}
 Let $\L$ be a linearization of the $G$-action on $ X $, and denote, as above, by $ \zeta(\L) $ 
 the $T$-equivariant line bundle on $Z$ 
obtained by pull-back to $ Y $ and descent to $ Z $. If Assumptions \ref{assumptions} hold, then 
\begin{equation} \label{wallcrosseq}
	\chi(M_+, \L)-\chi(M_-, \L)=
\res_{t=0,\infty}\int_{Z^0}\frac{\mathrm{ch}_t(\zeta(\L)\big{|}_{Z^0})}{E_t(N_{Z^0})}\mathrm{Todd}(Z^0)\frac{dt}{t},
\end{equation}
where $ N_{Z^0} $ is the $ T $-equivariant bundle on $ Z^0 $ described in Lemma 
\ref{normallemma},     $ 
\mathrm{ch}_t$ is the $ T $-equivariant Chern character,  and  
$E_t(N_{Z^0})$ is the K-theoretical Euler class of $N^*_{Z^0}$.
\end{theorem}
\begin{proof}[Proof of Theorem \ref{ThaddTh}]
The Atiyah-Bott 
fixed-point formula \cite{AtiyahBottfixed} applied to the line bundle $ \L $ on 
our master space $ Z $ yields
\begin{equation}\label{ABeq}
	\chi_t(Z,\L) =
\sum_{F\subset Z^T} \int_F  
\frac{\mathrm{ch}_t(\zeta(\L)\big{|}_F)}{E_t(N_F)}\mathrm{Todd}(F),
\end{equation}
where the sum is taken over the connected components of the fixed 
point locus $ Z^T $.

	In Proposition \ref{BundleN}, we identified these 
	components as $ M_+,M_- $ and $ {Z^0}$. 
	According to Lemma \ref{normallemma}, for $ M_- $, 
	the normal bundle is simply $ \zeta(\Lvec) $, and thus the 
	contribution of $ M_- $ is equal to
	\[ \int_{M_-} \frac{\mathrm{ch}(\L)\mathrm{Todd}(M_-)}
	{1-t\exp(-c_1(\Lvec))}. \]
	A similar calculation gives the contribution of $ M_+ $ 
	as
	\[ \int_{M_+} \frac{\mathrm{ch}(\L)\mathrm{Todd}(M_+)}
	{1-t^{-1}\exp(c_1(\Lvec))}. \]
	
	 We observe that  $ \chi_t(Z,\L) $ is a Laurent 
polynomial in $ t $ since it  is the 
alternating sum of $ T $-characters of finite dimensional 
vector spaces. Thus, as a function 
of $ t $, $ \chi_t(Z,\L) $ has poles only at $ t=0,\infty 
$, 
and by the Residue Theorem, we have 
\[ \res_{t=0,\infty}
\chi_t(Z,\L)\frac{dt}{t}=0.\] 
On the other hand, since 
\[ \res_{t=0,\infty} \frac{A}{1-t^{-1}B} \frac{dt}{t} 
= -A\;\text{ and }\; \res_{t=0,\infty} \frac{A}{1-tB} 
\frac{dt}{t} = A ,\] 
we have 
\begin{multline*}
	\res_{t=0,\infty}\int_{M_-} 
	\frac{\mathrm{ch}(\L)\mathrm{Todd}(M_-)}
	{1-t\exp(-c_1(\Lvec))}\frac{dt}{t} = 
	\chi(M_-,L)\text{ 
		and }\\
	\res_{t=0,\infty}\int_{M_+} 
	\frac{\mathrm{ch}(\L)\mathrm{Todd}(M_+)}
	{1-t^{-1}\exp(c_1(\Lvec))}\frac{dt}{t} = -\chi(M_+,L).
\end{multline*}

Now, applying the functional $ 
\res_{t=0,\infty}  $ to the two sides of 
\eqref{ABeq} multiplied by $ dt/t $
gives us the desired result \eqref{wallcrosseq}.

\end{proof}

\end{section}

\begin{section}{Wallcrossings in parabolic moduli spaces}
	\label{sec:wallcrpara}
In this section we apply Theorem \ref{ThaddTh} to wall crossings in the moduli space of parabolic bundles. 

From now on, we assume that $d=0$, and we write $\Delta$ for the corresponding 
set of admissible parabolic weights $\Delta_0$. Recall from Section \ref{S2.2} that 
for regular $c\in\Delta$, the moduli space of stable parabolic bundles $P_0(c)$ is the 
GIT quotient $XQ\sslash^c PSL(\chi),$ where $XQ$ is the subspace  of the total space of the flag bundle over the Quot scheme. 
Let us fix a partition 
$\Pi=(\Pi^{\p},\Pi^{\pp})$ and an  integer $l$, and introduce the notation $ 
\Delta'_l $ and $ \Delta^{\pp}_{-l}$ for the simplices of parabolic weights of 
$ \Pi^{\p} $ and $\Pi^{\pp} $.
Let $\phi\in \Sigma_r$ be the unique permutation which sends $\{1,...,r^\p\}$ to $\Pi^{\p}$ preserving the order of first $r^\p$ and the last $r^{\pp}$ elements.
We choose $c^0=(c^0_1,...,c^0_r)\in{S}_{\pi,l}$ 
and two regular elements $c^+, c^- \in \Delta$ in two neighboring chambers 
separated by the wall $S_{\Pi,l}$, such that 
$$c^\pm=c^0\pm\epsilon(...,0,1,0,...,0,-1)$$ for some positive 
$\epsilon\in\mathbb{Q}$, where $1$ and $-1$ are on the $\phi(r')^{\text{th}}$ 
and $r^\text{th}$ places, respectively. 
Let $$c^\p=\sum_{i\in \Pi^{\p}}c^0_ix_i\in\Delta^\p_l \text{\,\,  and \,\,} c^{\pp}=\sum_{i\in \Pi^{\pp}}c^0_ix_i\in\Delta^{\pp}_{-l}. $$
For $(k,\lambda)\in\mathbb{Z}\times\Lambda$, consider the 
polynomials $$q_\pm(k,\lambda) = \chi(P_0(c^\pm), 
\mathcal{L}_0(k;\lambda)).$$ 
Our goal is to calculate the difference of these two 
polynomials.

\noindent \textbf{Notation:} To simplify our notation, from now on, we omit the index $t$ from the symbols for equivariant characteristic classes.

\subsection{The master space construction}\label{S6.1}
We construct the master space $Z$ from \S\ref{S5.2} using the 
following data: 
\begin{itemize}
\item a smooth variety $X=XQ$ (cf. \S\ref{S2.2});
\item  linearizations $L^\pm=L(k; \lambda^\pm)$ of the $G$-action on $X$ (cf. \S\ref{S2.2}), such that $\lambda^\pm/k=c^\pm$;
\item the linearization $\Lvec = L(0; x_{\phi(r^\p)}-x_r)$   of the $G$-action on $X$.
\end{itemize}

The following statement is easy to verify.
\begin{lemma}\label{Sigma}(\cite[\S3.2]{BodenH}) The subset $X^0\subset{X}$ is the set of points representing vector bundles $W$ on $C$, such that $W$ splits as a direct sum $W^\p\oplus{W^{\pp}}$, where $W^\p$ and $W^{\pp}$ are, respectively, $c^\p$ and $c^{\pp}$-stable parabolic bundles. Therefore, we have the following description of the locus $Z^0$: 
$$Z^0 = \{W=W^\p\oplus W^{\pp} \, | \,W^\p\in \widetilde{P}_{l}(c^\p); \, W^{\pp}\in\widetilde{P}_{-l}(c^{\pp}); \, \, det(W)\simeq\mathcal{O}\}.$$
\end{lemma}
\begin{remark}\label{CohSplit}
Note that  $Z^0$ is fibered over $Jac^l$ with fibre $P_l(c^\p)\times{P_{-l}(c^{\pp})}$ by the determinant map $\widetilde{P}_l(c^\p)\to Jac^l$ and 
$$H^*(Z^0,\mathbb{Q})\simeq H^*(P_l(c^\p)\times{P_{-l}(c^{\pp})},\mathbb{Q})\otimes H^*(Jac^l,\mathbb{Q}).$$
\end{remark}
\begin{remark}
If the rank of the vector  bundle $W\in \widetilde{P}_l(c)$ is $1$, then $c=l$ and $\widetilde{P}_l(l)$ is isomorphic to  $Jac^l$, while ${P}_l(l)$  is a point.
\end{remark}
Now we need to verify the hypotheses of Theorem \ref{ThaddTh}. First note
that in our present construction $ X $ is not projective, however, it contains 
all semisimple points of the Quot scheme for all possible polarizations, and 
hence the missing points of the Quot scheme have no effect on any of our 
constructions.

Assumptions \ref{assumptions} (1)-(2) are trivially satisfied, so we study the 
action of the stabilizer $G_x\subset{SL_\chi}$ of point $x\in{X}$ on the fiber 
$\Lvec_x\setminus{0}$.               
\begin{itemize}
\item For a general point $x\in X$ the stabilizer of $x$ is  the center $\mathbb{Z}_\chi\subset SL(\chi)$, which acts trivially on the fiber $\Lvec_x\setminus{0}$.
\item For $x\in X^0$, any element of the stabilizer of $x$ induces an automorphism of the corresponding vector bundle $W=W^\p\oplus{W^{\pp}}$, so $G_x\simeq \mathbb{C}^*\times\mathbb{C}^*\subset GL_\chi$. Then  $(t_1,t_2)\in{G}_x$ is in $SL_\chi$ if and only if $t_1^{\chi^\p}t_2^{\chi^{\pp}}=1$, where $\chi^\p=\chi(W^\p)$ and $\chi^{\pp}=\chi(W^{\pp})$. Note that $(t_1,t_2)$ acts on $\Lvec_x$ as $t_1t_2^{-1}$, and we need $t_1=t_2$ (hence $t_1^\chi=1$) for this action to be trivial, so the stabilizer of any  point in $\Lvec_x\setminus{0}$ is the center $\mathbb{Z}_\chi\subset SL_\chi$.
\end{itemize}
Then 
the action of $G={P}SL(\chi)$ is free on $Y\setminus (Y(0:\cdot) \cup Y(\cdot:0))$ and the action of $G_x\subset PSL(\chi)$ on $\Lvec_x\setminus{0}$ induces an isomorphism $G_x\simeq \mathbb{C}^*\simeq T$. 

Now by Theorem \ref{ThaddTh}, the wall-crossing polynomial 
$q_-(k;\lambda)-q_+(k;\lambda)$  is equal to 
\begin{equation}\label{integralt}
\begin{split}
\res_{t=0,\infty}\int_{Z^0} \frac{ch(\mathcal{L}_0(k;{\lambda})\big{|}_{Z^0})}{E(N_{Z^0})}\mathrm{Todd}(Z^0) \, \frac{dt}{t}.
\end{split}
\end{equation}

Note that in our case, the $T$-action on $Z$ is free outside the fixed locus 
$Z^T$, so as a function in $t\in T$, the integral in (\ref{integralt}) may have 
poles only at  $t=0,1,\infty$.
Then, using the Residue Theorem and  substituting $t=e^u$, we conclude that 
(\ref{integralt}) equals 
\begin{equation}\label{integral}
\begin{split}
-\res_{u=0}\int_{Z^0} \frac{ch(\mathcal{L}_0(k;{\lambda})\big{|}_{Z^0})}{E(N_{Z^0})}\mathrm{Todd}(Z^0) \, du,
\end{split}
\end{equation}
and thus our goal is to calculate this integral. 

Our first step is to identify the characteristic classes 
under the integral sign (cf. Proposition \ref{intl} for the 
result). 

We start with the study of the restriction of the line bundle 
$\mathcal{L}_0(k;{\lambda})$ to the fixed locus 
$Z^0\subset Z$.
Let $\mathcal{J}$ be the Poincare bundle over $Jac\times C$, 
such that  $c_1(\mathcal{J})_{(0)}=0$; define $\eta\in H^2(Jac)$ by $(\sum_ic_1(\mathcal{J})_{(e_i)}\otimes e_i)^2=-2\eta\otimes\omega$ (cf. \S\ref{S2.3}), then (cf. \cite{Zagier}) for any $m\in\mathbb{Z}$
\begin{equation}
\int_{Jac}e^{\eta m}=m^g.
\end{equation}
Recall that  for a parabolic weight $c=(c_1,...,c_r)\in \Delta$ we have set $c_{\Pi^\p}= \sum_{i\in\Pi^{\p}}c_i.$

\begin{lemma}\label{restriction}
Let $\lambda=(\lambda_1,...,\lambda_r)\in \Lambda$, $k\in\mathbb{Z}^{>0}$ and let $\Pi=(\Pi^{\p},\Pi^{\pp})$ be a nontrivial partition with $r\in\Pi^{\pp}$. Let 
$$\lambda^\p=\sum_{i\in \Pi^{\p}}\lambda_ix_i \text{\,\,  and \,\,} \lambda^{\pp}=\sum_{i\in \Pi^{\pp}}\lambda_ix_i,$$
and define $\delta$ by $(\lambda/k)_{\Pi^\p}=l+\delta$. Then 
\begin{equation*}
\begin{split}
ch & (\mathcal{L}_0(k;{\lambda})\big{|}_{Z^0})= \, e^{k{\delta}u}exp\left(\frac{\eta k}{r'}+\frac{\eta k}{r^{\pp}}\right)\cdot  \\ & ch(\mathcal{L}_l(k;\lambda_1',...,\lambda'_{r'}-k\delta)\boxtimes\mathcal{L}_{-l}(k;\lambda_1^{\pp},...,\lambda_{r^{\pp}}^{\pp}+k\delta)),
\end{split}
\end{equation*}
where $\boxtimes$ denotes the external tensor product of line bundles on $P_l(c^\p)\times{P_{-l}(c^{\pp})}$.
\end{lemma}

\begin{lemma}\label{H1ParHom}
Denote by $\widetilde{U}'$ and $\widetilde{U}^{\pp}$ the universal bundles over $\widetilde{P}_{l}(c')\times{C}$ and  $\widetilde{P}_{-l}(c^{\pp})\times{C}$ with the standard normalization (cf. \S\ref{S2.3} ), and denote by $\pi$ projections along $C$. Then the 
 equivariant normal bundle to the fixed locus $Z^0\subset{Z}$ is 
$${N_{Z^0}} = R_T^1\pi_*(ParHom(\widetilde{U}',\widetilde{U}''))\oplus R_T^1\pi_*(ParHom(\widetilde{U}'',\widetilde{U}')),$$ 
where $T\simeq\mathbb{C}^*$-action has weights $(-1,1)$. 
\end{lemma}

\begin{subsection}{Calculation of the characteristic classes of ${N_{Z^0}}$}\label{S6.2} Before we calculate the equivariant K-theoretical Euler class of the conormal bundle $N^*_{Z^0}$, we need to  introduce  some notations. Recall that for $1\leq i, j\leq r$, the differences $x_i-x_j\in V^*$ are linear functions on $V$, and the function $x_i-x_j$ corresponds to the linearization $L_0(0;x_i-x_j)$ on $X$, which descends to the line bundle
$\mathcal{L}_0(0;x_i-x_j)$ on the moduli space $P_0(c)$ (cf. \S\ref{S2.2}). As 
in \S\ref{S5.2}, we denote by $\zeta(L_0(0;x_i-x_j)$ the line bundle on $Z$ 
obtained by the pullback and then descent. This way, we obtain a correspondence 
between the linear functions $x_i-x_j$ and the $T$-equivariant  line bundles on 
$Z$.

As $Z^0$ is a connected component of the fixed locus of the $T$- action on $Z$, 
and thus its equivariant cohomology factors:
$ H^*_T(Z^0)\simeq H^*(Z^0)\otimes\C[u] $. In particular, there are canonical embeddings $ H^*(Z^0)\hookrightarrow H^*_T(Z^0) $ and $ \C[u]\hookrightarrow H^*_T(Z^0) $.

Recall the definition of the permutation $\phi\in\Sigma_r$ given at the beginning of this chapter: $ \phi $ takes the first $ r' $ numbers to $ \Pi' $, preserving the order of the first $r^\p$ and the last $r^{\pp}$ elements.
We introduce the symbols 
\begin{equation}\label{identification}
\begin{split}
& z^\p_i-z^\p_j = c_1(\zeta(L_0(0; x_{\phi(i)}-x_{\phi(j)}))\big{|}_{Z^0}), \,\, (1\leq i,j\leq r^\p) \\& z^{\pp}_i-z^{\pp}_j = c_1(\zeta({L}_0(0;x_{\phi(r^\p+i)}-x_{\phi(r^\p+j)}))\big{|}_{Z^0}),  \,\, (1\leq i,j\leq r^{\pp})\\&
 u = (z^\p_{r^\p}-z^{\pp}_r) = c_1(\zeta({L}_0(0;x_{\phi(r^\p)}-x_r))\big{|}_{Z^0})
 \end{split}
 \end{equation}
 for the equivariant cohomology classes in $H^2_T(Z^0)$.  The last equalities 
 are consistent with Lemma \ref{trivrestr}. 
\begin{remark}
Note that  (cf. Remark \ref{CohSplit}) $$z_i^\p-z_j^\p =c_1(\mathcal{F}^\p_{r-i+1}/\mathcal{F}^\p_{r-i}\otimes(\mathcal{F}^\p_{r-j+1}/\mathcal{F}^\p_{r-j})^*)\in H^2(P_l(c^\p)),$$ $$z^{\pp}_i-z^{\pp}_j = c_1(\mathcal{F}^{\pp}_{r-i+1}/\mathcal{F}^{\pp}_{r-i}\otimes(\mathcal{F}^{\pp}_{r-j+1}/\mathcal{F}^{\pp}_{r-j})^*)\in H^2(P_{-l}(c^{\pp})),$$ where $\mathcal{F}^\p_i$ and $\mathcal{F}^{\pp}_i$ are the flag bundles (cf. \S\ref{S2.3}) on $P_0(c^\p)$ and $P_0(c^{\pp})$, correspondingly.
\end{remark} 
Taking into account these identifications, functions on $ V $ give rise to 
equivariant cohomology classes on $ Z^0 $. To make the splitting $ 
H^*_T(Z^0)\simeq H^*(Z^0)\otimes\C[u] $, explicit, however, we will write these 
classes in the form $ f_u(z',z'') $, thinking of them as functions of the 
differences of the $ z_i^{\p} $s and the differences of the $ z_i^{\pp} $s, 
depending on the parameter $ u $. With this convention, we introduce
\begin{equation*}
\begin{split}
&w_u^\times(z^\p,z^{\pp}) = \prod_{\substack{i, j \\ \phi(i)<\phi(r^\p+j)}}2\mathrm{sinh}(z^\p_i-z^{\pp}_j) \prod_{\substack{i, j \\ \phi(r^\p+j)<\phi(i)}}2\mathrm{sinh}(z^{\pp}_j-z^{\p}_i), \\
&\rho_u^{\times}(z^\p,z^{\pp}) = \frac{1}{2} \sum_{\substack{i, j  \\ \phi(i)<\phi(r^\p+j)}}(z^\p_i-z^{\pp}_j)+\frac{1}{2}\sum_{\substack{i, j \\ \phi(r^\p+j)<\phi(i)}}(z^{\pp}_j-z^{\p}_i), \\
\end{split}
\end{equation*}
where according to \eqref{identification}, $$z^\p_i-z_j^{\pp} = (z^\p_i -z^\p_{r^\p}) +u - (z^{\pp}_j -z^{\pp}_{r}) =  c_1(\zeta(\mathcal{L}_0(0; x_{\phi(i)}-x_{\phi(r^\p+j)}))\big{|}_{Z^0})\in H^2_T(Z^0).$$
Now we are ready to write down our formula for the K-theoretical Euler class
$E(N_{Z^0})$ (cf. definition \ref{def:AB} with $t=e^u$).

\begin{proposition}\label{NormalTodd}
\begin{equation*}
\begin{split}
E(N_{Z^0})^{-1} = & \, (-1)^{lr+r^\p r^{\pp}(g-1)}e^{-rlu}exp\left(\frac{\eta r}{r^\p}+\frac{\eta r}{r^{\pp}}\right) w_u^\times(z^\p,z^{\pp})^{1-2g}exp(\rho_u^{\times}(z^\p,z^{\pp})) \\ & ch(\mathcal{L}_l(r^{\pp}; -l,...,-l,-l+rl)\boxtimes\mathcal{L}_{-l}(r^\p; l,...,l,l-rl)). \\
\end{split} 
\end{equation*}
\end{proposition}
\begin{proof}
It follows from the short exact sequence (\ref{SESparhom})  for parabolic morphisms that 
$$ ch(-\pi_!(ParHom(\widetilde{U}^{\pp},\widetilde{U}^\p))=-ch(\pi_!(Hom(\widetilde{U}^{\pp},\widetilde{U}^\p)))+\sum_{\substack{i, j  \\ \phi(i)<\phi(r^\p+j)}}e^{z^\p_i-z^{\pp}_j}$$
and 
$$ch(-\pi_!(ParHom(\widetilde{U}^\p,\widetilde{U}^{\pp}))= -ch(\pi_!(Hom(\widetilde{U}^\p,\widetilde{U}^{\pp}))+\sum_{\substack{i, j \\ \phi(r^\p+j)<\phi(i)}}e^{z^{\pp}_j-z^\p_i},$$
so by Lemma \ref{H1ParHom} 
\begin{equation}\label{chN}
\begin{split}
& ch(N_{Z^0})= ch(-\pi_!(Hom(\widetilde{U}^{\p},\widetilde{U}^{\pp}))\oplus -\pi_!(Hom(\widetilde{U}^{\p},\widetilde{U}^{\pp})^*)) \\ &+ \sum_{\substack{i, j  \\ \phi(i)<\phi(r^\p+j)}}e^{z^\p_i-z^{\pp}_j}+\sum_{\substack{i, j \\ \phi(r^\p+j)<\phi(i)}}e^{z^{\pp}_j-z^\p_i}.
\end{split}
\end{equation}

\begin{lemma}\label{K-th} We denote by $[f(x)]^{W}$ the multiplicative class of the vector bundle $W$ given by the function $f(x)$ in Chern roots of $W$. 
Let $S$ be a vector bundle on $P\times C$ with $T$-weight 1, and $\pi: P\times C \to P$ projection along the curve, then 
$$E(-\pi_!S\oplus -\pi_!S^*)^{-1}=(-1)^{\mathrm{rk}(-\pi_!S)}\frac{exp(-ch_2(S)_{(2)})}{[(2\mathrm{sinh}(x/2))^{2g-2}]^{S_p}}.$$
\end{lemma}
\begin{proof} 
Note that 
$$E(-\pi_!S)^{-1}=\left[\frac{1}{1-t^{-1}e^{-x}}\right]^{-\pi_!S}=\left[\frac{-te^x}{1-te^x}\right]^{-\pi_!S}$$ and 
$$E(-\pi_!S^*)^{-1}=\left[\frac{1}{1-te^{-x}}\right]^{-\pi_!S^*}=\left[\frac{1}{1-te^x}\right]^{(-\pi_!S^*)^*}.$$  
Applying Serre duality and the Grothendieck-Riemann-Roch Theorem we obtain
\begin{multline*}
 ch(-\pi_!S)+ch((-\pi_!S^*)^*)= ch(-\pi_!S) + ch(\pi_!(S\otimes K_C))= \\  ch(-\pi_!S)+ \pi_*(ch(S\otimes{K_C})\mathrm{Todd}(C))= \\  ch(-\pi_!S)+ ch(\pi_!S)+(2g-2)ch(S_p)= (2g-2)ch(S_p),
\end{multline*}
where $K_C$ is the canonical sheaf on the curve $C$, hence 
$$\left[\frac{1}{1-te^x}\right]^{-\pi_!S\oplus(-\pi_!S^*)^*}=\left[\frac{1}{(1-te^x)^{2g-2}}\right]^{S_p}=\frac{exp(-c_1(S_p)(g-1))}{[(2\mathrm{sinh}(x/2))^{2g-2}]^{S_p}}.$$
Since 
$$[-te^x]^{-\pi_!S}= (-1)^{\mathrm{rk}(-\pi_!S)}exp(c_1(-\pi_!S))$$ 
and by the Grothendieck-Riemann-Roch theorem
$$ch_1(-\pi_!S)=ch_1(S_p)(g-1)-ch_2(S)_{(2)},$$ we conclude that 
$$[-te^x]^{-\pi_!S}= (-1)^{\mathrm{rk}(-\pi_!S)}exp(c_1(S_p)(g-1))exp(-ch_2(S)_{(2)}),$$
which finishes the proof of Lemma \ref{K-th}.
\end{proof}

Note that the last two terms in (\ref{chN}) are the sums of Chern characters of line bundles, so they contribute the multiplicative factor 
$$\frac{exp(\rho_u^\times(z^\p,z^{\pp}))}{w_u^\times(z^\p,z^{\pp})}$$ to the equivariant class $E(N_{Z^0})^{-1}$;
and using Lemma \ref{K-th} with  $S=Hom(\widetilde{U}^{\p},\widetilde{U}^{\pp})$, we obtain that the inverse of the K-theoretical Euler  class of the first term in (\ref{chN}) is 
$$(-1)^{lr+r^\p r^{\pp}(g-1)}w_u^\times(z^\p,z^{\pp})^{2-2g}exp(-ch_2(Hom(\widetilde{U}^{\p},\widetilde{U}^{\pp}))_{(2)}).$$
Note that
\begin{multline*}
 -ch_2(Hom(\widetilde{U}^{\p},\widetilde{U}^{\pp}))_{(2)}=\frac{1}{2}c_2(\mathrm{End}_0(\widetilde{U}^\p\oplus\widetilde{U}^{\pp}))_{(2)} \\  -\frac{1}{2}c_2(\mathrm{End}_0(\widetilde{U}^\p))_{(2)} -\frac{1}{2}c_2(\mathrm{End}_0(\widetilde{U}^{\pp}))_{(2)}\\  =c_1\left(\mathcal{L}^r_0\big{|}_{Z^0}\otimes\mathcal{L}_l(-r^\p;-l,...,-l)\boxtimes\mathcal{L}_{-l}(-r^{\pp};l,...,l)\right).
\end{multline*}
The latter equality follows from Lemma \ref{devisiblelevel}. Finally, using Lemma \ref{restriction} to calculate the Chern character of $\mathcal{L}^r_0\big{|}_{Z^0}$, we obtain the formula for the  class  $E(N_{Z^0})^{-1}$, and the proof of the Lemma is complete.
\end{proof}

\end{subsection}
\begin{subsection}{The wall-crossing formula}\label{S5.3}
Putting Lemma \ref{restriction} and Proposition 
\ref{NormalTodd} together, we obtain the following.
\begin{proposition}\label{intl}
The wall-crossing term  (\ref{integral}) is equal to 
\begin{multline*}
K\res_{u=0}e^{(k\delta-rl)u}\int\limits_{P_l(c^\p)\times P_{-l}(c^{\pp})}\big[
(w_u^{\times}(z^\p,z^{\pp}))^{1-2g} exp({\rho_u^{\times}(z^\p,z^{\pp})})\cdot\\ 
ch(\mathcal{L}_l(k+r^{\pp};
\lambda_1^\p-l,...,\lambda^\p_{r^\p-1}-l,\lambda^\p_{r^\p}-l-k{\delta}+rl) 
\boxtimes \\ \mathcal{L}_{-l}(k+r^\p;\lambda^{\pp}_1+l,...,
\lambda^{\pp}_{r^{\pp}-1}+l,\lambda^{\pp}_{r^{\pp}}+l+k\delta-rl)) 
\mathrm{Todd}(P_{l}(c^\p)\times P_{-l}(c^{\pp}))\big]
 \, du,
\end{multline*}
where $\delta$ is a parameter depending on $\lambda$ and the wall $S_{\Pi,l}$ (cf. Lemma \ref{restriction}) and $K$ is the constant $(-1)^{lr+r^\p r^{\pp}(g-1)}\frac{(r(k+r))^g}{(r^\p r^{\pp})^g}$.
\end{proposition}

Now all that is left to do is to perform the integral, 
using an induction on the rank based on Corollary 
\ref{cor:thmwithk}. We will begin with the case $ l=0 $, as 
it is simpler. For $l=0$, the integral from Proposition \ref{intl} has the form 
\begin{multline}\label{intl=0}
 \int\limits_{P_0(c^\p)\times P_{0}(c^{\pp})}\big[
w_u^{\times}(z^\p,z^{\pp})^{1-2g} e^{\rho_u^{\times}(z^\p,z^{\pp})} \mathrm{Todd}(P_0(c^\p))\mathrm{Todd}(P_0(c^{\pp})) \\ ch(\mathcal{L}_0(k+r^{\pp};\lambda_1^\p,...,\lambda^\p_{r^\p-1},\lambda^\p_{r^\p}-k{\delta}) \boxtimes\mathcal{L}_{0}(k+r^\p;\lambda^{\pp}_1,...,\lambda^{\pp}_{r^{\pp}-1},\lambda^{\pp}_{r^{\pp}}+k{\delta}))  \big].
\end{multline}
The inductive hypothesis \eqref{eqnpoly} maybe cast in the 
following form 
\begin{multline}\label{Corpoly1}
	\int\limits_{P_0(c)}
	ch(\L_0(k;\lambda))\mathrm{Todd}(P_0(c))
	= \\ \tilde{N}_{r,k} \sum_{\bb\in\DD}
	\iber_{\bb} [ \exp\scp\lambda{x/\kk}\cdot 
	w_\Phi(x/\kk)^{1-2g}] 
	(\rho/\kk- 
	[c]_\bb).
\end{multline}

Now let us fix $ k $, and allow to vary $ \lambda $. We can 
extend this equality by linearity to arbitrary linear 
combinations of Chern characters of line bundles of the 
form $$ \sum_i ch(\L_0(k;\lambda^i)) = ch(\L_0(k;0))\cdot\sum_ich(\L_0(0;\lambda^i)). $$ Since any polynomial on 
$ V $, up to a fixed degree may be represented as a linear 
combination of exponential functions of the form $ 
\exp\scp\lambda{x/\kk} $, formula \eqref{Corpoly1} may be 
generalized in the following way.

\begin{lemma}\label{Wint} Let $ G(x) $ be a formal power series on $ V 
$, and denote by $ G(z) $ the characteristic class in $ 
H^*(P_0(c)) $ obtained by the identification of functions on $V$ and cohomology classes of $P_0(c)$, 
described before the equation (\ref{identification}). Then we have
	\begin{multline}\label{Corpoly2}
		\int\limits_{P_0(c)}
		ch(\L_0(k;0))G(z)\mathrm{Todd}(P_0(c))
		= \tilde{N}_{r,k}\cdot\\ \sum_{\bb\in\DD}
		\iber_{\bb} [ G(x/\kk)\cdot 
		w_\Phi^{1-2g}(x/\kk)] 
		(\rho/\kk- 
		[c]_\bb).
	\end{multline}	
\end{lemma}

Finally, let $\mathcal{D}^\p$ and $\mathcal{D}^{\pp}$ be Hamiltonian bases (cf. \S\ref{S4.5}). Since $$w_{\Phi^\p}(x/\widehat{k})w_{\Phi^{\pp}}(x/\widehat{k})w^\times_u(x/\widehat{k}) = w_{\Phi}(x/\widehat{k}),$$ $$\rho^\p(x/\widehat{k})\rho^{\pp}(x/\widehat{k})\rho^\times_u(x/\widehat{k}) = \rho(x/\widehat{k}),$$ 
where $w_{\Phi^\p}, w_{\Phi^{\pp}}$ and $\rho^\p, \rho^{\pp}$ 
are naturally defined for the root systems $\Phi^\p$ and $\Phi^{\pp}$ (cf. \S\ref{S4.5}),  the 
 integral (\ref{intl=0}) is equal to 
\begin{multline*}
\tilde{N}_{r^\p,k+r^{\pp}}\tilde{N}_{r^{\pp},k+r^\p} 
\sum_{\textbf{B}^\p\in\mathcal{D}^\p}
\sum_{\textbf{B}^{\pp}\in\mathcal{D}^{\pp}}
\iber_{\textbf{B}^\p}\iber_{\textbf{B}^{\pp}} 
[w_\Phi(x/\widehat{k})^{1-2g}e^{\rho(x/\widehat{k})}] \\
 ((\lambda_1^\p,...,\lambda^\p_{r^\p-1},
\lambda^\p_{r^\p}-k{\delta})/\widehat{k}-[c^\p]_{\textbf{B}^\p}  
+(\lambda^{\pp}_1,...,\lambda^{\pp}_{r^{\pp}-1},\lambda^{\pp}_{r^{\pp}}+k\delta)/\widehat{k}-[c^{\pp}]_{\textbf{B}^{\pp}}).
\end{multline*}
Identifying $u$ (cf. \eqref{identification}) with the "link" element  of the diagonal basis $\DD=(\alpha^{\phi(r^\p),r}\, \DD^\p \, \DD^{\pp})$ (cf. \S\ref{S4.5}), and
 moving the factor $e^{k\delta u}$ from Proposition \ref{intl} inside the argument of $\iber$,
 we obtain the proof of the following theorem for $l=0$.

\begin{theorem}\label{wallcrossint} 
	Let $c^\pm\in\Delta$ be 
in the neighbouring chambers; then the wall-crossing term 
$$\chi(P_0(c^+),\mathcal{L}_0(k;\lambda))-
\chi(P_0(c^-),\mathcal{L}_0(k;\lambda))$$
 is equal to 
	\begin{equation*}
		\begin{split}
			&(k+r)\tilde{N}_{r,k}
			\sum_{\textbf{B}^\p\in\mathcal{D}^\p}\sum_{\textbf{B}^{\pp}\in\mathcal{D}^{\pp}}\res_{\alpha^{\phi(r^\p),r}
			=0}\iber_{\textbf{B}^\p}\iber_{\textbf{B}^{\pp}}[w_\Phi(x/\widehat{k})^{1-2g}](\widehat{\lambda}/\widehat{k}-[c^+]_{\textbf{B}})
			 \, d\alpha^{\phi(r^\p),r},
		\end{split}
	\end{equation*}
	where $\mathcal{D}^{\p}$ and $\mathcal{D}^{\pp}$ are 
	the diagonal bases of $\Phi^\p$ and $\Phi^{\pp}$  (cf. \S\ref{S4.5})
	correspondingly.
\end{theorem}
\begin{remark}\label{sameWall-crossing}
	Note that this wall-crossing term coincides with the 
	one from Proposition \ref{wcrnice}.
\end{remark}

\begin{example}\label{exint} It follows from Example \ref{exwall} that in case of rank 3, the permutation  $\phi\in\Sigma_3$ sends $(1,2,3)$ to $(1,3,2)$. Then $u=c_1(\mathcal{F}_1^\p\otimes{\mathcal{F}^{\pp}_1}^*)$
and let $z=z^{\pp}_1-z^{\pp}_2=c_1(\mathcal{F}_2^{\pp}/\mathcal{F}_1^{\pp}\otimes{\mathcal{F}_1^{\pp}}^*)$. Then the inverse of the K-theoretical Euler class of the conormal bundle is (cf. Proposition \ref{NormalTodd})
$$ch(\mathcal{L})e^{\frac{9\eta}{2}}e^{\frac{z}{2}}\left(2\mathrm{sinh}\left(\frac{u}{2}\right)2\mathrm{sinh}\left(\frac{z-u}{2}\right)\right)^{1-2g},$$
where $\mathcal{L}=\mathcal{L}_0(2;0,0)$ is a line bundle on the moduli space $P_0$ of rank-2 degree-0 stable parabolic bundles.
The Chern character of the restriction of the line bundle $\mathcal{L}_0(k; \lambda_1,\lambda_2,\lambda_3)$ to $\Sigma$ is 
$$e^{\frac{3k\eta}{2}}ch(\mathcal{L}_0^k)e^{\lambda_1z+\lambda_2u}.$$ Hence the wall-crossing term
$$\chi(P_0(<),\mathcal{L}_0(k,\lambda))-\chi(P_0(>),\mathcal{L}_0(k,\lambda))$$ is equal to 
\begin{multline*}
-\left(\frac{3(k+3)}{2}\right)^g\res_{u=0}\frac{e^{\lambda_2u}}{(2\mathrm{sinh}(\frac{u}{2}))^{2g-1}} \cdot  \int\limits_{P_0} \frac{ch(\mathcal{L}_0(k+1;\lambda_1+\frac{1}{2}, -\lambda_1-\frac{1}{2}))}{(2\mathrm{sinh}(\frac{z-u}{2}))^{2g-1}} \mathrm{Todd}(P_0) du.
\end{multline*}
The integral is the Euler charactersitics of a line bundle on a moduli space of degree-$0$ rank-$2$ stable parabolic bundles, so we can calculate it using the induction by rank. It is equal to 
$$(-1)^{g-1}(2(k+3))^g\res_{z=0}\frac{e^{(\lambda_1+1)z}}{(2\mathrm{sinh}(\frac{z-u}{2})2\mathrm{sinh}(\frac{z}{2}))^{2g-1}(1-e^{(k+3)z})}dz, $$
so the wall-crossing term is 
$$(-3(k+3)^2)^g\res_{u=0}\res_{z=0}\frac{e^{\lambda_1z+\lambda_2u+z}}{\tilde{w}_\phi(z,u)^{2g-1}(1-e^{(k+3)z})}dzdu,$$
where $\tilde{w}_\phi(z,u)=2\mathrm{sinh}(\frac{z-u}{2})2\mathrm{sinh}(\frac{u}{2})2\mathrm{sinh}(\frac{z}{2})$. Note that this is exactly the same polynomial as in Example \ref{exdiff} after changing $(z,u)$ to $(x,-y)$.
\end{example}
\end{subsection}

\section{Tautological Hecke correspondences}\label{sec:hecke}
If $l\neq 0$, then we need one more step in our proof, which uses the Hecke correspondence to calculate the wall-crossing term (\ref{integral}).
\subsection{The Hecke correspondence}
Given a rank-$r$ degree-$d$ vector bundle $W$ with a full flag $0\subsetneq F_1\subsetneq ... \subsetneq F_r=W_p$ at  $p$, one can obtain a rank-$r$ degree-$d-1$ vector bundle $W^\p$ with a full flag $0\subsetneq G_1\subsetneq ... \subsetneq G_r=W'_p$ using the tautological Hecke correspondence  construction as follows. 

The evaluation map  $W\to W_p$ induces the short exact sequence of the associated sheaves of sections
\begin{equation}\label{alphahecke}
0 \to \mathcal{W}' \overset{\tilde{\alpha}}\to \mathcal{W}\to W_p/F_{r-1} \to 0
\end{equation} 
on curve $C$. Since $\mathcal{W}^\p$ is a kernel of $\tilde{\alpha}$, it is a locally free sheaf, thus gives a rank-$r$ vector bundle $W^\p$ over $C$ with $det(W')\simeq det(W)\otimes\mathcal{O}(-p)$. The image of the associated morphism of vector bundles $\alpha$ at the point $p$  is $F_{r-1}\subset {W}_p$, so $\alpha_p: W^\p_p \to W_p$ has a one-dimensional kernel $G_1\subset W^\p_p$. Moreover, compositions of $\alpha_p$ with the quotient morphisms $F_{r-1}\to F_{r-1}/F_{i}$ induce a full flag of the corresponding kernels $G_1\subsetneq ... \subsetneq G_{r-1}\subsetneq G_r=W'_p$ in $W'_p$.

Denote this operator
between the sets of isomorphism classes of degree-$d$ and $d-1$ vector bundles with a flag at $p$  by 
$$\mathcal{H}: (W, F_*)\mapsto (W',G_*).$$ 
Similarly, for any $m\geq 0$, one can define the operator $\mathcal{H}^{m}$ between the sets of  isomorphism classes  of degree-$d$ and $d-m$ vector bundles with a flag at the point $p$ by iterating the above construction $m$ times. Clearly, these maps are independent of the parabolic weights.  

\begin{proposition}\label{Hecke}
Let $c\in\Delta$ be a regular (cf. page \pageref{reg}) point. Then the operator $\mathcal{H}$ induces an isomorphism between the moduli spaces $P_d(c_1,...,c_r)$ and $P_{d-1}(c_2,...,c_r,c_1-1)$.
\end{proposition}
\begin{proof}
First, we need to show that if $W\in P_d(c_1,...,c_r)$ is a parabolic stable bundle with parabolic weights $(c_1,...,c_r)$, then $W^\p$, its image under the Hecke operator $\mathcal{H}$, is parabolic stable with respect to parabolic weights $(c_2,...,c_r,c_1-1)$. For this, consider the subbundle $V'\subset{W'}$ and let $\alpha(V')=V\subset{W}$ (cf. (\ref{alphahecke})) be its image. Since $W$ is parabolic stable, $$\textit{par}\textit{slope}(V)<\textit{par}\textit{slope}(W)=\textit{par}\textit{slope}(W^\p).$$
We need to prove that $\textit{par}\textit{slope}(V')<\textit{par}\textit{slope}(W')$. There are two possible cases:
\begin{itemize}
\item If $\alpha$ maps $V'$ to $V$ isomorphically, then $deg(V')=deg(V)$ and $V_p\subset{F_{r-1}}$, hence $\textit{par}\textit{slope}(V')=\textit{par}\textit{slope}(V)<\textit{par}\textit{slope}(W^\p)$. 
\item Otherwise, $deg(V')=deg(V)-1$, and $V_p$ is not contained in ${F_{r-1}}$, so one of the parabolic weights of $V'$ is $c_1-1$. Then, as in the previous case, $\textit{par}\textit{slope}(V')=\textit{par}\textit{slope}(V)$, and the result follows.
\end{itemize}
To show that the map $\mathcal{H}$ is an isomorphism, note that $\mathcal{H}^{r}$
maps 
\begin{equation}\label{HeckeR}
P_d(c_1,c_2,...,c_r) \to P_{d-r}(c_1-1,c_2-1,,...,c_r-1).
\end{equation}
It is easy to check that given $\mathcal{W}$ and iterating the associated morphism of locally free sheaves of sections \eqref{alphahecke} $r$ times, we obtain a subsheaf $\mathcal{W}^\p\subset \mathcal{W}$ of sections of $W$ which vanishes at the point $p$.
So the map \eqref{HeckeR} is just tensoring by $\mathcal{O}(-p)$, and hence it is an isomorphism.
\end{proof}
Now we can define an operator $\mathcal{H}^{m}$ for any $m\in\mathbb{Z}$, taking the inverse map if necessary.
We will need the following statement, which follows from Proposition \ref{Hecke} and the construction of $\mathcal{H}^m$.

\begin{corollary}\label{LineHecke}
Let $m\geq 0$. Then under the isomorphism $\mathcal{H}^{m}$
the line bundle $\mathcal{L}_d(k; \lambda_1,...,\lambda_r)$ corresponds to the line bundle $\mathcal{L}_{d-m}(k; \lambda_{r-m+1},...,\lambda_r,\lambda_1-k,...,\lambda_{r-m}-k)$.
\end{corollary}

\subsection{The effect of the Hecke correspondence on the integral} 
Recall that our  goal is to calculate the wall-crossing term from Proposition \ref{intl}. For simplicity, we assume that $l$ is positive (the other case 
	is analogous).
We apply the Hecke operators $\mathcal{H}^l$ and $\mathcal{H}^{-l}$ 
to the moduli spaces $P_l(c')$ and $P_{-l}(c^{\pp})$ to obtain  $$P_0'=P_0(c'_{l+1},...,c'_{r'},c'_1-1,...,c'_l-1)\simeq P_l(c') \text{\, and \,}$$ $$P_0^{\pp}=P_0(c^{\pp}_{r^{\pp}-l+1}+1,...,c^{\pp}_{r^{\pp}}+1,c^{\pp}_1,...,c^{\pp}_{r^{\pp}-l})\simeq P_{-l}(c^{\pp}).$$ 
Recall (cf. page \pageref{gact}) that there is a natural action of the group $\Sigma_r$ on $V^*$, 
and hence (cf. page \pageref{identification}) on $H^2(P_l(c')\times{P_{-l}(c^{\pp})})$. 
Let $\tau^\p\in \Sigma_{r^\p}$ and $\tau^{\pp}\in \Sigma_{r^{\pp}}$ be the cyclic permutations defined by   
$$\tau^\p\cdot(c^\p_1-1,...,c^\p_{l}-1,c^\p_{l+1},...,c^\p_{r^\p})=(c'_{l+1},...,c'_{r'},c'_1-1,...,c'_l-1)$$ and $$\tau^{\pp}\cdot(c^{\pp}_1,...,c^{\pp}_{r^{\pp}-l},c^{\pp}_{r^{\pp}-l+1}+1,...,c^{\pp}_r+1)=(c^{\pp}_{r^{\pp}-l+1}+1,...,c^{\pp}_{r^{\pp}}+1,c^{\pp}_1,...,c^{\pp}_{r^{\pp}-l}).$$ And set $\tau=(\tau^\p,\tau^{\pp})\in\Sigma_{r^\p}\times\Sigma_{r^{\pp}}\subset \Sigma_r$.
 Note that $$\tau^\p\cdot(-l+r^\p,...,-l+r^\p,-l,...,-l)=\tau^\p\cdot\rho^\p-\rho^\p$$ and $$\tau^{\pp}\cdot(l,...,l,l-r^{\pp},...,l-r^{\pp})= \tau^{\pp}\cdot\rho^{\pp}-\rho^{\pp},$$
so applying the Hecke operator $\mathcal{H}^l\times\mathcal{H}^{-l}$ to the wall-crossing term from Proposition \ref{intl} and using Corollary \ref{LineHecke}, we obtain that the wall-crossing term (\ref{integral}) is equal to 
\begin{multline}\label{final}
K\res_{u=0}e^{(k\delta-rl)u}\int_{P_0^\p\times P^{\pp}_0} \big(\tau\cdot w_u^{\times}(z^\p,z^{\pp})^{1-2g}e^{\tau\cdot\rho_u^{\times}(z^\p,z^{\pp})} \\ 
 ch(\mathcal{L}_0(k+r^{\pp};\tau^\p\cdot(\lambda^\p_1-\widehat{k},...,\lambda^\p_l-\widehat{k},\lambda^\p_{l+1},...,\lambda^\p_{r^\p-1},\lambda^\p_{r^\p}-k\delta+rl))  \\ 
 ch(\mathcal{L}_0(k+r^\p;\tau^{\pp}\cdot(\lambda^{\pp}_1,...,\lambda^{\pp}_{r^\p-l},\lambda^{\pp}_{r^{\pp}-l+1}+\widehat{k},...,\lambda^{\pp}_{r^{\pp}}+\widehat{k}+k\delta-rl)) \\
e^{\tau^{\pp}\cdot\rho^{\pp}(z^\p,z^{\pp})-\rho^{\pp}(z^\p,z^{\pp})} e^{\tau^\p\cdot\rho^\p(z^\p,z^{\pp})-\rho^\p(z^\p,z^{\pp})} \mathrm{Todd}(P_0^\p)  \mathrm{Todd}(P_0^{\pp}) \big) \, du.
\end{multline}
As in \S\ref{S5.3}, according to Lemma \ref{Wint}, we can calculate this integral using the induction on rank. 
Let $\mathcal{D}^\p$ and $\mathcal{D}^{\pp}$ be two Hamiltonian diagonal bases. Then $\tau^\p(\mathcal{D}^\p)$ and $\tau^{\pp}(\mathcal{D}^{\pp})$ are also Hamiltonian diagonal bases (cf. Remark \ref{rem:symmdiag}) and the integral in (\ref{final}) is  equal to 
\begin{multline}\label{ibernotshift}
(-1)^{lr}\tilde{N}_{r^\p,k+r^{\pp}}\tilde{N}_{r^{\pp},k+r^\p} \sum_{\textbf{B}^\p\in\tau^\p(\mathcal{D}^\p)}\sum_{\textbf{B}^{\pp}\in\tau^{\pp}(\mathcal{D}^{\pp})} \\ \iber_{\textbf{B}^\p}\iber_{\textbf{B}^{\pp}} [
\tau\cdot w^\times_u(x/\widehat{k})^{1-2g}(w_{\Phi^\p}(x/\widehat{k})w_{\Phi^{\pp}}(x/\widehat{k}))^{1-2g}
e^{\tau\cdot\rho(x/\widehat{k})}] \\ 
(\tau^\p\cdot(\lambda^\p_1-\widehat{k},...,\lambda^\p_l-\widehat{k},\lambda^\p_{l+1},...,\lambda^\p_{r^\p-1},\lambda^\p_{r^\p}-k{\delta}+rl)/\widehat{k}- \\ [\tau^\p\cdot(c^\p_1-1,...,c^\p_l-1,c^\p_{l+1},...,c'_{r'})]_{\textbf{B}^\p}+ \\ \tau^{\pp}\cdot(\lambda^{\pp}_1,...,\lambda^{\pp}_{r^\p-l},\lambda^{\pp}_{r^{\pp}-l+1}+\widehat{k},...,\lambda^{\pp}_{r^{\pp}-1}+\widehat{k},\lambda^{\pp}_{r^{\pp}}+\widehat{k}+k\delta-rl)/\widehat{k}- \\ [\tau^{\pp}\cdot(c_1^{\pp},...,c^{\pp}_{r^\p-l},c^{\pp}_{r^{\pp}-l+1}+1,...,c^{\pp}_{r^{\pp}}+1)]_{\textbf{B}^{\pp}}).  
\end{multline}
To arrive at Theorem \ref{wallcrossint}, we need to make additional transformations of formula \eqref{ibernotshift}: first, we shift $\lambda^\p$ and $\lambda^{\pp}$
, and then we apply Lemma \ref{lem:symres} to eliminate the cyclic permutation $\tau$.

Note that given an ordered basis $\textbf{B}\in\mathcal{B}$ and an element $v\in V^*$ such that $\{v\}_{\textbf{B}}=0$, for any weight $\lambda\in{\Lambda}$ and positive integer $k$ one have 
\begin{equation}\label{trivial1}
({\lambda}+\widehat{k}v)/\widehat{k}-[c+v]_{\textbf{B}}= {\lambda}/\widehat{k}-[c]_{\textbf{B}}.
\end{equation}
In particular, to perform the shift of $\lambda^\p$ in \eqref{ibernotshift}, we use the following equality for any $\mathbf{B}^\p\in\mathcal{D}^\p$ :
\begin{multline}\label{shift1}
(\lambda^\p_1-\widehat{k},...,\lambda^\p_l-\widehat{k},\lambda^\p_{l+1},...,\lambda^\p_{r^\p-1},\lambda^\p_{r^\p}-k{\delta}+rl)/\widehat{k} - \\ [(c^\p_1,...,c^\p_{r^\p-1},c_{r^\p}^\p-l)-(1,...,1,0,...0,-l)]_{\mathbf{B}^\p} = \\ (\lambda^\p_1,...,\lambda^\p_{r^\p-1},\lambda^\p_{r^\p}-k{\delta}+rl-l\widehat{k})/\widehat{k}-[(c^\p_1,...,c^\p_{r^\p-1},c^\p_{r^\p}-l)]_{\mathbf{B}^\p},
\end{multline}
which clearly remains true after changing $\mathcal{D}^\p$ to $\tau^\p(\mathcal{D}^\p)$ and applying $\tau^\p$ to both sides of the equation. Similarly, shifting the last terms of (\ref{ibernotshift}) by $\tau^{\pp}(0,...,0,-1,...-1,-1+l)$, 
 we can rewrite (\ref{ibernotshift}) as 
\begin{multline*}
(-1)^{lr}\tilde{N}_{r', k+r^{\pp}}\tilde{N}_{r^{\pp},k+r^\p}\sum_{\textbf{B}^\p\in\tau^\p(\mathcal{D}^\p)}\sum_{\textbf{B}^{\pp}\in\tau^{\pp}(\mathcal{D}^{\pp})}\iber_{\textbf{B}^\p}\iber_{\textbf{B}^{\pp}} \\ [
\tau\cdot w^\times_u(x/\widehat{k})^{1-2g}(w_{\Phi^\p}(x/\widehat{k})w_{\Phi^{\pp}}(x/\widehat{k}))^{1-2g}
e^{\tau\cdot\rho(x/\widehat{k})}] \\ 
 (\tau^\p\cdot(\lambda^\p_1,...,\lambda^\p_{r^\p-1},\lambda^\p_{r^\p}-k{\delta}+rl-l\widehat{k})/\widehat{k}-[\tau^\p\cdot(c^\p_1,...,c'_{r'-1},c'_{r'}-l)]_{\textbf{B}^\p}+ \\ \tau^{\pp}\cdot(\lambda^{\pp}_1,...,\lambda^{\pp}_{r^{\pp}-1},\lambda^{\pp}_{r^{\pp}}+k\delta-rl+l\widehat{k})/\widehat{k}-[\tau^{\pp}\cdot(c^{\pp}_1,...,c^{\pp}_{r^{\pp}-1},c^{\pp}_{r"}+l)]_{\textbf{B}^{\pp}}). 
\end{multline*}
Finally, identifying $u$ (cf. \eqref{identification}) with the "link" element of the diagonal basis $\tau(\DD)=(\alpha^{\tau\phi(r^\p), \tau(r)}\, \tau^\p(\DD^\p) \, \tau^{\pp}(\DD^{\pp}))$ (cf. \S\ref{S4.5}) and 
\begin{itemize}
\item moving the factor $e^{(k\delta-rl)u}$ from (\ref{final}) inside the argument of $\iber_{\textbf{B}}$, where $\textbf{B}=(\alpha^{\tau\phi(r^\p), \tau(r)}\, {\textbf{B}}^\p \,  {\textbf{B}}^{\pp})$,
\item applying (\ref{trivial1}) with $\textbf{B}=(\alpha^{\tau\phi(r^\p), \tau(r)}\, {\textbf{B}}^\p \,  {\textbf{B}}^{\pp})$ and $v= l\alpha^{\tau\phi(r^\p), \tau(r)}$,
\item applying Lemma \ref{lem:symres},
\item and using the fact that $$\tau^{-1}\cdot (w_{\Phi^\p}(x/\widehat{k})w_{\Phi^{\pp}}(x/\widehat{k}))=(-1)^{lr}w_{\Phi^\p}(x/\widehat{k})w_{\Phi^{\pp}}(x/\widehat{k}),$$
\end{itemize}
 we obtain the formula of Theorem \ref{wallcrossint} for arbitrary $l\in\mathbb{Z}$.
\end{section}

\begin{section}{Affine Weyl symmetry and the proof of  part I of Theorem 
\ref{main}} \label{sec:weylsymm}

In this section, we prove certain symmetry properties of our Hilbert polynomials on the left hand side of \eqref{eq:ver0}, and we finish the proof of part I of Theorem \ref{main}. We start with the basic instance of 
symmetry of Hilbert polynomials: relative Serre duality.

\subsection{Serre duality}
\begin{proposition}\label{Serrebasic}
Let $\mathcal{E}\to{X}$ be a rank 2 vector bundle over a smooth variety $X$, $\pi: Y=\mathbb{P}(\mathcal{E})\to X$ its projectivization and $\omega_{X/Y}$ the relative cotangent line bundle. Then $$\chi(Y,\pi^*\mathcal{L}\otimes\omega_{X/Y}^m)=-\chi(Y,\pi^*\mathcal{L}\otimes\omega_{X/Y}^{-m+1})$$ for any line bundle $\mathcal{L}\in Pic(X)$.
\end{proposition}
\begin{proof}
It follows from the short exact sequence 
$$0\to \Omega_{X/Y}\otimes\mathcal{O}_X\to \pi^*\mathcal{E}(-1)\to \mathcal{O}_X\to 0,$$ 
that  $$\omega_{X/Y}=\wedge^2(\pi^*\mathcal{E}(-1))=\pi^*(\wedge^2\mathcal{E})\otimes\mathcal{O}(-2).$$ By Serre duality for families of curves \cite[Ch III, \S7-8]{Hartshorne} we have then
\begin{multline*}
\chi(Y, \pi^*(\mathcal{L}\otimes(\wedge^2\mathcal{E})^m)\otimes\mathcal{O}(-2m))= \\ -\chi(Y, \pi^*(\mathcal{L}\otimes(\wedge^2\mathcal{E})^m)\otimes\pi^*(\wedge^2\mathcal{E})^{-2m+1}\otimes\mathcal{O}(2m-2)).
\end{multline*}
\end{proof}

Now we can iterate this statement to the case of flag bundles.
\begin{proposition}\label{Serre}
Let $\pi: Y =\mathrm{Flag}(\mathcal{E})\to X$ be a rank $r$ flag bundle over $X$. Let $\mathcal{L}$ be a line bundle on $X$, and $\mathcal{F}_1$, $\mathcal{F}_2/\mathcal{F}_1, ..., \mathcal{F}_r/\mathcal{F}_{r-1}$ the standard flag line bundles on $Y$. For $k\in\mathbb{Z}$ and $\lambda=(\lambda_1,...,\lambda_r)\in\Lambda$ denote by 
$$\mathcal{L}{(k;\lambda)}=(\pi^*\mathcal{L})^k\otimes(\mathcal{F}_r/\mathcal{F}_{r-1})^{\lambda_1}
\otimes(\mathcal{F}_{r-1}/\mathcal{F}_{r-2})^{\lambda_{2}}\otimes...\otimes\mathcal{F}_1^{\lambda_r}.$$
Consider the polynomial $$q(k;\lambda_1,\lambda_2, ...,\lambda_r)=\chi(Y, \mathcal{L}{(k;\lambda_1,\lambda_2, ..., \lambda_r)})$$ in $(k,\lambda)\in \mathbb{Z}\times\Lambda$ and extend this definition to $\mathbb{R}\times V^*$. Then $q(k; \lambda-\rho)$ is anti-invariant under the permutations of $\lambda_1, \lambda_2, ...,\lambda_r$.
\end{proposition}

\begin{subsection}{The Weyl anti-symmetry of the functions $ q_1 $ 
and $ q_{-1} $}\label{S7.2}
Armed with this statement, we are ready to take on the symmetries of the 
Hilbert polynomial of our parabolic moduli spaces. We note that the two sets $ 
\Delta_{\pm1} $ of weights for degree-${\pm1}$ stable parabolic 
bundles are simplices with one of their vertices at 
$(\frac{1}{r},...,\frac{1}{r})$ and $(\frac{-1}{r},...,\frac{-1}{r})$, 
correspondingly (cf. \S\ref{S2.2}).  
 
 Denote by $N_{\pm1}$ the moduli spaces of rank-$r$ degree-$\pm1$  stable 
 vector bundles and by $UN$ any universal bundle over $N_{\pm1}\times C$ (cf. 
 e.g \cite{AtiyahBott}). 
\begin{lemma}\label{flagbundle}
Let $ c=(c_1,...,c_r)$ be a parabolic weight from the chamber in $\Delta_1$, 
which has as one of its vertices the (regular) point 
$(\frac{1}{r},...,\frac{1}{r})$. Then the  moduli space $P_{1}(c)$ of rank-$r$ 
degree-$1$ stable parabolic bundles is isomorphic to the flag bundle 
$\mathrm{Flag}({UN}_p)$ over $N_{1}$.
 An analogous statement holds in the case of degree $ -1$ and the point 
 $(\frac{-1}{r},...,\frac{-1}{r})\in\Delta_{-1}$.
\end{lemma}
\begin{proof} 
A simple calculation shows that the point $(c_1,...,c_r)\in \Delta_1$, such 
that all $c_i>0$, lies inside the chamber in $\Delta_1$ with the vertex 
$(\frac{1}{r},...,\frac{1}{r})$. Hence it is enough to prove the first 
statement for the moduli space $P_{1}(c_1,...,c_r)$ with positive parabolic 
weights.

Moreover, it is sufficient to show that if $(W,F_*)$ is a parabolic stable 
vector 
bundle which represents a point in $P_{1}(c_1,...,c_r)$, then $W$ is stable as 
an ordinary bundle. Assume that $W$ admits a proper subbundle $W^\p$ with 
$\textit{slope}(W^\p)\geq\textit{slope}(W)=\frac{1}{r}$, then 
$\mathrm{deg}(W^\p)\geq1$. Since all parabolic weights of $W$ are positive, 
this implies that $\textit{parslope}(W^\p)>0=\textit{parslope}(W)$, and 
therefore $W$ is parabolic unstable. The proof for degree-($-1$) bundles is 
analogous. 
\end{proof}

Denote the moduli spaces described above by $P_{1}(>)$ and $P_{-1}(<)$, 
correspondingly, and their images 
under the Hecke isomorphisms $\mathcal{H}$ and $\mathcal{H}^{-1}$ by 
$P_{0}(>)$ and $P_0(<)$.

The following statement is straightforward (cf. Lemma \ref{devisiblelevel}).
\begin{lemma}\label{PicN}
The line bundles $\mathcal{L}_{1}(r;1,...,1)$ and 
$\mathcal{L}_{-1}(r;-1,...,-1)$ on $P_{1}(>)$ and $P_{-1}(<)$  defined in Lemma 
\ref{parpic} may be obtained as pullbacks of the ample generators of the Picard 
groups $\mathrm{Pic}(N_{\pm1})$.
\end{lemma}
\begin{example}\label{exsymm}
In case of rank-3 parabolic bundles the moduli space $P_1(c_1,c_2,c_3)$ with $2c_3>c_1+c_2-1$ is a flag bundle over $N_1$ and it is isomorphic to the moduli space $P_0(>)$ from Example \ref{exwall}, while the moduli space $P_{-1}(c_1,c_2,c_3)$ with $2c_1<c_2+c_3+1$ is a flag bundle over $N_{-1}$ and it is isomorphic to $P_0(<)$. 
\end{example}
Now we establish the Weyl anti-symmetry of the polynomials
$$q_{-1}(k;\lambda_1,..., \lambda_r)= 
\chi(P_{0}(<),\mathcal{L}_0(k;\lambda_1,..., \lambda_r))$$ and 
$$q_1(k;\lambda_1,..., \lambda_r)= \chi(P_{0}(>),\mathcal{L}_0(k;\lambda_1,..., 
\lambda_r))$$ defined on $\mathbb{R}\times \Lambda$, as in Proposition 
\ref{Serre}. Let $\tau\in\Sigma_r$ be the cyclic permutation, such that 
$\tau\cdot(c_1,...,c_r)=(c_2,...,c_r,c_1)$, and consider two points in $V^*$: 
\begin{multline*}
	  \theta_1[k] = 
\frac{k+r}{r}\cdot(1,1,\dots,1)-(k+r)x_r-\rho= \tau\cdot(\frac{k}{r}-k,\frac{k}{r},...,\frac{k}{r}) -\tau\cdot(\rho) = \\
\left(\frac{k}{r}-\frac{r-1}{2}+1, 
\frac{k}{r}-\frac{r-1}{2}+2,...., 
\frac{k}{r}-\frac{r-1}{2}+r-1, 
-k+\frac{k}{r}-\frac{r-1}{2}\right)
\end{multline*}
and
\begin{multline*}
	\theta_{-1}[k] = 
-\frac{k+r}{r}\cdot(1,1,\dots,1)+(k+r)x_1-\rho = \\ \tau^{-1}\cdot(-\frac{k}{r},...,-\frac{k}{r},-\frac{k}{r}+k) -\tau^{-1}\cdot(\rho) = \\
\left(k-\frac{k}{r}+\frac{r-1}{2},-\frac{k}{r}-
\frac{r-1}{2}, 
-\frac{k}{r}-\frac{r-1}{2}+1,...,-\frac{k}{r}-\frac{r-1}{2}+
r-2\right).
\end{multline*} 

\begin{proposition}\label{propsymm}
The polynomials
$q_1(k;\lambda+\theta_1[k])$ and $q_{-1}(k; \lambda+\theta_{-1}[k])$
are anti-invariant under the action of the Weyl group by permutations of $\lambda_1,..., \lambda_r$.
\end{proposition}
\begin{proof}
Recall that the moduli space $P_0(>)$ is isomorphic to the flag bundle $P_1(>)$ over $N_1$ under the Hecke isomorphism $\mathcal{H}^{-1}$. Then using Corollary \ref{LineHecke},  Proposition \ref{Serre} and Lemma \ref{PicN}, for any permutation $\sigma\in\Sigma_r$ we obtain
\begin{multline*}
q_1(k;\sigma\cdot\lambda+\theta_1[k])\overset{\mathrm{def}}= \chi(P_{0}(>),\mathcal{L}_0(k;\sigma\cdot\lambda+\theta_1[k])) \overset{{\ref{LineHecke}}}= \\  \chi(P_1(>), \mathcal{L}_1(k;\tau^{-1}\cdot\sigma\cdot\lambda+(\frac{k}{r},...,\frac{k}{r})-\rho))\overset{{\ref{Serre}},\ref{PicN}}= \\ (-1)^{\sigma}\chi(P_1(>), \mathcal{L}_1(k;\tau^{-1}\cdot\lambda+(\frac{k}{r},...,\frac{k}{r})-\rho))  \overset{{\ref{LineHecke}}}=  \\ (-1)^{\sigma}\chi(P_{0}(>),\mathcal{L}_0(k;\lambda+\theta_1[k])) \overset{\mathrm{def}}=(-1)^{\sigma}q_1(k;\lambda+\theta_1[k]).
\end{multline*}
The proof for $q_{-1}$ is similar.
\end{proof}

The two group actions in Proposition \ref{propsymm} may be combined in the following manner.
For $ k\ge0 $, we define an action of the \textit{affine 
Weyl group} $ \Sigma\rtimes \Lambda $ on $ \Lambda\times\Z $, 
which acts trivially on the second factor, the level, 
and the action at level $ k $ is given by setting 
$$ \sigma.\lambda=\sigma\cdot(\lambda+\rho)-\rho \text{\,\,\, 
and \,\,\,}  \gamma.\lambda=\lambda+(k+r)\gamma\text{ 
for }
\sigma\in\Sigma,\,\gamma\in\Lambda. $$ We denote the 
resulting group of affine-linear transformations of $ V^* $ 
by $ \affweyl $, and note that the action is defined in such 
a way that 
\begin{equation}\label{sigmadot}
	\sigma.\lambda+\rho=\sigma\cdot(\lambda+\rho)
	\text{ and 
	}(\gamma.\lambda+\rho)/\kk=\gamma+(\lambda+\rho)/\kk
\end{equation}

It is easy to verify that the stabilizer subgroup
\[ 
\Sigma_r^+\overset{\mathrm{def}}=\mathrm{Stab}(\theta_{1}[k],
\affweyl)\subset\affweyl \]
is generated by the   
transpositions $s_{i,i+1},\;1\leq i\leq r-2$ and the reflection $\alpha^{r-1,r}\circ s_{r-1,r};$ 

similarly,
$$ 
\Sigma_r^-\overset{\mathrm{def}}=\mathrm{Stab}(\theta_{-1}[k],
\affweyl)\subset\affweyl $$ is generated by $s_{i,i+1}$, $2\leq 
i\leq r-1$ and the reflection $ 
\alpha^{1,2}\circ s_{1,2} $.

Then Proposition \ref{propsymm} maybe recast in the 
following form: the polynomial $q_1(k;\lambda)$ is 
anti-invariant with respect to the copy $ 
\Sigma_r^+ $ of the symmetric group $ \Sigma_r $, while 
$q_{-1}(k;\lambda)$ is anti-invariant with respect to the 
copy $ \Sigma_r^- $ of the symmetric group $ \Sigma_r $.

The following statement is straightforward:
\begin{lemma}\label{lem:generate}
Both subgroups $ \Sigma_r^\pm$ are isomorphic to $\Sigma_r$ and for $r>2$, the two subgroups generate the affine Weyl 
group $ \affweyl $.
\end{lemma}
\end{subsection}
\begin{subsection}{The Weyl anti-symmetry of the polynomials $ p_1 
$ and $ p_{-1} $}\label{S7.3}

Following \eqref{eqnpoly}, we define the two polynomials
\[ p_{\pm1}(k;\lambda)=
\sum_{\bb\in\DD}
\iber_{\bb} [ w_\Phi^{1-2g}(x/\kk)] (\lala/\kk- 
[\theta_{\pm1}]_\bb),
 \]
where $ \theta_1=\frac1r\cdot(1,1,\dots,1)-x_r $, and
$ \theta_{-1}=-\frac1r\cdot(1,1,\dots,1)+x_1 $.

\begin{proposition}\label{propsymmIBer}
The polynomial $ p_1(k;\lambda) $ is anti-invariant with 
respect to $ \Sigma_r^+ $, and $ p_{-1}(k;\lambda) $ is 
anti-invariant with respect to $ \Sigma_r^- $.
\end{proposition}
\begin{proof}
We recall that the points $ \theta_{\pm1}[k] $ are the 
fixed points of the actions of $ \Sigma^{\pm} $, and 
clearly $ \lim_{k\to\infty} \theta_{\pm1}[k]/k = 
\theta_{\pm1} $. This means, that we can fix a small open 
ball $ D \subset V^*$ centered at $ 
\theta_1 $ such that
\begin{equation}\label{eq:smallball}
	\lambda/k\in D\Longrightarrow 
	\forall\sigma\in\Sigma^+:\;
	(\sigma.\lambda+\rho)/\kk\sim\theta_1.
\end{equation}

Then for $ \lambda/k\in D $ we have
\[ p_{1}(k;\lambda)=
\sum_{\bb\in\DD}
\iber_{\bb} [ w_\Phi^{1-2g}(x/\kk)] (\{\lala/\kk\}_\bb).
\]
Now, let us consider a generator of $ \Sigma^+ 
$ of the type   $ \sigma=s_{i,i+1} $, $ 1\leq 
i\leq r-2 $. Using \eqref{sigmadot}, and Lemma
\ref{lem:symres}, and the fact that $\sigma\cdot w_\Phi=-w_\Phi$ we obtain
\begin{multline*}
	 p_{1}(k;\sigma.\lambda)=
\sum_{\bb\in\DD}
\iber_{\bb} [ w_\Phi^{1-2g}(x/\kk)] (\sigma\cdot\{\lala/\kk\}_\bb)=
\\
\sum_{\bb\in\DD}
\iber_{\bb} [(-w_\phi)^{1-2g}(x/\kk)] 
(\{\lala/\kk\}_\bb)= 
-p_{1}(k;\lambda)
\end{multline*}

The case of the last generator $ 
\alpha^{r-1,r}\circ s_{r-1,r} $ is similar, but 
after the substitution, we  need to use the equality
$ \{\alpha^{r-1,r}+\lala/\kk\}_\bb =\{\lala/\kk\}_\bb$ to 
obtain $p_1(k;\kk\alpha^{r-1,r}+s_{r-1,r}.\lambda)=
-p_{1}(k;\lambda) $.
\end{proof}
\end{subsection}

\subsection{Proof of part I. of Theorem \ref{main}}\label{S8.4} 
 Recall that in Lemma \ref{lemmaBwall}, we introduced a chamber structure on $\Delta\subset 
V^*$ created by the walls $S_{\Pi,l}$, where $\Pi=(\Pi^\p,\Pi^{\pp})$ is a
 nontrivial partition, and $ l\in\Z $.  
Before we proceed, we introduce some extra notation. Denote by 
$$\widecheck{\Delta} = \{(k;a)| \, a/k\in\Delta\} \subset \mathbb{R}^{>0}\times V^*$$
 the cone over $\Delta\subset V^*$, and let  
 $$\widecheck{\Delta}^{\mathrm{reg}} = \{(k;a)| \, a/k\in\Delta \, \text{is 
 regular}\}\subset \widecheck{\Delta}$$ be the set of its regular points. 
 Denote by $\widecheck{S}_{\Pi,l}\subset\widecheck{\Delta}$  the cone over the 
 wall $S_{\Pi,l}\subset \Delta$; then $\widecheck{\Delta}^{\mathrm{reg}}$ is 
 the complement of the union of walls $\widecheck{S}_{\Pi,l}$ in 
 $\widecheck{\Delta}$. Finally,  denote by  
 $\widecheck{\Delta}^{\mathrm{reg}}_\Lambda$ the intersection of the lattice 
 $\mathbb{Z}^{>0}\times\Lambda$ with $\widecheck{\Delta}^{\mathrm{reg}}$.

By substituting $\textit{\k c} =\lambda/k$, we can consider the left-hand side and the right-hand side of formula I. of Theorem 
\ref{main} as functions in $(k,\lambda)\in\widecheck{\Delta}^{\mathrm{reg}}_\Lambda$. We denote by $q(k;\lambda)$ and $p(k;\lambda)$  the left-hand side and the right-hand side, correspondingly. 

We showed that $q(k;\lambda)$ and $p(k;\lambda)$  are \textit{polynomials} on the cone over each chamber in $\Delta$ (cf. Theorem \ref{diaginv}, \S\ref{S2.4}). 
We proved that the wall-crossing terms, i.e. the differences between polynomials on neighbouring chambers, for $q(k;\lambda)$ (cf. Theorem \ref{wallcrossint})  and for $p(k;\lambda)$ (cf. Proposition \ref{wcrnice}) coincide, hence there exists a polynomial $\Theta(k;\lambda)$ on $\mathbb{Z}^{>0}\times\Lambda$, such that the restriction of $\Theta(k;\lambda)$ to $\widecheck{\Delta}^{\mathrm{reg}}_\Lambda$ is equal to  the difference $p(k;\lambda)-q(k;\lambda)$. 

Now for $r>2$, we can conclude that 
$$\Theta(k;\lambda)=p_1(k;\lambda)-q_1(k;\lambda)=p_{-1}(k;\lambda)-q_{-1}(k;\lambda),$$
 where $p_{\pm1}(k;\lambda)$ and $q_{\pm1}(k;\lambda)$  are the restrictions of $p(k;\lambda)$ and $q(k;\lambda)$  to two specific chambers defined in  \S\ref{S7.3} and \S\ref{S7.2}. Then, according to Propositions \ref{propsymm} 
and \ref{propsymmIBer}, the polynomial $\Theta(k;\lambda)$ is 
anti-invariant with respect to the action of the subgroups $\Sigma^\pm_r$, and hence by 
Lemma \ref{lem:generate}, it is anti-invariant under the action of the entire 
affine Weyl group $\affweyl$. It is easy to see that any such polynomial 
function has to vanish, and thus $p(k;\lambda)=q(k;\lambda)$, and this 
completes the proof of part I. of Theorem \ref{main} for the case when $\lambda/k\in\Delta$ is regular. 

As in Corollary \ref{cor:thmwithk}, we can extend $p(k;\lambda)$ from the interior of each chamber to its boundary by polynomiality. Clearly, to prove part I. of Theorem \ref{main} for the cases when $ \lambda/k $ is not regular, it is sufficient to show, that these extensions from the chambers containing $ \lambda/k $ in their closure give the same value on $ (k;\lambda) $. It follows from Remark \ref{well-def:lambda/k}, that this is the case, and this completes the proof of part I. of Theorem \ref{main}  (cf. 
Remark \ref{well-def-non-reg}).

\end{section}

\begin{section}{Rank 2, two points}
	\label{sec:2points}

Unfortunately, the argument above does not work for $ r=2 $, because, in this 
case,  $ \theta_1[k]=\theta_{-1}[k] $, the groups $ \Sigma_r^- $ and $ \Sigma_r^+ $ coincide, and thus they do not generate the entire affine Weyl group. The way out is to pass to the 2-punctured case.

\subsection{Wall-crossing}

 We will thus fix two points: $ p,s\in C $, and study the moduli space of rank-2, 
 stable parabolic bundles $W$ 
 with fixed determinant isomorphic to $\mathcal{O}(pd)$, with parabolic 
 structure given by a line $F_1\subset W_p$ with weight $(c, -c)$, and a line
 $G_1\subset W_s$ with weight $(a, -a)$.

Now we need to repeat the analysis of our work so far in this somewhat simpler 
case; some details thus will be omitted.

Set $d=0$; then the space of admissible weights (cf. Figure \ref{square}) is a square
$$\Box = \{(c, a) \, |  \, 1>2c>0, \, 1>2a>0,\, \},$$
which has two adjacent chambers defined by the conditions 
$$c>a \text{ \, and \, } c<a.$$
Denote the corresponding moduli spaces by  $P_0(c>a)$ and $P_0(c<a)$.
\begin{figure}[H]
	\centering
	\begin{tikzpicture}[scale=2]
		\draw [fill=lightgray!30] (0,0) -- (0,1) -- (1,1);
		\draw [fill=gray!30] (0,0) -- (1,0) -- (1,1);
		\draw  [ultra thick, red] (0,0) -- (1,1);
		\draw  [ultra thick] (0,0) -- (0,1);
		\draw  [ultra thick] (0,0) -- (1,0);
		\draw  [ultra thick] (1,0) -- (1,1);
		\draw  [ultra thick] (0,1) -- (1,1);
		\node [above] at (0.35,0.65) {\tiny $P_0(c<a)$};
		\node [below] at (0.65,0.35) {\tiny $P_0(c>a)$};
		\node [below] at (0,0) {\tiny $(0,0)$};
		\node [below] at (1,0) {\tiny $(\frac{1}{2},0)$};
		\node [above] at (0,1) {\tiny $(0,\frac{1}{2})$};
		\node [above] at (1,1) {\tiny $(\frac{1}{2},\frac{1}{2})$};
\end{tikzpicture}
\setlength{\belowcaptionskip}{-8pt}\caption{The space of admissible weights in the case of rank $r=2$, two points.} \label{square}
\end{figure}
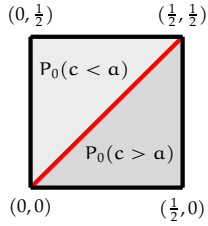
Again, we have universal bundles over $P_0(c>a)\times C$ and $P_0(c<a)\times 
C$, which we will denote by the same symbol 
 $ U $; this bundle is endowed with two flags, $\mathcal{F}_1\subset \mathcal{F}_2=U_p$ and $\mathcal{G}_1\subset \mathcal{G}_2=U_s$.
For $\mu,\lambda\in\mathbb{Z}$, we 
introduce the line bundle 
\begin{equation*}
\begin{split}
\mathcal{L}(k; \lambda, {\mu})= & 
\mathrm{det}({U}_p)^{k(1-g)}\otimes\mathrm{det}(\pi_*(U))^{-k} \\ 
&\otimes(\mathcal{F}_2/\mathcal{F}_{1})^{\lambda}
\otimes(\mathcal{F}_1)^{-\lambda}
\otimes(\mathcal{G}_2/\mathcal{G}_1)^{\mu}\otimes(\mathcal{G}_1)^{-\mu}.
\end{split}
\end{equation*}

We repeat the construction of the master space from Section 5.1, choosing a point $(c^0,c^0)$ on the wall and  two points
$$(c,a)^\pm=(c^0,c^0)\pm \epsilon(1,0)\in\Box, \,\,\, 
\epsilon\in\mathbb{Q}_{>0}$$ from the adjacent chambers. We can identify the 
fixed point set $ Z^0 $ as follows.

\begin{lemma}
The locus $Z^0$ defined in Proposition \ref{BundleN} is  $$Z^0\simeq Jac^0 \simeq \{V=L\oplus{L^{-1}} \, | \, L_p\subset F_1,  L_s^{-1}\subset G_1\}.$$
\end{lemma}
As in \S\ref{S6.1}, 
denote by $\mathcal{J}$ the universal bundle over $Jac^0\times C$ normalized in 
such a way that $c_1(\mathcal{J})_{(0)}=0$ (cf. (\ref{Kunneth})). Define 
$$\eta\in H^2(Jac)\; \text{\,\, by\, \, } \;\left(\sum_ic_1(\mathcal{J})_{(e_i)}\otimes 
e_i\right)^2=-2\eta\otimes\omega;$$ we have then  $\int_{Jac}e^{\eta m}=m^g $ for 
$m\in\mathbb{Z}$.

Let  $\pi: Jac^0\times{C} \to Jac^0$ be the projection and ${N_{Z^0}}$ be the 
equivariant normal bundle to  $Z^0$ in $Z$. Then, as in Lemma \ref{H1ParHom}, Proposition 
\ref{NormalTodd} and Lemma \ref{restriction},   we obtain the identifications:
\begin{itemize} 
\item ${N_{Z^0}} = R_T^1\pi_*(ParHom(\mathcal{J},\mathcal{J}^{-1}))\oplus R_T^1\pi_*(ParHom(\mathcal{J}^{-1},\mathcal{J})),$
where $T\simeq\mathbb{C}^*$-action has weights (-1,1);

\item  $E({N_{Z^0}})^{-1}=(-1)^{g}(2\mathrm{sinh}\left(\frac{u}{2}\right))^{-2g}exp({4\eta});$

\item  $ch_T(\mathcal{L}{(k;\lambda,\mu)}\big{|}_{Z^0})=exp(2k\eta)exp(u(\lambda-\mu)).$
\end{itemize}

Now we define the polynomials:
$$h_>(k;\lambda,\mu)\overset{\mathrm{def}}=\chi(P_0(c>a),\mathcal{L}{(k;\lambda,\mu)}),
\;h_<(k;\lambda,\mu)\overset{\mathrm{def}}=\chi(P_0(c<a),\mathcal{L}{(k;\lambda,\mu)}).$$
and, applying Theorem \ref{ThaddTh}, we obtain the following expression for 
their difference.
\begin{lemma}\label{2pointdiff} The wall crossing term equals
\begin{equation*}
h_>(k;\lambda,\mu)-h_<(k;\lambda,\mu) = \
 (-1)^{g}(2k+4)^g\res_{\substack{u=0}}\frac{exp(u(\lambda-\mu))}{(2sinh\left(\frac{u}{2}\right))^{2g}}\,
  du.
\end{equation*}
\end{lemma}
\subsection{Symmetry}
Denote by $P_{-1}(c>a)$ the image of the moduli space $P_0(c>a)$ under the 
Hecke isomorphism $\mathcal{H}$ (cf. \S\ref{sec:hecke}) at the point $p$ 
and by  $P_{-1}(c<a)$  the image of the moduli space $P_0(c<a)$ under the Hecke 
isomorphism $\mathcal{H}$ at the point $s$. 

We have the following analogue of Lemma \ref{flagbundle}.
\begin{lemma}\label{2pointp1fibr}
Denote by $N_{-1}$ the moduli space of rank-2 degree-($-$1) stable bundles on 
$C$ and by $UN$ any universal bundle over $N_{-1}\times C$. Then the moduli spaces $P_{-1}(c>a)$ and 
$P_{-1}(c<a)$ are isomorphic to the bundle $\mathbb{P}(UN_p)\times\mathbb{P}(UN_s)$ over 
$N_{-1}$.
\end{lemma}
Denote by $\mathcal{T}[p]$ and $\mathcal{T}[s]$ the vertical tangent lines of $\mathbb{P}(UN_p)$ and $\mathbb{P}(UN_s)$, respectively, and by $\L_{-1}$ the pullback of the ample generator of the Picard group of $N_{-1}$ to $\mathbb{P}(UN_p)\times \mathbb{P}(UN_s)$ (cf. Lemma \ref{PicN}). Then a simple calculation shows the following. 

\begin{lemma}
Under the Hecke isomorphism $\mathcal{H}$ at $ p $, the line bundle $\mathcal{L}(2k;\lambda,\mu)$ on $P_0(c>a)$  corresponds to the line bundle $\L_{-1}^{k}\otimes\mathcal{T}[p]^{\lambda-k}\otimes\mathcal{T}[s]^\mu$ on $P_{-1}(c>a)$. \\
Under the Hecke isomorphism $\mathcal{H}$ at the point $s$,  $\mathcal{L}(2k;\lambda,\mu)$ on $P_0(c<a)$  corresponds to  the line bundle $\L_{-1}^{k}\otimes\mathcal{T}[p]^\lambda\otimes\mathcal{T}[s]^{\mu-k}$ on $P_{-1}(c<a)$.
\end{lemma}

As in \S\ref{S7.2}, applying Serre duality for families of curves (cf. Proposition \ref{Serre}) to the line bundles on the two $\mathbb{P}^1\times\mathbb{P}^1$ bundles over $N_{-1}$,  
we obtain that the polynomials $h_>(k;\lambda,\mu)$ and $h_<(k;\lambda,\mu)$ are anti-invariant under the action of the Weyl group $\Sigma_2\times\Sigma_2$ with the center at $\theta_1[k]=(\frac{k+1}{2},\frac{-1}{2})$ and $\theta_2[k] = (\frac{-1}{2}, \frac{k+1}{2})$, correspondingly (cf. Figure \ref{fig:squarecen}).
In other words, we obtain the following 4 identities.
\begin{lemma}\label{2pointsym}
$$h_>(k;\lambda,\mu)=-h_>(k;\lambda,-\mu-1)=-h_>(k;-\lambda+k+1,\mu);$$
$$h_<(k;\lambda,\mu)=-h_<(k;-\lambda-1,\mu)=-h_<(k;\lambda,-\mu+k+1).$$
\end{lemma}
\begin{figure}[H]
	\centering
	\begin{tikzpicture}[scale=2]
		\draw [fill=lightgray!30] (0,0) -- (0,1) -- (1,1);
		\draw [fill=gray!30] (0,0) -- (1,0) -- (1,1);
		\draw  [ultra thick, red] (0,0) -- (1,1);
		\draw  [ultra thick] (0,0) -- (0,1);
		\draw  [ultra thick] (0,0) -- (1,0);
		\draw  [ultra thick] (1,0) -- (1,1);
		\draw  [ultra thick] (0,1) -- (1,1);
		\node [above] at (0.35,0.65) {\tiny $P_0(c<a)$};
		\node [below] at (0.65,0.35) {\tiny $P_0(c>a)$};
		\node [below] at (0,0) {\tiny $(0,0)$};
		\node [below] at (1,0) {\tiny $(2,0)$};
		\node [above] at (0,1) {\tiny $(0,2)$};
		\node [above] at (1,1) {\tiny $(2,2)$};
		\draw  [gray, dashed] (-0.5,-0.25) -- (1.5,-0.25);
	        \draw  [gray, dashed] (-0.25,-0.5) -- (-0.25,1.5);
	        \draw  [gray, dashed] (-0.5,1.25) -- (1.5,1.25);
	        \draw  [gray, dashed] (1.25,-0.5) -- (1.25,1.5);
	        \draw [fill] (-0.25,1.25) circle [radius=0.02];
	         \draw [fill] (1.25,-0.25) circle [radius=0.02];
		\node [above left] at (-0.25,1.25) {\tiny $\theta_2[4]$};
		\node [below right] at (1.25,-0.25) {\tiny $\theta_1[4]$};	        
\end{tikzpicture}
\setlength{\belowcaptionskip}{-8pt}\caption{$k=4$, $r=2$, two points.} \label{fig:squarecen}
\end{figure}
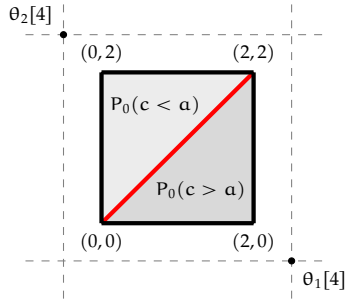
Now, define the polynomials  
$$\widetilde{h}_>(k;\lambda,\mu)=(-1)^{g-1}(2k+4)^g\res_{\substack{u=0}}\frac{exp(u(\lambda+\mu+1))
	-exp(u(\lambda-\mu))}{(2\mathrm{sinh}\left(\frac{u}{2}\right))^{2g}(1-e^{u(k+2)})}du;$$
$$\widetilde{h}_<(k;\lambda,\mu)=(-1)^{g-1}(2k+4)^g\res_{\substack{u=0}}\frac{exp(u(\lambda+\mu+1))
	-exp(u(\lambda-\mu+k+2))}{(2\mathrm{sinh}\left(\frac{u}{2}\right))^{2g}(1-e^{u(k+2)})}du,$$
and from here we can follow the logic of the proof of part I of Theorem 
\ref{main}.

\begin{proposition}\label{2pointthm}
The polynomials introduced above, in fact, coincide:
$$h_>(k;\lambda,\mu)=\widetilde{h}_>(k;\lambda,\mu)\quad\text{and}\quad
h_<(k;\lambda,\mu)=\widetilde{h}_<(k;\lambda,\mu).$$
\end{proposition}
\begin{proof} It is a simple exercise to show that 
$\widetilde{h}_>(k;\lambda,\mu) $ and $\widetilde{h}_<(k;\lambda,\mu) $ satisfy 
the identities appearing in Lemmas \ref{2pointdiff} and \ref{2pointsym}, and 
hence 
the polynomial
	\[  \Theta(k;\lambda,\mu)= h_>(k;\lambda,\mu)-\widetilde{h}_>(k;\lambda,\mu)
	=h_<(k;\lambda,\mu)-\widetilde{h}_<(k;\lambda,\mu)
	 \]
	 satisfies all four $ \Sigma_2 $-symmetries listed in Lemma 
	 \ref{2pointsym}. These groups together generate a double action of the 
	 affine Weyl group $ \widetilde{\Sigma} $ in $ \lambda $ and $ \mu $ 
	 separately, and this implies the vanishing of $ \Theta $.
\end{proof}
As $P_0(c>a)$ is a $\mathbb{P}^1$-bundle over the moduli space of rank-2 
degree-0 stable parabolic bundles $P_0(c,-c)$,  substituting $\mu=0$ in $\widetilde{h}_>$, we obtain the Verlinde formula for rank 2.
\begin{corollary}
$$\chi(P_0(c,-c), 
\mathcal{L}_0{(k;\lambda)})=(-1)^{g-1}(2k+4)^{g}\res_{\substack{u=0}}
\frac{exp(u(\lambda+\frac{1}{2}))}{(2sinh\left(\frac{u}{2}\right))^{2g-1}(1-e^{u(k+2)})}du.$$
\end{corollary}

\end{section}

\begin{section}{The combinatorics of the $ [Q,R]=0 $ }\label{S10}

In this section, we give a proof of the second part of 
Theorem \ref{main}. 
Let $\lambda/k\in\Delta$, and fix a regular element $\textit{\k c}\in\Delta$ in a chamber containing $\lambda/k$  in its closure, and another regular element $\widehat{\textit{\k c}}\in\Delta$ containing  $\lala/\kk$ in its closure.
Our goal is to prove the 
the equality $p_{\textit{\k c}}(k;\lambda)=
p_{\widehat{\textit{\k c}}}(k;\lambda)$, where we define
\begin{equation}\label{p_c}
p_{c}(k;\lambda)=\tilde{N}_{r,k}\sum_{\bb\in\DD}
\iber_{\bb} [ w_\Phi^{1-2g}(x/\kk)] (\lala/\kk- 
[c]_\bb)
\end{equation} 
for a regular
$c\in\Delta$ and diagonal basis $\mathcal{D}$. This is a subtle 
statement, which is a combinatorial-geometric projection of the idea of 
quantization commutes with reduction (or $ [Q,R]=0 $ for short, cf. 
\cite{MSj,SzV}).

If  
$\lambda/k\sim\widehat{\lambda}/\widehat{k}$, i.e. when $\lambda/k$ 
and $\widehat{\lambda}/\widehat{k}$ are regular elements in the same chamber in 
$\Delta$, then $p_\textit{\k c}(k;\lambda)=
p_{\widehat{\textit{\k c}}}(k;\lambda)$ is a tautology. We assume thus 
that this is not the case, and  
denote by $\mathcal{S}(k,\lambda)$ the set of walls 
separating $\textit{\k c}$ and $\widehat{\textit{\k c}}$, or containing  either $\lambda/k$ or $\widehat{\lambda}/\widehat{k}$ or both. Equivalently, the wall $S_{\Pi,l}$ belongs to $\mathcal{S}(k,\lambda)$ if 
$$(\lambda/k)_{\Pi'}\geq l\geq (\widehat{\lambda}/\widehat{k})_{\Pi'} \text{\, or \, } (\lambda/k)_{\Pi'}\leq l\leq (\widehat{\lambda}/\widehat{k})_{\Pi'},$$ where $c_{\Pi^\p} = \sum_{i\in\Pi^\p}c_i $ for an element 
$c=(c_1,...,c_r)\in V^*$. 
Clearly, there is a path in $\Delta$ connecting $\textit{\k c}$ and $\widehat{\textit{\k c}}$, which intersects only walls from $\mathcal{S}(k,\lambda)$ in a generic points. 
Then to prove  
the equality $p_{\textit{\k c}}(k;\lambda)=
p_{\widehat{\textit{\k c}}}(k;\lambda)$, it is enough to show the following, at first sight 
somewhat surprising fact. 
\begin{proposition}\label{wcrgen} Assume $ g\ge1 $, $\lambda/k\in\Delta$, 
$S_{\Pi,l}\in 
\mathcal{S}(k,\lambda)$ and let
$c^\pm\in\Delta$ be two points in two neighboring chambers 
separated by the wall $S_{\Pi,l}$. Then 
\begin{equation}\label{wallcrosseq}
	p_{c^+}(k;\lambda)=p_{c^-}(k;\lambda).
\end{equation}
\end{proposition}
\begin{proof}
The difference of the two sides of \eqref{wallcrosseq} is expressed as a 
residue in \eqref{integralt}.  The integral in \eqref{integralt} is a rational 
expression in the variable $ t $, and our plan is to show by degree count in $ 
t $ and $ t^{-1} $ that 
its residues at zero and at $ \infty $ vanish.
 We define the degree of the quotient of two polynomials $ R=P/Q $ 
of the 
variable $t$ as $ \deg_t(R)=\deg_t(P)-\deg_t(Q) $, and we set $ 
\deg_{t^{-1}}(R)=\deg_t(R(t^{-1})) $. Then, clearly, 
\[ \deg_t(R)<0\Longrightarrow \res_{t=\infty}R\,\frac{dt}{t}=0
\quad\text{and}\quad
\deg_{t^{-1}}(R)<0\Longrightarrow \res_{t=0}R\,\frac{dt}{t}=0.
\]

A convenient expression  for 
\eqref{integralt} will be \eqref{final}, where we change variables via $ 
t=e^{u} $. In what follows, we will  always tacitly assume this substitution, 
and we will write, for example, $ \deg_{t^{\pm1}}(1/(e^u-e^{-u}))=-1 $. We thus 
obtain a 
formula of the form $ \res_{t=0,\infty}  f(t)\, dt/t $, and to show that this 
is zero, it is sufficient to show that $ deg_{t}(f)<0$ and $deg_{t^{-1}}(f)<0 $.

Now we observe that the variable $ u $ occurs only in the first line of 
\eqref{final}, and thus, calculating the degrees in $ t $ and $ t^{-1} $ 
separately, we obtain the following formula: 
\begin{equation}\label{degest}
	deg_{t^{\pm1}}(f)=\pm(k\delta-rl)+(1-2g)
	\mathrm{deg}_{t^{\pm1}}(\tau\cdot w_{u}^\times)+\mathrm{deg}_{t^{\pm1}}
	(\exp(\tau\cdot \rho^\times_{u})). 
\end{equation}
Recall that here $ \delta $ represents the distance of $ \lambda/k $ from the 
wall $ S_{\Pi,l} $, while $w_{u}^\times $ and $ \rho^\times_{u} $, represent 
the parts of the Weyl denominator and the $ \rho $-shift corresponding to roots 
connecting $ \Pi' $ and $ \Pi^{\pp} $, respectively.

We begin the study of this expression with some simple remarks. We recall that 
the permutation $ \tau $ preserves the partition $ \Pi=(\Pi',\Pi^{\pp}) $, and 
thus we have 
$$ 
\mathrm{deg}_{t^{\pm1}}(\tau\cdot w_u^\times)=\mathrm{deg}_{t^{\pm1}}( 
w_u^\times)=\frac{r^\p 
r^{\pp}}{2}.$$ Using, in addition, that $ \rho^\times_{u} $ is linear in $ u $, 
we obtain $$\mathrm{deg}_{t}
	(\exp(\tau\cdot \rho^\times_{u}))=-\mathrm{deg}_{t^{-1}}
	(\exp(\tau\cdot \rho^\times_{u}))=\mathrm{deg}_{t}
	(\exp( \rho^\times_{u})).$$
Combining these equalities, and assuming $ g\ge1 $, we arrive at the following 
conclusion. 
\begin{lemma}\label{wallcrossreduce} The inequality
	\begin{equation}\label{kdeltatoprove}
	 \left|(k\delta-rl) + \mathrm{deg}_{t}
(\exp(\rho^\times_{u}))\right| < \frac{r^\p 
	r^{\pp}}{2} 	
\end{equation}
implies the vanishing of the wall-crossing term: equality \eqref{wallcrosseq}.
\end{lemma}

Before we proceed, we introduce some notation. Denote by $$\mathrm{Inv}(\Pi) =  
\{(i,j) |\, \Pi^{\p}\ni i>j\in\Pi^{\pp} \}$$
the set of "inverted" pairs of elements of the partition $ \Pi $. The number of these pairs 
$ |\mathrm{Inv}(\Pi)| $ coincides 
with the standard notion of length of the shuffle permutation $\phi\in\Sigma_r$ 
introduced in \S\ref{sec:wallcrpara}.

Each pair $ (i,j) $ which is not inverted contributes $ +u/2 $ to $ \rho^\times_{u} $, while each inverted pair contributes $ -u/2 $, and thus we have
\begin{equation}\label{inveq}
  \mathrm{deg}_{t}
	(\exp(\rho^\times_{u}))= \frac{r^\p r^{\pp}}{2}-|\mathrm{Inv}(\Pi)| . 
\end{equation}

Also, recall the notation $c_{\Pi^\p} = \sum_{i\in\Pi^\p}c_i $ for an element 
$c=(c_1,...,c_r)\in V^*$; in particular, we have $(\lambda/k)_{\Pi'}=l+\delta$ and 
\[ \rho_{\Pi'}=\sum_{i\in\Pi'}   \frac{r+1}{2}-i. \]
The following is a simple exercise, whose proof will be omitted:
\begin{equation}\label{myfavourite}
\mathrm{deg}_{t}
	(\exp(\rho^\times_{u}))=\rho_{\Pi^\p}. 
\end{equation}

Now we come to a key point of our argument.
\begin{lemma}\label{myfavouritelemma2} If the intersection 
of the wall $S_{\Pi,l}$ with $\Delta$ is non-empty, then
\begin{equation}\label{rhoineq}
	 -\frac{r^\p r^{\pp}}{2} <lr-\rho_{\Pi^\p}< \frac{r^\p r^{\pp}}{2}. 
\end{equation}
\end{lemma}
\begin{proof}
Pick a point $c=(c_{1},...,c_{r})$ in the intersection $S_{\Pi,l}\cap\Delta$, 
and recall that for any $1\le i< j\le r$, we have   $0<c_{i}-c_{j}<1$, and
\[ \sum_{i\in\Pi'}c_i=-\sum_{i\in\Pi^{\pp}}c_i =l. \]
Then
\[ -|\mathrm{Inv}(\Pi)|<\sum_{(i,j)\in\mathrm{Inv}(\Pi)} (c_i-c_j)
\leq \sum_{i\in\Pi'}\sum_{j\in\Pi^{\pp}} (c_i-c_j), \]
and, similarly,
\[ \sum_{i\in\Pi'}\sum_{j\in\Pi^{\pp}} (c_i-c_j)<r^\p 
r^{\pp}-|\mathrm{Inv}(\Pi)|.
 \]

Now, since
\[  \sum_{i\in\Pi'}\sum_{j\in\Pi^{\pp}} 
(c_i-c_j)=r^{\pp}\sum_{i\in\Pi^{\p}}c_i-r^{\p} \sum_{j\in\Pi^{\pp}} c_j =lr, \]
we can conclude
$$-|\mathrm{Inv}(\Pi)|<lr<r^\p r^{\pp}-|\mathrm{Inv}(\Pi)|.$$
In view of \eqref{inveq} and \eqref{myfavourite}, these inequalities are 
equivalent to \eqref{rhoineq}, and this completes the proof.
\end{proof}

Now we are ready to prove \eqref{wallcrosseq}. The condition 
$S_{\Pi,l}\in\mathcal{S}(k,\lambda)$, i.e. that $S_{\Pi,l}$ separates $ 
\lambda/k$ and $\widehat{\lambda}/\widehat{k}$ or contains $ 
\lambda/k$ or $\widehat{\lambda}/\widehat{k}$, may 
occur in 
two ways.
\begin{itemize}
	\item 
	$(\lambda/k)_{\Pi'}\geq l\geq (\widehat{\lambda}/\widehat{k})_{\Pi'},$ 
	which is equivalent to the two inequalities: $ \delta\geq 0 $ and $ 
	lk+lr\geq \lambda_{\Pi'}+\rho_{\Pi'}$. After canceling $ lk $ and reordering 
	the terms, we can rewrite these as 
	\begin{equation}\label{strict1}
		0\geq k\delta-lr+\rho_{\Pi^\p}\geq \rho_{\Pi^\p}-lr.
		\end{equation}
	Using Lemma \ref{myfavouritelemma2} then we can conclude that 	
			$$0\geq k\delta-lr+\rho_{\Pi^\p}> 	- \frac{r^\p r^{\pp}}{2},$$
	which, in view of the equality \eqref{myfavourite}, implies	the necessary 
	estimate \eqref{kdeltatoprove}.
			
	\item The second case is similar:
	$(\lambda/k)_{\Pi'}\leq l\leq (\widehat{\lambda}/\widehat{k})_{\Pi'}$ is equivalent 
	to $ \delta\leq0 $ and $ lk+lr\leq \lambda_{\Pi'}+\rho_{\Pi'}$. This leads to 
			\begin{equation}\label{strict2}
			0\leq k\delta-lr+\rho_{\Pi^\p}\leq \rho_{\Pi^\p}-lr,
			\end{equation}
which, in turn, implies
			$$0\leq k\delta-lr+\rho_{\Pi^\p}< \frac{r^\p r^{\pp}}{2},$$
and hence \eqref{kdeltatoprove}.
\end{itemize}

This completes the proof of Proposition \ref{wcrgen}: indeed, a simple calculation shows that if $ 
\lambda/k\in\Delta$ then 
$\widehat{\lambda}/\widehat{k}\in\Delta, $ so the conditions of Lemma \ref{myfavouritelemma2} hold. We 
have just shown that this implies \eqref{kdeltatoprove}, and according to Lemma 
\ref{wallcrossreduce}, we can conclude the vanishing of the wall-crossing term
\eqref{wallcrosseq}.
\end{proof}

\begin{remark}\label{well-def:lambda/k} 
Note that if $\lambda/k\in\Delta$ is non-regular, then it belongs to some wall from the set $\mathcal{S}(k,\lambda)$. Hence
proposition \ref{wcrgen} implies that the right-hand side of formula (I.) of Theorem \ref{main} is a well-defined function on the cone over $ \Delta $:
\[ \{(k,\lambda)\in \mathbb{Z}^{>0}\times\Lambda|\,\lambda/k\in\Delta\}. \]

\end{remark}

\end{section}

\vskip 1cm


\begin{thebibliography}{99}



\bibitem[AMW]{AlexeevMW} {\sc Alexeev A., Meinrenken E., Woodward} The Verlinde formulas as fixed point formulas, \textit{J. Symplectic Geom.} \textbf{1} (2001) 1-46.
\bibitem[AB1]{AtiyahBottfixed}{\sc Atiyah M., Bott R.} The Lefschetz fixed point formula for elliptic complexes: II. Applications \textit{Ann. of Math.} \textbf{38} 451-491.
\bibitem[AB2]{AtiyahBott} {\sc Atiyah M., Bott R.} Yang-Mills equations over Riemann surfaces, \textit{Phil. trans. R. Soc. London} \textbf{308} (1982) 523-615.
\bibitem[B]{Bhosle} {\sc Bhosle U.N.} Parabolic vector bundles on curves,  \textit{Arkiv f\"or Math.} \textbf{27} (1989) 15-22.
\bibitem[BL]{BismutLabourie} {\sc Bismut J.-M.,  Labourie F.} Symplectic geometry and the Verlinde formulas.	Surveys in differential geometry: differential geometry inspired by string theory, \textit{Surv. Differ. Geom.}, \textbf{5} (1999) 97-311. 
\bibitem[BH]{BodenH} {\sc Boden H., Hu Y.} Variation of moduli of parabolic bundles, \textit{Matematische Annalen} \textbf{301 (3)} (1995) 539-559.
\bibitem[DH]{DolgachevH}{\sc Dolgachev I.V., Hu Y.} Variation of geometric invariant theory quotients, \textit{Publications Math\'ematiques de l'IH\'ES} \textbf{87} (1998) 5-56.
\bibitem[G]{Grothendieck}{\sc Grothendieck A.} Technique de descente et th\'eor\`emes d'existence en g\'eom\'etrie alg\'ebraique, IV: Les sch\'emas de Hilbert, \textit{S\'em. Bourbaki}
 \textbf{221} (1960-61); reprinted in \textit{Fondements de la g\'eom\'etrie alg\'ebrique} (1962).
\bibitem[H]{Hartshorne}{\sc Hartshorne R.} Algebraic geometry, \textit{Springer, New York} (1977). 
\bibitem[J]{Jeffrey} {\sc Jeffrey L.C.} The Verlinde formula for parabolic bundle, \textit{J. London Math. Soc.} \textbf{63 (3)} (2001) 754-768.
\bibitem[JK]{JeffreyK} {\sc Jeffrey L.C., Kirwan F.C.} Intersection theory on moduli spaces of holomorphic bundles of arbitrary rank on a Riemann surface, \textit{Ann. of Math.} \textbf{148} (1998) 109-196.
\bibitem[LM]{Loizides}{\sc Loizides Y., Meinrenken E.} The decomposition formula for Verlinde sums \textit{Ann. l'Inst. Fourier} to appear.
\bibitem[MS]{MehtaSeshadri} {\sc Mehta V.B., Seshadri C.S.} Moduli of vector bundles on curves with parabolic structures, \textit{Math. Ann.} \textbf{248} (1980) 205-239.
\bibitem[M]{Meinrenken} {\sc Meinrenken E.} Twisted K-Homology and Group-Valued Moment Maps, \textit{International Mathematics Research Notices} \textbf{20} (2012) 4563-4618.
\bibitem[MS]{MSj} {\sc Meinrenken E., Sjamaar R.} Singular reduction and quantization, \textit{Topology} \textbf{38} (1999)  699-762.
\bibitem[MF]{MumFog} {\sc Mumford D., Fogarty J.} Geometric invariant theory, 2nd edition, \textit{Springer-Verlag, Berlin} (1982). 
\bibitem[NR]{NRHecke} {\sc Narasimhan M.S., Ramanan S.} Geometry of Hecke cycles-I, \textit{ In C.P.Ramanujam - a Tribute} \textit{Springer-Verlag, Berlin} (1978) 291-345.
\bibitem[Se]{Seshadri} {\sc Seshadri S.C.} Moduli of vector bundles on curves with parabolic structures, \textit{Bull. Math. Soc.} \textbf{83(1)} (1977) 124-126.
\bibitem[S]{Sorger} {\sc Sorger C.} La formule de Verlinde, \textit{S\'eminaire Bourbaki : volume 1994/95, expos\'es 790-804, Ast\'erisque} \textbf{237} (1996) \textit{Expos\'e no. 794}. 
\bibitem[Sz1]{Szimrn} {\sc Szenes A.}  Iterated residues and multiple Bernoulli polynomials, \textit{IMRN Internat. Math. Res. Notices} \textbf{18} (1998) 937-956.
\bibitem[Sz2]{Szduke} {\sc Szenes A.} Residue theorem for rational trigonometric sums and Verlinde's formula, \textit{Duke Math. J.} \textbf{118 (2)} (2003) 189-227.
\bibitem[SzV]{SzV} {\sc Szenes A., Vergne M.}  $[Q,R]=0$ and Kostant partition functions. \textit{Enseign. Math.} \textbf{63} (2017) 471-516.
\bibitem[TW]{TelemanW} {\sc Teleman C., Woodward C.} The index formula for the moduli of $G$-bundles on a curve, \textit{Ann. of Math.} (2009) 495-527.
\bibitem[Th1]{ThaddeusFlip} {\sc Thaddeus M.} Geometric invariant theory and flips, \textit{J.Amer. Math. Soc.} \textbf{9} (1996) 691-723.
\bibitem[Th2]{ThaddeusVer} {\sc Thaddeus M.} Stable pairs, linear systems and the Verlinde formula, \textit{Inventiones mathematicae} \textbf{117.2} (1994) 317-354.
\bibitem[V]{Vergne} {\sc Vergne, M.}  Multiplicities formula for geometric quantization, I, II., \textit{Duke Math. J.} \textbf{82 (1)} (1996), 143-179, 181-194. 
\bibitem[Ver]{Ver} {\sc Verlinde E.}, Fusion rules and modular transformations in 2d conformal field theory, \textit{Nucl. Phys. B} \textbf{300} (1988) 360-376.
\bibitem[W]{Wittenrevisited} {\sc Witten E.}, Two Dimensional Gauge Theories Revisited, \textit{J.Geom.Phys} \textbf{9}  (1992) 303-368.
\bibitem[Z]{Zagier} {\sc Zagier D.} On the cohomology of moduli spaces of rank two vector bundles over curves, \textit{} \textbf{129} (1995) 15-22.
\end{thebibliography}
\end{document}